\documentclass[fms,times]{cuparticle}

\usepackage{latexsym,amssymb}

\usepackage{amsmath, amsthm, hyperref, color}
\usepackage{graphicx}
\usepackage[all]{xypic}
\usepackage{verbatim}
\usepackage{tikz}
\usepackage{shuffle}
\DeclareFontFamily{U}{shuffle}{}
\DeclareFontShape{U}{shuffle}{m}{n}{ <-8>shuffle7 <8->shuffle10}{}
\usepackage{mathtools}

\newtheorem{theorem}{Theorem}
\numberwithin{theorem}{section}
\newtheorem{proposition}[theorem]{Proposition}
\newtheorem{lemma}[theorem]{Lemma}
\newtheorem{corollary}[theorem]{Corollary}

\newtheorem{remark}[theorem]{Remark}
\newtheorem{example}[theorem]{Example}
\newtheorem{conjecture}[theorem]{Conjecture}

\newcommand{\PP}{\mathbb{P}}
\newcommand{\RR}{\mathbb{R}}
\newcommand{\QQ}{\mathbb{Q}}
\newcommand{\CC}{\mathbb{C}}
\newcommand{\ZZ}{\mathbb{Z}}
\newcommand{\KK}{\mathbb{K}}

\DeclarePairedDelimiter\abs{\lvert}{\rvert}

\newcommand{\Gim}{ \mathcal{G}^{\rm im}}

\newcommand{\Uim}{ \mathcal{U}^{\rm im}}

\newcommand{\Lim}{ \mathcal{L}^{\rm im}}

\newcommand{\Pim}{ \mathcal{P}^{\rm im}}

\volume{}
\doi{}

\begin{document}

\authorheadline{Carlos Am\'endola, Peter K. Friz and Bernd Sturmfels}
\runningtitle{Varieties of Signature Tensors}
 
\begin{frontmatter}

\title{Varieties of Signature Tensors}
\author[1]{Carlos Am\'endola}
 
\address[1]{Technische Universit\"at  M\"unchen
  \ead{carlos.amendola@tum.de}}

\author[2]{Peter Friz}
 
\address[2]{Technische Universit\"at Berlin and WIAS Berlin
  \ead{friz@math.tu-berlin.de}}
  
  \author[3]{Bernd Sturmfels}
 
\address[2]{MPI for Mathematics in the Sciences, Leipzig and UC Berkeley
  \ead{bernd@mis.mpg.de}}

\received{23 April 2018}
 
\begin{abstract}
 The signature of a parametric curve is a sequence of tensors
 whose entries are iterated  integrals.  This construction 
 is central to the theory of rough paths in  stochastic analysis.  It is
 here examined through the lens of algebraic geometry.
We introduce varieties of signature tensors for both deterministic paths
and random paths. 
For the former, we focus on piecewise linear paths, on polynomial paths,
and on varieties derived from free nilpotent Lie groups.
For the latter, we focus on Brownian motion and its mixtures.
\end{abstract}
 
\end{frontmatter}    

\section{Introduction}

A {\em path} is a continuous map $X : [0,1] \rightarrow \RR^d$. Unless otherwise stated, we assume that the coordinate functions $X_1,X_2,\ldots,X_d$ are
(piecewise) continuously differentiable functions. In particular, their differentials
$\,{\rm d}X_i (t) = X_i' (t) {\rm d} t \,$ 
obey the usual rules of
calculus. The  {\em $k$th signature tensor} of $X$  has order $k$ 
and format $d {\times} d {\times} \cdots {\times} d $. This tensor is denoted by $\sigma^{(k)}(X)$. Its $d^k$ entries $\sigma_{i_1 i_2 \cdots i_k}$  
are real numbers. These are computed as 
iterated integrals:
\begin{equation}
\label{eq:iteratedint}
 \sigma_{i_1 i_2 \cdots i_k} \,= \,
\int_{0}^1 \int_{0}^{t_k} \cdots \int_{0}^{t_3} \int_{0}^{t_2} 
{\rm d} X_{i_1} (t_1)  \,{\rm d} X_{i_2}(t_2) \, \cdots \,\,{\rm d} X_{i_{k-1}}(t_{k-1})\, {\rm d} X_{i_k}(t_k).
\end{equation}
We also write ${\rm d} X (t) = ({\rm d}X_1 (t), {\rm d} X_2(t), \ldots, {\rm d}X_d(t))$ for the differential of the path $X$. 
The signature tensor can be written as an integral
 over tensor products of such differentials:
\begin{equation}
\label{eq:tensor}
\sigma^{(k)}(X) \,\,\, = \,\,\,
\int_\Delta {\rm d} X(t_1)\otimes {\rm d} X(t_2)   \otimes \cdots \otimes {\rm d} X (t_k).
\end{equation}
This integral is over the simplex
$ \Delta =  \bigl\{ (t_1,\ldots,t_k) \in \RR^k\,: 0 \leq t_1 \leq \cdots \leq t_k \leq 1 \bigr\}$.

Signature tensors were introduced in the 1950s by Kuo Tsai Chen \cite{Chen54, Chen57, Chen58}.
They are central to the theory of rough paths \cite{Lyo98}, a revolutionary view on stochastic analysis, with important contributions by Terry Lyons, Martin Hairer and many others (cf.~\cite{FH, FV, LCL07, LQ02}).

  The {\em step-$n$ signature} of a path $X$ is the sequence of its first $n$ signature tensors:
\begin{equation}
\label{eq:tensorsupton}
 \sigma^{\leq n}(X) \, = \, \bigl( \,1, \,\sigma^{(1)}(X), \, \sigma^{(2)}(X), \,\sigma^{(3)}(X), \,\ldots ,\sigma^{(n)}(X)\, \bigr). 
\end{equation} 
Here $\sigma^{(1)}(X) =  X(1) - X(0) $ is a vector, $\sigma^{(2)}(X)$ is a matrix, etc.
As $n$ goes to infinity, we get a tensor series 
$\sigma(X)$, referred as  {\em the signature} of $X$.
Chen \cite{Chen58} proved that, up to a natural equivalence relation, each path $X$ 
is determined by its signature~$\sigma(X)$.
Hambly and Lyons \cite{HL} refined Chen's theory
and obtained quantitative results for paths of bounded variation. This was extended in \cite{Horatio} to arbitrary geometric rough paths. 
Algorithms for recovering paths $X$ from their signature $\sigma(X)$ are 
of considerable interest for applications \cite{CK, LyoICM, LS, LX1, LX2}.

This paper concerns paths $X$ that depend
on finitely many parameters. We focus on two families:
polynomial paths and piecewise linear paths. In the first case,
each coordinate $X_i$ is a polynomial in $t$ of degree $m$.
In the second case, $X_i$ is piecewise linear with $m$ pieces.
Each family has $dm$ parameters. The integral (\ref{eq:iteratedint})
is a polynomial of degree $k$ in these parameters.
The signature tensor
$\sigma^{(k)}(X)$ is a polynomial function of  the $dm$ parameters.
These tensors trace out an algebraic variety in 
the space of  all tensors. Our aim is to study such varieties.

Stochastic analysis is the natural home for paths and their signatures.
The present article is different. It examines paths and their signatures through the lens of
algebraic geometry \cite{CLO}, commutative algebra \cite{Eis} and
representation theory \cite{Lan}. Our inspiration is drawn from
 algebraic statistics~\cite{DSS} and  tensor decomposition \cite{KB}.
While the tensors we defined have their entries in~$\RR$, the associated 
 varieties comprise tensors  over the complex numbers $\CC$.
 We work with homogeneous polynomials, so they are
  complex projective varieties.

\smallskip
This paper is organized as follows.
In Section \ref{sec:fromto} we fix notation, 
we define our objects of study, and we compute first examples.
 We explain the passage from the real tensors $\sigma(X)$,
where $X$ ranges over paths of interest, to 
homogeneous ideals and complex projective varieties.
The resulting distinction between
{\em image} and  {\em variety}
is ubiquitious in applied algebraic geometry.

 Section \ref{sec:vsm} concerns signature matrices. The case  $k=2$
is surprisingly rich and interesting.
Theorem \ref{thm:realmatrixstuff} shows that
the signature images for polynomial paths of degree $m$ 
agree with those for piecewise linear paths with $m$ segments.
This result does not
generalize to higher order tensors. Theorem \ref{thm:matricesmain}  characterizes
the signature ideals in terms of rank constraints.

In Section \ref{sec:universal} we investigate the 
algebraic underpinnings of the theory of rough paths \cite[\S 7]{FV}.
We study the free Lie algebra and the associated nilpotent Lie groups \cite{Reu}.
Theorem \ref{thm:lyndonmain} characterizes their vanishing ideals 
by identifying explicit Gr\"obner bases.
The order $k$ component of the nilpotent Lie group defines
a variety that contains the signature tensors of all deterministic paths.
It serves as the ambient space for the varieties
 in the next section. 

Section \ref{sec:pwlp} studies signature varieties for
piecewise linear paths and for polynomial paths.
Chen's Theorem \cite{Chen54} gives a parametric representation for the former,
via tensor products of exponentials. 
The two cases are similar, but different (Theorem \ref{thm:different}).
Both families interpolate between   the Veronese variety,
which encodes line segments, and the universal variety (Theorem \ref{thm:chains}).
Numerous instances are computed explicitly. As an application,
we answer a question due to  Lyons and Xu \cite{LX2} concerning axis-parallel paths.
We conclude the section by introducing the {\em rough Veronese},
a variety comprising signatures of rough paths.

In Section \ref{sec:DimIde} we examine the dimensions and identifiability of
our signature varieties. Theorem \ref{thm:EPidentifiability} shows that the universal variety $\mathcal{U}_{d,k}$ 
is identifiable. We conjecture the same
for signature varieties of piecewise linear and polynomial paths
and we offer supporting evidence.

In Section \ref{sec:expected} we shift gears, by turning to random paths  and their
expected signatures. 
We focus on Brownian motion
and mixtures of Brownian motion. This leads to a
non-abelian refinement of the moment varieties in \cite{AFS, ARS}.
The varieties of expected signatures are
computed in several cases.
These are not contained in the universal variety
from Section \ref{sec:universal}.
However, all expected signatures are contained in the
convex hull of the universal variety.

The main contribution of this paper is a new
bridge between applied algebraic geometry and stochastic analysis, where
rough paths are encoded in signature tensors. Varieties of
such tensors offer a concise representation
of geometric data seen in numerous applications.

\section{From Integrals to Projective Varieties} \label{sec:fromto}

In this section we develop the foundations for signature varieties of paths.

\subsection{Computing Iterated Integrals}

A first observation is that the integrals (\ref{eq:iteratedint})
do not change if the path $X$
undergoes translation. We will therefore always assume 
$X(0) = 0$. For such a path $X$ in $\RR^d$, 
consider the entries of the signature matrix 
$\,\sigma^{(2)}(X) = (\sigma_{ij})$. They can be expressed as one-dimensional integrals:
$$  \sigma_{ij} \,\,= \,\, \int_{0}^1 \int_{0}^{t}  
{\rm d} X_{i} (s)  \,{\rm d} X_{j}(t) \,= \,
\int_{0}^1 X_i(t)  X_j'(t) {\rm d}t . $$
The symmetrization of this $d \times d$ matrix
is the matrix $X(1) \cdot X(1)^T $. This symmetric matrix has rank $\leq 1$.
Conversely, every $d \times d$ matrix $S$ whose symmetrization
$S + S^T$ has rank $\leq 1$ and is positive semidefinite arises from some path $X$ in $\RR^d$.
The set of  such matrices is the {\em signature image}
\begin{equation} 
     \Uim_{d,2} \,\,= \,\, \bigl\{\, \tfrac{1}{2} v \otimes v + Q \,\,:\, v \in \RR^d, \,\,
     Q \in \RR^{d \times d} \text{  skew-symmetric }  \,\bigr\} \,\, \subset \, \,\RR^{d \times d}  \ .
      \label{preP22}
\end{equation}
This converse is not obvious. It follows, for instance, from  Theorem~\ref{thm:realmatrixstuff}
below.

We are interested in polynomial equations that hold on the signature image.
Consider the case $d=2$. The symmetric part of our $2 \times 2$ matrix has
rank $1$, so its determinant vanishes. This means that the following quadratic
equation holds on the signature image:
\begin{equation}
 (\sigma_{12}  + \sigma_{21})^2 -  4 \sigma_{11} \sigma_{22} \,\,\, = \,\,\, 0.    \label{eq:favoritequadric}
 \end{equation}
 Let $\mathcal{U}_{2,2}^\RR$ denote the set of solutions $S = (\sigma_{ij})$  to this equation
 in the space of matrices  $\RR^{2 \times 2}$.
This is the smallest variety containing the signature image. 
The latter is the {\em semialgebraic~set}
$$ \mathcal{U}_{2,2}^{\rm im}\,\, = \,\, \bigl\{ \, S \in \mathcal{U}_{2,2}^\RR \,:\, \sigma_{11} \geq 0
\,\, {\rm and} \,\, \sigma_{22} \geq 0 \,\bigr\}.$$
For an algebraic geometer, it is natural to view the
$\sigma_{ij}$ as homogeneous coordinates for the projective space $\PP^3$.
In that ambient space,  the quadratic equation (\ref{eq:favoritequadric}) defines 
the following surface, which will be our {\em signature variety} in this case:
 $$ \mathcal{U}_{2,2} \,\,= \,\, \bigl\{\, [ \,\sigma_{11} : \sigma_{12} : \sigma_{21} : \sigma_{22}\, ]
  \in \PP^3 \text{ such that }  (\ref{eq:favoritequadric}) \text { holds} \,\bigr\} .
 $$
See the next subsection for
 the relevant philosophy.
 
We next present explicit formulas for the entries of the signature tensors for
two special paths in $\RR^d$. It is instructive to check the computation of
these integrals by hand for $d\leq 3$ and $k \leq 3$.

\begin{example}[The Canonical Axis Path] \label{ex:cap} \rm
Let $C_{\rm axis}$ be the path 
from $(0,0,\ldots,0)$ to $(1,1,\dots,1)$  given by $d$ linear
steps in the unit directions $e_1,e_2,\ldots,e_d$ (in that order).
One finds  its signature tensor $\sigma^{(k)}(C_{\rm axis})$
either directly by integration, or via
Chen's Formula~(\ref{eq:PL1}) for piecewise linear paths.
The entry $\sigma_{i_1 i_2 \cdots i_k}$ 
is zero unless  $i_1 \leq i_2 \leq \cdots \leq i_k$.
In that case, it equals $1 / k!$ times the
number of distinct permutations of the string $i_1 i_2 \cdots i_k$.
For example, if $k=4$ then 
$ \sigma_{1111} = \frac{1}{24},\, \sigma_{1112} = \frac{1}{6},\, \sigma_{1122} = 
\frac{1}{4},\, \sigma_{1123} = \frac{1}{2},\,
 \sigma_{1234} = 1$ and $\sigma_{1243} = 0$. 
  \hfill $\diamondsuit $
\end{example}

We next specify a blueprint for paths
whose coordinates are polynomials.

\begin{example}[The Canonical Monomial Path] \label{ex:cmp} \rm
Let $C_{\rm mono}$ denote the monomial path $t \mapsto (t,t^2,t^3,\ldots,t^d)$. It also
travels from $(0,0,\ldots,0)$ to $(1,1,\ldots,1)$, but now along the rational normal curve. The entries
$\sigma_{i_1 i_2 \cdots i_k}$ of the signature tensor  $\sigma^{(k)}(C_{\rm mono})$ are given 
in (\ref{eq:iteratedint}), by integrating
$i_1 i_2 \cdots i_k $ times $ t_1^{i_1-1} t_2^{i_2-1} \cdots t_k^{i_k-1}$
 over the simplex $\Delta$. This computation reveals, for $1 \leq i_1,i_2, \ldots, i_k \leq d$,
$$  \quad \sigma_{i_1 i_2 \cdots i_k} \,\, = \,\,\,
\frac{i_1}{i_1} \cdot \frac{i_2}{i_1+i_2} \cdot \frac{i_3}{i_1+i_2+i_3} 
\,\cdots \, \frac{i_k}{i_1 + i_2 + \cdots + i_k}  .
\qquad 
$$
\end{example}

The general linear group $G = {\rm GL}(d,\RR)$ acts on
paths in $\RR^d$. It also acts on signatures, by the
$k$th tensor power representation. This action extends 
to non-invertible $d \times d$ matrices.
The map $X \mapsto \sigma^{(k)}(X)$
is equivariant for this action.  This allows us to derive
formulas for the following two special classes of paths.

Let ${\bf X} = (x_{ij})$ be a $d \times d$ matrix. Its $j$th column represents
either a linear step or a vector of coefficients associated to the monomial $t^j$.
 Thus, ${\bf X} C_{\rm axis}$ is a piecewise linear path
with $d$ pieces, and ${\bf X} C_{\rm mono}$ is
a polynomial path of degree $d$.  Conversely, all such paths
arise from some choice of  matrix ${\bf X}$.
The signature tensor $\sigma^{(k)}( {\bf X} C_\bullet) $
is the image of $\sigma^{(k)}(C_\bullet)$ under action by ${\bf X}$.
For $k=2$, this
action is multiplying the matrix $\sigma^{(2)} (C_\bullet)$ 
by ${\bf X}$ on the left and by ${\bf X}^T$  on the right.

\begin{example}[$d=3,k=2$] \label{ex:d3k2} \rm
The signature matrices of the  canonical paths are
\begin{equation}
\label{ex:axismono}
\sigma^{(2)}( C_{\rm axis}) \,=\, \begin{pmatrix}
\,\, \frac{1}{2} & 1 & 1 \, \medskip \\
\,\, 0 & \frac{1}{2} & 1 \,\medskip \\
\,\, 0 & 0 & \frac{1}{2}\, \,\end{pmatrix}
\quad \hbox{and} \quad
\sigma^{(2)} (C_{\rm mono}) \,=\, \begin{pmatrix} 
\,\,\frac{1}{2} & \frac{2}{3} & \frac{3}{4} \,\medskip \\
\,\,\frac{1}{3} & \frac{2}{4} & \frac{3}{5} \,\medskip \\
\,\,\frac{1}{4} & \frac{2}{5} & \frac{3}{6} \,\,
\end{pmatrix}.
\end{equation}
The symmetric part of each signature matrix is the same constant rank $1$ matrix:
$$ \sigma^{(2)}(C_{\rm axis}) + \sigma^{(2)}(C_{\rm axis})^T\,\,=\,\,
\sigma^{(2)}(C_{\rm mono}) + \sigma^{(2)}(C_{\rm mono})^T \,\, = \,\,
 \begin{pmatrix}
 1 & 1 & 1 \\
 1 & 1 & 1 \\
1 & 1 & 1 
\end{pmatrix}.
$$
We encode cubic paths $ X = {\bf X} C_{\rm mono}$
and $3$-step paths $X = {\bf X} C_{\rm axis}$
by $3 \times 3$ matrices
$$ {\bf X} \, = \,
\begin{pmatrix} x_{11} & x_{12} & x_{13} \\
x_{21} & x_{22} & x_{23} \\
x_{31} & x_{32} & x_{33}
\end{pmatrix} . $$
In each case, the map  $\,X \mapsto \sigma^{(2)}(X) \,$ from paths to
signature matrices is encoded by the quadratic map
$\,{\bf X} \mapsto {\bf X}  \sigma^{(2)}(C_\bullet) {\bf X^T}$. This takes
the space of $3 \times 3 $ matrices
to itself. The image consists of certain
$3 \times 3$ matrices $S$ whose symmetric part
$\frac{1}{2}(S+S^T)$ has rank $1$. \hfill $\diamondsuit $
\end{example}

\begin{example}[$d=2,k=3$] \label{ex:d2k3} \rm
Consider a general quadratic path in the plane $\RR^2$:
\begin{equation} \label{quadratic} 
       X(t) \,\,=\,\, (x_{11} t + x_{12} t^2, 
x_{21} t + x_{22} t^2)^T \,\,=  \,\,
\begin{pmatrix} x_{11} & x_{12}  \\
x_{21} & x_{22} 
\end{pmatrix} 
\cdot
\begin{pmatrix} t  \\
t^2
\end{pmatrix} .
\end{equation}
 Its third signature $(k=3)$ is a $2 \times 2 \times 2$ tensor. The eight entries 
of $\sigma^{(3)}(X)$ are
\begin{equation}
\label{eq:entriesare} \begin{matrix}
\sigma_{111}  & = &     \frac{1}{6} (x_{11} + x_{12})^3 &  & \smallskip \\
\sigma_{112}  & = &    \frac{1}{6} (x_{11} {+} x_{12})^2 (x_{21}{+}x_{22}) & + &
 \frac{1}{60} (5x_{11}+4x_{12})(x_{11} x_{22} - x_{21} x_{12})
\smallskip \\
\sigma_{121}  & = &  \frac{1}{6} (x_{11} {+} x_{12})^2 (x_{21}{+}x_{22}) & + &
\quad  \frac{1}{60} (2 x_{12})(x_{11} x_{22} - x_{21} x_{12})
 \smallskip \\
\sigma_{211}  & = &   \frac{1}{6} (x_{11} {+} x_{12})^2 (x_{21}{+}x_{22}) & - &
 \frac{1}{60} (5x_{11}+6x_{12})(x_{11} x_{22} - x_{21} x_{12})
 \smallskip \\
\sigma_{122}  & = &    \frac{1}{6} (x_{11} {+} x_{12}) (x_{21}{+}x_{22})^2 & + &
\frac{1}{60} (5x_{21}+6x_{22})(x_{11} x_{22} - x_{21} x_{12})
\smallskip \\
\sigma_{212}  & = &    \frac{1}{6} (x_{11} {+} x_{12}) (x_{21}{+}x_{22})^2 & - &
\quad  \frac{1}{60} (2 x_{22})(x_{11} x_{22} - x_{21} x_{12})
\smallskip \\
\sigma_{221}  & = &  \frac{1}{6} (x_{11} {+} x_{12}) (x_{21}{+}x_{22})^2 & - &
 \frac{1}{60} (5x_{21}+4x_{22})(x_{11} x_{22} - x_{21} x_{12})
\smallskip \\
\sigma_{222}  & = &     \frac{1}{6} (x_{21} + x_{22})^3 & 
\end{matrix}
\end{equation}
This is the transformation of the special tensor $\sigma^{(3)}(C_{\rm mono})$ in Example~\ref{ex:cmp}
under the action of the $2 \times 2$ matrix ${\bf X} = (x_{ij})$.
The symmetrization of $\sigma^{(3)}(X)$ is a tensor of rank one, because  
\begin{equation}
\label{eq:itfactors}
\begin{small}
 \begin{matrix} 
&  \sigma_{112} + \sigma_{121} + \sigma_{211} &= &
 \sigma_{11\, \shuffle \, 2} = \sigma_{11}  \sigma_{2} = \frac{1}{2}  \sigma_{1}^2  \sigma_{2} &
= & \frac{1}{2}  (x_{11}{+}x_{12})^2 (x_{21}{+}x_{22}) 
 \smallskip \\
 & \sigma_{122} + \sigma_{212} + \sigma_{221} & = & 
 \sigma_{1\, \shuffle \, 22} = \sigma_{1}  \sigma_{22} = \frac{1}{2}  \sigma_{1} \sigma_{2}^2 &
 =& \frac{1}{2}  (x_{11}{+}x_{12}) (x_{21}{+}x_{22})^2.
 \end{matrix}
 \end{small}
\end{equation}
The {\it shuffle relations} in the middle will be defined and play an important role in Section~\ref{sec:universal}.
 Of course, the equality of left- and right-hand sides in (\ref{eq:itfactors}) can be checked directly from (\ref{eq:entriesare}).
  
  We wish to find relations among the eight entries of $\sigma^{(3)}(X)$ that hold for all quadratic paths. Hence, we want to eliminate the parameters $x_{ij}$ from (\ref{eq:entriesare}).
This task, known as {\em implicitization}, arises frequently in computer algebra.
The standard approach using {\em Gr\"obner bases} is explained in the undergraduate textbook by
Cox, Little and O'Shea \cite[\S 3.3]{CLO}. The output of the implicitization is the
prime ideal $P_{2,3,2}$ of polynomial relations among the $\sigma_{ijk}$.

The ideal $P_{2,3,2}$  is  generated by nine quadratic relations.
Three of them are
$$
\begin{small}
\begin{matrix}  {\rm degree} & & \hbox{ideal generator} \smallskip \\
(4,2) &  &  3  \sigma_{121}^2- \sigma_{111}  \sigma_{122}
-10  \sigma_{112}  \sigma_{211} + 2  \sigma_{121}  \sigma_{211}
     +2  \sigma_{211}^2+11  \sigma_{111}  \sigma_{212}
     -7  \sigma_{111}  \sigma_{221} \\
(3,3) & &
     10  \sigma_{122}  \sigma_{211} - 4  \sigma_{112}  \sigma_{212} -
     7  \sigma_{121}  \sigma_{212}
     +2  \sigma_{112}  \sigma_{221}-4  \sigma_{121}  \sigma_{221}
     +3  \sigma_{111}  \sigma_{222} \\
(2,4) & &  3  \sigma_{212}^2-10  \sigma_{122}  \sigma_{221}
+2  \sigma_{212}  \sigma_{221} + 2  \sigma_{221}^2- \sigma_{112}  \sigma_{222}
+11  \sigma_{121}  \sigma_{222} -7  \sigma_{211}  \sigma_{222} \\
\end{matrix}
\end{small} 
$$     
These three constraints on $2 \times 2 \times 2$ signature tensors  are specific to  quadratic paths in $\RR^2$.
The other six generators are valid for \underbar{all} paths in $\RR^2$. Those are shown in Example \ref{ex:twothree}.
All nine quadrics can be found with any computer algebra system 
that offers Gr\"obner bases. The computations with ideals reported in this 
paper were carried out with
the software {\tt Macaulay2} \cite{M2}.  \hfill $\diamondsuit $
\end{example}

\subsection{The Many Lives of a Variety} \label{sec:manylives}

We now step  back to contemplate the title of this article.
{\em Signature tensors} were explained in the previous subsection.
Our next goal is to explain what is meant by their {\em varieties}.
This is important because our presentation rests on conventions from algebraic geometry 
that may be unfamiliar to readers from stochastic analysis.
The key point is that the geometric objects we study are
complex projective varieties along with their 
homogeneous ideals.

Consider an arbitrary map $\sigma: \RR^p \rightarrow \RR^q$ whose
coordinates $\sigma_i$ are homogeneous polynomials in $p$ variables
of degree $k$.
We are interested in the image of $\sigma$.
Algebraically, this is represented by an ideal $F$ in the polynomial ring in $q$
variables, namely the kernel of the associated ring homomorphism
from that polynomial ring  to the polynomial~ring in $p$ variables.
The ideal $F$ is {\em homogeneous} and  {\em prime}. It is generated by homogeneous
polynomials. Being prime means that the quotient modulo $F$ is an integral domain
\cite[\S 5.1, Prop.~4]{CLO}.

The image of $\sigma$ is a {\em semialgebraic set} in $\RR^q$,
denoted by $\,\mathcal{F}^{\rm im} = {\rm image}(\sigma)$.
This set is often very complicated. One works with outer
approximations by algebraic varieties, namely the real and complex zeros of the ideal $F$. 
The {\em complex variety} $\mathcal{F}^\CC = \mathcal{V}^\CC(F) \subset \CC^q$
contains the {\em real variety} $\mathcal{F}^\RR = \mathcal{V}^\RR(F) \subset \RR^q$.
Both of the following two inclusions are usually strict:
\begin{equation}
\label{eq:Finclusions} \mathcal{F}^{\rm im} \subset \mathcal{F}^\RR \subset \mathcal{F}^\CC . \end{equation}

\begin{example} \label{ex:zweifunf} \rm
Let $p=2, q= 4$ and $\sigma : u \mapsto u \otimes u$ the map
that takes a vector $u = (u_1,u_2)$ to the associated rank one matrix
$u \otimes u = \begin{small} \begin{pmatrix} u_1^2 & u_1 u_2 \\ u_1 u_2 & u_2^2 \end{pmatrix} \end{small}$.
The ideal $F$ is generated by the linear form $\sigma_{12} - \sigma_{21}$
and the quadratic form $\sigma_{11} \sigma_{22} - \sigma_{12} \sigma_{21}$.
The variety  $\mathcal{F}^\RR$ is a real surface in $\RR^4$, and
$\mathcal{F}^\CC$ is a complex surface in $\CC^4$.
These are the symmetric $2 \times 2$ matrices of rank at most $1$.
The image $\mathcal{F}^{\rm im}$ is the subset of matrices in $\mathcal{F}^\RR$
that are positive semidefinite.
\hfill $\diamondsuit$
\end{example}

From a geometric point of view, it is natural to work modulo
scaling. This means we work in a projective space:
two non-zero vectors are identified when they are parallel.
Our degree $k$ map $\sigma$ induces
a rational map of complex projective spaces $ \sigma' : \PP_\CC^{p-1} \dashrightarrow \PP_\CC^{q-1}$,
and this restricts to a rational map
of real projective spaces $\sigma'' : \PP_\RR^{p-1} \dashrightarrow \PP_\RR^{q-1}$.
Being {\em rational} means that $\sigma'$ may not be everywhere defined.
 If $u \in   \CC^p \backslash \{0\}$ satisfies $\sigma(u) = 0$ then
 $\sigma'$ is not defined at  the corresponding point $[u] \in \PP_\CC^{p-1}$.
 Such points are known as {\em base points}.

 Every subset $\mathcal{U}$ of $\RR^q$ or $\CC^q$ gives rise to a subset
 in $\PP_\RR^{q-1}$ or $\PP_\CC^{q-1}$,  namely the set  $[\mathcal{U}]$ of lines
 that are spanned by vectors in $\mathcal{U}$. We write this in symbols as
 $$ [ \mathcal{U}] \,\, = \,\, \{ \,[u] \,: u \in \mathcal{U} \backslash \{0\} \,\}. $$
  If $\mathcal{U} \backslash \{0\} $ is an affine variety 
   then $[\mathcal{U}]$ is a projective variety \cite[\S 8.2]{CLO}.
  
  With this, the chain of inclusions in   (\ref{eq:Finclusions})
  is now taken  into projective space:
  \begin{equation}
\label{eq:Finclusions2}
{\rm image}(\sigma'') =  [\mathcal{F}^{\rm im}] \,\subset \, [\mathcal{F}^\RR ]
\,\subset\, [ \mathcal{F}^\CC ]. \end{equation}
We similarly have the inclusion $\,{\rm image}(\sigma') \subset [\mathcal{F}^\CC]\,$ 
in the complex projective space $\PP^{q-1}_\CC$.
If we recast Example~\ref{ex:zweifunf} in the projective setting then the left inclusion 
in (\ref{eq:Finclusions2}) becomes an equality. Indeed, for this example,
  $\, [\mathcal{F}^{\rm im}] \,= \, [\mathcal{F}^\RR ] \,$ in $\PP^3_\RR$
 because every $2 \times 2$ matrix of rank $1$ is either positive or negative semidefinite.
 Similarly, $\,{\rm image}(\sigma')= [\mathcal{F}^\CC]\,$ holds in $\PP^3_\CC$.

The key word {\em Zariski closure} refers to the
passage in (\ref{eq:Finclusions}) or (\ref{eq:Finclusions2})
from images on the left to varieties on the right.
The Zariski closure is the smallest variety containing a given~set.

We have now introduced quite a few decorations for the symbol $\mathcal{F}$,
and these decorations are important when making precise geometric statements
pertaining to the three dichotomies 
$$ \text{{\em image}\, versus {\em closure of image}, \,\,
{\em real} \,versus {\em complex},\,\,
{\em affine}\, versus {\em projective}.
}
$$
In the literature in applied algebraic geometry,
these distinctions are often swept under the rug.
This is a matter of convenience and simplicity.
For the most part, our article will follow that convention.
The symbol $\mathcal{F}$   will 
stand for the complex projective variety $[\mathcal{F}^\CC]$.
Some of our results will pertain to the images $\mathcal{F}^{\rm im}$
and $[\mathcal{F}^{\rm im}]$ over the real numbers,
and in those cases the decorations for $\mathcal{F}$ will be used and highlighted.
Thus $\PP^{q-1}$ is  a projective space over $\CC$,
and  $\mathcal{F}$ is a subvariety in that space. It is tacitly
understood that $\mathcal{F}$ arose from some specific polynomial map $\sigma$
and that $\mathcal{F}$ is a proxy for $\mathcal{F}^{\rm im} = {\rm image}(\sigma)$.
For instance, the map $\sigma$ in Example~\ref{ex:zweifunf} defines a 
 quadratic curve $\mathcal{F}$ in a plane $\PP^2$  that lives in $\PP^3$.

Both of the fields $\RR$ and $\CC$ are peripheral when it comes to 
computations with software.
In most applications, including ours, the coordinates of $\sigma$ 
have coefficients in the field $\QQ$ of rational numbers.
All polynomials seen in this article have rational coefficients.
In particular, the ideal $F$ lives in a polynomial ring with coefficients in $\QQ$.
The variety $\mathcal{F}$  is represented through its ideal~$F$.

\smallskip 

We conclude that the title of this paper refers to the algebro-geometric
study of a certain {\em signature map} $\sigma$.
 Its coordinates are homogeneous polynomials of degree $k$.
 We seek to find the homogeneous prime ideal $F$ of
 relations among these polynomials.
    This is the {\em signature ideal} $F$. Its 
  {\em signature variety} $\mathcal{F}$ is the
   complex projective  variety, unless explicitly stated otherwise.
   The associated {\em signature image}    $\mathcal{F}^{\rm im}$
   is a semialgebraic set in real affine space.
   
Parameter identifiability is an important issue, and we 
shall address this in Section~\ref{sec:DimIde}.
In  Example \ref{ex:zweifunf},
the map $\sigma: \RR^2 \rightarrow \mathcal{F}^{\rm im}$ is two-to-one.
Each nonzero matrix in $\mathcal{F}^{\rm im} $
has two preimages $(u_1,u_2)$ and $(-u_1,-u_2)$.
On the other hand, the map $\sigma' : \PP^1 \rightarrow \mathcal{F}$ is one-to-one.
Every point on the quadratic curve $\mathcal{F}$ is uniquely
represented by a point $[u_1:u_2]$ in the parameter line $\PP^1$.
We express this by saying  that the variety $\mathcal{F}$ is {\em rationally identifiable}.

Throughout this paper we use caligraphic letters for varieties
and roman letters for ideals. So, if $\mathcal{F}$ denotes a variety
then we write $F$ for its radical ideal.  If $\mathcal{F}$ is irreducible then
$F$ is prime, and if $\mathcal{F}$ is projective then $F$ is homogeneous.
For any field  $\KK$, such as $\CC $ or $\RR$ or $\QQ$, we also use the notation
$\mathcal{F}(\KK)$ instead of $\mathcal{F}^\KK$ for the set of points with coordinates in $\KK$.

We illustrate this for some objects that will appear in Section \ref{sec:universal}.
The ($n$-truncated) free Lie algebra over $\RR$ with $d$ generators is denoted by  
${\rm Lie}^n(\RR^d)$. This
 is a linear subspace in an affine space of dimension $d+d^2+\cdots+d^n$, namely the truncated tensor algebra
over $\RR$. Its ideal lives in a polynomial ring over $\QQ$ in
$d+d^2+\cdots+d^n$ variables. It is generated by the linear forms 
$\sigma_{I \,\shuffle \,J}$ in Lemma \ref{lem:shuffleforms}.
The corresponding Lie group $ \Gim_{d,n}  = \exp ({\rm Lie}^n(\RR^d) )$ is
also an affine variety, with ideal $G_{d,n}$ generated by
the quadratic polynomials $\sigma_I \sigma_J -\sigma_{I \,\shuffle \,J}$.
The intersection of this ideal with the polynomial subring for a fixed
level $k$, where the variables $\sigma_I$ satisfy $|I| = k$, is 
a homogeneous prime ideal $U_{d,k}$. The associated 
projective variety $\mathcal{U}_{d,k}$ is our universal variety.
Its subset of real points, denoted $\mathcal{U}_{d,k}(\RR)$, contains
the signature tensor $\sigma^{(k)}(X)$ for any deterministic path $X$ in $\RR^d$.
This is the Chen-Chow Theorem \ref{thm:chen}.

\section{Varieties of Signature Matrices} \label{sec:vsm}

We here study varieties of $d \times d$ signature matrices $S = \sigma^{(2)}(X)$ of
paths $X$ in $\RR^d$. Any $ d \times d$  matrix $S$ can be written uniquely as the sum of
its symmetric part and skew-symmetric part:
$$ \qquad S \,=\, P+Q, \qquad \hbox{where} \quad
P = \frac{1}{2}(S+S^T)  \quad \hbox{and} \quad
Q = \frac{1}{2}(S-S^T).$$
The $\binom{d+1}{2}$ entries $p_{ij}$ of $P$
and the $\binom{d}{2}$ entries $q_{ij}$ of $Q$ serve
as coordinates on the space of matrices, either $\RR^{d \times d}$ or $\PP^{d^2-1}$.
Thus, in this section we transform our coordinates
as follows:
$$\sigma_{ii} = p_{ii} \,,\,\,
\sigma_{ij} = p_{ij} + q_{ij} \quad \hbox{and} \quad
\sigma_{ji} = p_{ij} - q_{ij} \quad \hbox{ for } \quad 1 \leq i < j  \leq d. $$

The signature matrix of a path $X$ in $\RR^d$ has real entries.
We here restrict to certain families of paths that depend on a parameter $m$.
The resulting signature matrices $\sigma^{(2)}(X)$ form
the signature image. This is a semialgebraic subset in the
 real affine space $\RR^{d \times d}$, 
as discussed in Subsection \ref{sec:manylives}.
We shall represent it as an orbit of the group $ {\rm GL}(d,\RR)$
with respect to its natural action on the tensor product
$\,\RR^{d \times d} = \RR^d \otimes \RR^d$.
The signature variety is the Zariski closure of this orbit in the 
projective space $\PP^{d^2-1}$ of complex $d \times d$ matrices.

\subsection{Real Signature Matrices} \label{sec:rsm}

We fix positive integers $m \leq d$.
Let $S_{\rm axis}^{[d,m]}$ denote the $d {\times} d$ matrix
whose upper left $m {\times} m$ block is the upper triangular
matrix $\sigma^{(2)}(C_{\rm axis})$ in Example \ref{ex:cap},
with zeros in all other entries. Similarly, we write $S_{\rm mono}^{[d,m]}$ 
for the $d {\times} d$ matrix whose upper left $m {\times} m$ block is the 
matrix $\sigma^{(2)}(C_{\rm mono})$ in Example \ref{ex:cmp}. 
See equation (\ref{eq:axismono32}) in Example~\ref{ex:axismono32}
for the case $d=3,m=2$. The decompositions of our matrices into
symmetric and skew-symmetric parts are 
\begin{equation}
\label{eq:SPVpencil}
 S_{\rm axis}^{[d,m]} \,= \, 
P_{\rm axis}^{[d,m]} \,+\,
Q_{\rm axis}^{[d,m]} \quad \hbox{and} \quad
S_{\rm mono}^{[d,m]} \,= \, 
P_{\rm mono}^{[d,m]} \,+\,
Q_{\rm mono}^{[d,m]}. 
\end{equation}

A  piecewise linear path $X$ with $m$ segments in $\RR^d$  is represented
by any $d {\times} d$ matrix ${\bf X} = (x_{ij})$ whose first $m$ columns are the steps.
The signature matrix of this path equals
\begin{equation}
\label{eq:XSXaxis}
 \sigma^{(2)}(X) \,\, = \,\, {\bf X}  \cdot S_{\rm axis}^{[d,m]} \cdot {\bf X}^T. 
\end{equation}
The signature image for piecewise linear paths with $m$ segments in $\RR^d$ is the orbit
\begin{equation} \label{def:LS2}
    \Lim_{d,2,m} \,\,:= \,\,\bigl\{ \,{\bf X}  \cdot S_{\rm axis}^{[d,m]}\cdot {\bf X}^T\, : {\bf X} \in \RR^{d \times d} 
\,    \bigr\} 
    \,\, \subset \,\, \RR^{d \times d}.
 \end{equation} 

A polynomial path of degree $m$ in $\RR^d$  is represented
by any $d {\times} d$ matrix ${\bf X} = (x_{ij})$ whose first $m$ columns
give the coefficients of the polynomials. As in (\ref{eq:XSXaxis}),
it has signature matrix 
\begin{equation}
\label{eq:XSXmono}
 \sigma^{(2)}(X) \,\, = \,\, {\bf X}  \cdot S_{\rm mono}^{[d,m]} \cdot {\bf X}^T. 
\end{equation}
The signature image for polynomial paths of degree $m$ in $\RR^d$ is the
resulting orbit
\begin{equation} \label{def:PS2}
     \Pim_{d,2,m}  \,\,:= \,\,\bigl\{ \,{\bf X} \cdot S_{\rm mono}^{[d,m]} \cdot {\bf X}^T\, : {\bf X} \in \RR^{d \times d}
     \,\bigr\} \,
       \,\, \subset \,\, \RR^{d \times d}.  
\end{equation} 

\begin{example}[$m=1$ and Veronese] \rm The signature image for linear paths equals
$$      \Lim_{d,2,1} \,=\, \Pim_{d,2,1}\, = \,\bigl\{ \tfrac{1}{2} X X^T \ : \ X \in \RR^d \bigr\}  .
$$      
These are the symmetric positive definite matrices of rank $\leq 1$.
The Zariski closure of the set $[\Lim_{d,2,1}]$ in $\PP^{d^2-1}$
 is the {\em Veronese variety},  which consists of all
complex symmetric rank $1$ matrices, regarded up to scaling.
\hfill $\diamondsuit$
\end{example} 

\begin{example}[$m=2$] \label{ex:matrix2m} \rm
The signature image for two-step paths equals
\begin{equation}
\label{eq:PVsatisfies}
  \Lim_{d,2,2} \,=\, \bigl\{  \tfrac{1}{2} (X_1 + X_2)  (X_1 + X_2)^T  +  \tfrac{1}{2} (X_1  X_2^T - X_2 X_1^T) \ : X_1,X_2 \in \RR^d \bigr\}  .
\end{equation}  
For a matrix $S = P+Q \in \Lim_{d,2,2}$,
the symmetric part $P$ lies in $\Lim_{d,2,1}$ and the
skew-symmetric part $Q$ has rank $\le 2$. 
The orthogonal complement of  ${\rm span} ( X_1,X_2)$ lies in the kernel of $P$ and of $Q$.
Hence, the  $d \times 2d$ matrix $[\,P\, Q\,]$, built by concatenating $P$ and~$Q$, has rank $\le 2$. 
 
 The signature image $\Lim_{d,2,2}$ is not closed.
 To see this, consider any matrix $S$ with $P = 0$ and $Q$ skew-symmetric of rank $2$.
 Then $S$ satisfies the constraints above but it does not lie in $\Lim_{d,2,2}$.
Otherwise, $P = 0$ would imply $X_1 = -X_2$ and hence $Q = 0$.
However $S$ lies in the closure of $\Lim_{d,2,2}$. One can construct a sequence
of nearly antipodal pairs  $(X^\varepsilon_1, X^\varepsilon_2) \in (\RR^d)^2$ such that
$Q^\varepsilon = Q$ for all $\varepsilon > 0$ and $P^\varepsilon \to 0$    for $\varepsilon \to 0$.  
The same phenomenon arises for polynomial paths of the form (\ref{quadratic}) of degree $m=2$.
A quadratic path with $P = 0$ satisfies $Q = 0$. Such paths
 are {\em tree-like} \cite{HL}. A typical case is the path
  $ t \mapsto  \bigl( t(1{-}t), t (1{-}t) \bigr ) $. 
  \hfill $\diamondsuit$
\end{example} 

The following theorem elucidates the relationship between piecewise linear  paths and polynomial paths
in $\RR^d$. They have the same $d \times d$ signature matrices.

\begin{theorem} \label{thm:realmatrixstuff}
If $k=2$ then the signature image for piecewise linear paths equals
the signature image  for polynomial paths.
In symbols, $\,\Lim_{d,2,m}=\Pim_{d,2,m}  \subset  \RR^{d \times d} \ . $
\end{theorem} 

\begin{proof} We abbreviate the matrices in (\ref{eq:SPVpencil}) by
$ A:= S_{\rm axis}^{[d,m]}$ and $M := S_{\rm mono}^{[d,m]} $.
It is enough to treat the case $m=d$, since this proves the result for all $m \le d$.
We shall build an invertible matrix $H$ such that $H M H^T = A$. 
This will prove that $A$ and $M$ have the same orbit under the action by ${\rm GL}(d,\RR)$
which clearly yields the desired set identity in $\RR^{d \times d}$.
The proof that follows proceeds by induction on~$d$; to this end, introduce the ``refined'' notation $M^{[d]} = S_{\rm mono}^{[d,d]}$.
For $d=1$ we have $A=M=(\tfrac{1}{2})$ and  $H=(1)$.
To go from dimension $d$ to $d+1$ we write
\[ M^{[d+1]}
\,\,= \,\,
\left(\begin{array}{@{}c|c@{}}
  M &
  \begin{matrix}
  \tfrac{d+1}{d+2} \\
  \vdots \\
  \end{matrix}
\\ \hline
  \begin{matrix}
  \tfrac{1}{d+1} \dots 
    \end{matrix}
  & \tfrac{d+1}{2(d+1)}
\end{array}\right) \,\, = \,\,
\left(\begin{array}{@{}c|c@{}}   M &
  \begin{matrix}    \vdots \\    \end{matrix}
\\ \hline  \begin{matrix}    \dots      \end{matrix}
  & \tfrac{1}{2}
\end{array}\right).
\]
Here $M = M^{[d]}$.
Our goal is to find a row vector $x \in \RR^d$ and a scalar $y \in \RR$ with
\begin{equation} \label{equ:matrixidentityforLP}
\left(\begin{array}{@{}c|c@{}}
  \begin{matrix}
  x^T 
  \end{matrix}
  & y
\\ \hline
  H
  &
  \begin{matrix}
  0^T
  \end{matrix}
\end{array}\right)
\left(\begin{array}{@{}c|c@{}}
  M &
  \begin{matrix}
  \vdots \\
  \end{matrix}
\\ \hline
  \begin{matrix}
  \dots
    \end{matrix}
  & \tfrac{1}{2} 
\end{array}\right)
\left(\begin{array}{@{}c|c@{}}
  \begin{matrix}
  x 
  \end{matrix}
  & H^T 
\\ \hline
  \begin{matrix}
  y
  \end{matrix}
  &0 
\end{array}\right)
\,\,=\,\,
\left(\begin{array}{@{}c|c@{}}
  \begin{matrix}
  * 
  \end{matrix}
  & **
\\ \hline
  \begin{matrix}
  o
  \end{matrix}
  & H M H^T 
\end{array}\right)
\,\, =\,\,
\begin{tiny}
\left(\begin{array}{@{}c|c@{}}
  \begin{matrix}
  \tfrac{1}{2}
  \end{matrix}
  & 1  \cdots 1
\\ \hline
  \begin{matrix}
  0 \\
  \vdots \\
  0
  \end{matrix}
  & A  
\end{array}\right)
\end{tiny}
\,\, = \,\, S_{\rm axis}^{[d+1,m]} .
\end{equation}
The identity in the middle requires that  $(*) = \tfrac{1}{2}$, $(**) = (1 \cdots 1)$ and $o = 0^T= (0, \ldots, 0)^T$. 
We carry out the (block) matrix multiplication, starting with
\[
\left(\begin{array}{@{}c|c@{}}
  M &
  \begin{matrix}
  \vdots \\
  \end{matrix}
\\ \hline
  \begin{matrix}
  \dots 
  \end{matrix}
  & \tfrac{1}{2} 
\end{array}\right)
\left(\begin{array}{@{}c|c@{}}
  \begin{matrix}
  x 
  \end{matrix}
  & H^T 
\\ \hline
  \begin{matrix}
  y
  \end{matrix}
  &0 
\end{array}\right)
=
\left(\begin{array}{@{}c|c@{}}
  \begin{matrix}
  M x + y \Big( \vdots \Big)
  \end{matrix}
  & M H^T
\\ \hline
  (\cdots)x + \tfrac{1}{2} y
  &
  \begin{matrix}
  (\cdots)H^T
  \end{matrix}
\end{array}\right) .
\]
The requirement $o = 0^T$ in (\ref{equ:matrixidentityforLP}) implies
$\,    H \Big(M x + y \Big( \vdots \Big)\Big) = 0 \,$ and hence $\, M x + y \Big( \vdots \Big) = 0 $.
With this, the condition $(*) = \tfrac{1}{2}$  translates into one quadratic equation in
$d+1 $ unknowns $(x,y)$:
\begin{equation}
\label{eq:onequadratic}
          y ((\cdots)x + \tfrac{1}{2} y)\,\, =\,\, y \left( \sum_{i=1}^d \tfrac{i}{d+1+i} x_i + \tfrac{1}{2} y \right)
          \,\, =\,\, 1/2.
\end{equation}
The final condition $(**) = (1,\ldots 1)$ gives $d$  linear equations for $(x,y) \in \RR^{d+1}$:
\begin{equation} \label{solve4x}
        x^T M H^T + y ( \cdots) H^T \,=\, (1, \ldots , 1) \quad {\rm in} \,\,\, \RR^d .
\end{equation}
Here $(\cdots) = (\tfrac{1}{d+2},..., \tfrac{d}{2d+1})$. 
The solution set to (\ref{solve4x}) is a line in $\RR^{d+1}$. For
any given $y \in \RR$, there is a unique solution $x=x(y)$. 
(It is an a fortiori consequence of (\ref{eq:cauchymatrix}) below that $M$ is invertible, as is $M H^T$ using the induction hypothesis.) 
Substituting this  into (\ref{eq:onequadratic}), we obtain a quadratic equation in one variable $y$.
 Let $y$ be one of its complex solutions. We claim that $y$ is real.
 
 Let $\tilde H$ be the desired $(d+1) \times (d+1)$ transformation matrix on the left of
  (\ref{equ:matrixidentityforLP}). By Laplace expansion with respect to the last column,
  its determinant equals $\,\det \tilde H = \pm y \det H$. Taking determinants on
  both sides of  (\ref{equ:matrixidentityforLP}), and using $\det H = \det H^T \not= 0$, we conclude
  $$ y^2 \cdot \det (H)^2 \cdot  \det \bigl( M^{[d+1]} \bigr)  \, =\, \tfrac{1}{2} \det (A) \,= \,\frac{1}{2^{d+1}} . $$
At this point it suffices to know that  $\det (M^{[d]}) >0$ for all $d \in \mathbb{N}$. In fact, we claim 
\begin{equation}
\label{eq:cauchymatrix}
 {\rm det}(M^{[d]}) \,\, = \,\, d! \cdot \frac{\prod_{1 \leq i < j \leq d} (j-i)^2}{\prod_{i=1}^d \prod_{j=1}^d (i+j)}. 
 \end{equation}
This identity holds because $M^{[d]}$ becomes a {\em Cauchy matrix} after we divide its $j$th column by $j$.
This positivity of ${\rm det}(M^{[d]})$ implies that $y^2$ is positive and hence $y$ is real.  From (\ref{solve4x}),
we  obtain a (real) vector $x=x(y)$. The resulting matrix $\tilde H$ completes the induction step.
\end{proof}

\subsection{Determinantal Varieties}

 The precise characterization of signature images is a subtle matter. 
We already saw this in Example \ref{ex:matrix2m}. In this article we
address the easier problem of characterizing the polynomial equations
that vanish on these images. In other words, we study the signature
varieties and their ideals.  The inequalities that hold on
the signature images are left to future research.

Let $M_{d,m}$ denote the prime ideal of homogeneous polynomials that vanishes on
$\,\Lim_{d,2,m}=\Pim_{d,2,m} $. This lives in
the polynomial ring over $\QQ$ with $d^2$ variables, namely the
$\binom{d+1}{2}$ variables $p_{ij}$ and the $\binom{d}{2}$ variables $q_{ij}$.
The corresponding variety $\mathcal{M}_{d,m}$ lives in
the projective space $\PP^{d^2-1}$  of $d \times d$ matrices.
 This signature variety is the Zariski closure of the signature image.
Our main result in this section characterizes this variety in terms of 
the $d \times 2d$ matrix $[P \, Q ]$.

\begin{theorem} \label{thm:matricesmain}
For each $d$ and $m$, the following varieties in 
$\,\PP^{d^2-1}$ coincide:
\begin{enumerate}
\item The variety  of signature matrices 
 of piecewise linear paths with $m$ segments.
\vspace{-0.17cm}
\item The variety of signature matrices 
 of polynomial paths of degree $m$.
\vspace{-0.17cm}
\item The variety of matrices $\,S = P+Q$, with $P$ symmetric and $Q$ skew-symmetric, such that 
$\,{\rm rank}(P) \leq 1\,$ and $\,{\rm rank} \bigl( [\,P\, Q\,] \bigr) \leq m$.
\end{enumerate}
For each fixed value of $d$, we have the following chain of  varieties in $\PP^{d^2-1}$:
\begin{equation}
\label{eq:Mchain}
\mathcal{M}_{d,1} \subset \mathcal{M}_{d,2} \subset \mathcal{M}_{d,3} \subset \,
\cdots \,\subset \mathcal{M}_{d,d} = \mathcal{M}_{d,d+1}
= \mathcal{M}_{d,d+2} = \cdots
\end{equation}
Fix $m \leq d$. Then $\mathcal{M}_{d,m}$  is irreducible of dimension $md-\binom{m}{2}-1$.
If $m$ is odd then its ideal $M_{d,m}$ is generated by
 the $2 $-minors of $P$ and  the $(m+1) $-pfaffians of $Q$.
 If $m$ is even then $M_{d,m}$ is generated by
 the $2 $-minors of $P$,  the $(m+2) $-pfaffians of $Q$,
 and the entries in $P \cdot C_m(Q)$  where $C_m(Q)$ is the circuit matrix
 of $m$-pfaffians.
   \end{theorem}

The matrix $Q$ is skew-symmetric. Its principal minors of odd size are zero.
Each principal minor of even size $m$ is the square of a polynomial
of degree $m/2$ in the $q_{ij}$. These polynomials are the {\em $m$-pfaffians} of $Q$.
The {\em circuit matrix} $C_m(Q)$ has format
$d \times \binom{d}{m+1}$, with columns indexed by the
$(m{+}1)$-element subsets $I$ of $\{1,2,\ldots,d\}$.
The entry in row $i$ and column $I$ is $0$ unless
 $i \in I$. In that case it equals the $m$-pfaffian of $Q$ that is indexed by
$I \backslash \{ i \} $. Here the pfaffian must be taken with the correct sign.
 For instance, if $d=6$ and $m=4$ then
$$ C_4(Q) =  \begin{tiny} \begin{bmatrix}
0 & \phantom{-}q_{34} q_{56}-q_{35} q_{46}+q_{36} q_{45} 
  &  \phantom{-}q_{24} q_{56}-q_{25} q_{46}+q_{26} q_{45} 
& \cdots & 
    q_{23} q_{45}-q_{24} q_{35}+q_{25} q_{34} \\
  -q_{34} q_{56} + q_{35} q_{46} - q_{36} q_{45} & 0 & 
\phantom{-}   q_{14} q_{56} - q_{15} q_{46} + q_{16} q_{45} 
& \cdots &
   q_{13} q_{45} - q_{14} q_{35} + q_{15} q_{34} \\
 - q_{24} q_{56}+q_{25} q_{46}-q_{26} q_{45} & 
 - q_{14} q_{56}+q_{15} q_{46}-q_{16} q_{45} & 0 &
 \cdots &
   q_{12} q_{45}-q_{14} q_{25}+q_{15} q_{24} \\
    -q_{23} q_{56}+q_{25} q_{36}-q_{26} q_{35} & 
    -q_{13} q_{56}+q_{15} q_{36}-q_{16} q_{35} &
    -q_{12} q_{56}+q_{15} q_{26}-q_{16} q_{25} & 
\cdots &
     q_{12} q_{35}-q_{13} q_{25}+q_{15} q_{23}  \\
       -q_{23} q_{46} + q_{24} q_{36} - q_{26} q_{34} & 
  -q_{13} q_{46} + q_{14} q_{36} - q_{16} q_{34} & 
    -q_{12} q_{46} + q_{14} q_{26} - q_{16} q_{24} &  
\cdots &
 q_{12} q_{34} - q_{13} q_{24} + q_{14} q_{23} \\
    -q_{23} q_{45} + q_{24} q_{35} - q_{25} q_{34} & 
    -q_{13} q_{45} + q_{14} q_{35} - q_{15} q_{34} &
    -q_{12} q_{45} + q_{14} q_{25} - q_{15} q_{24} 
&    \cdots & 0 
\end{bmatrix} \! . \end{tiny}
$$
The entries of this matrix are the $2$-pfaffians of the $6 \times 6$ matrix $Q$.
The $4$-pfaffians are homogeneous quartics with $15$ terms.
For any $m$ and $d$,  the columns of the
circuit matrix $C_d(Q)$ are canonical generators for the kernel of the matrix $Q$,
provided $Q$ has rank $m$. Think of Cramer's rule for ${\rm ker}(Q)$.
These generators correspond to {\em circuits} in matroid theory \cite{OXL}.

\begin{example}[$d=3,m=2$] \label{ex:axismono32} \rm
The variety $\mathcal{M}_{3,2}$ has dimension $4$ and degree $6$ in $\PP^8$. It is 
the Zariski closure of the common orbit of the following two matrices under the ${\rm GL}(3,\RR)$-action:
\begin{equation}
\label{eq:axismono32}
S_{\rm axis}^{[3,2]}\,=\, \begin{pmatrix}
\,\, \frac{1}{2} & 1 & 0 \, \medskip \\
\,\, 0 & \frac{1}{2} & 0 \,\medskip \\
\,\, 0 & 0 & 0 \, \,\end{pmatrix}
\qquad \hbox{or} \qquad\,
S_{\rm mono}^{[3,2]}\,=\, 
 \begin{pmatrix} 
\,\,\frac{1}{2} & \frac{2}{3} & 0  \,\medskip \\
\,\,\frac{1}{3} & \frac{2}{4} & 0 \,\medskip \\
\,\, 0 & 0 & 0 \, \,\end{pmatrix}.
\end{equation}
The variety $\mathcal{M}_{3,2}$ is cut out by  the $2 \times 2$ minors of  $P = (p_{ij})$ and
the $3 \times 3$ minors~of 
$$ [P \, Q ] \quad = \quad
\begin{bmatrix}
 \, p_{11} & p_{12} & p_{13} & 0  & q_{12} & q_{13} \,\\
\, p_{12} & p_{22} & p_{23} & -q_{12}  & 0 & q_{23} \,\\
\,  p_{13} & p_{23} & p_{33} & -q_{13} & \! -q_{23} & 0 \,
 \end{bmatrix}.
 $$
 However, these do not generate the prime ideal $M_{3,2}$.
For that we need
 nine quadrics, namely the six $2$-minors of $P$ and the three entries of
 $P \cdot C_2(Q)$ where $C_2(Q) = \bigl( q_{23}, - q_{13}, q_{12}\bigr)^T$.
 The fourfold $\mathcal{M}_{3,2}$ contains the 
 Veronese surface $\mathcal{M}_{3,1} \simeq \PP^2$, which lives in
 a $\PP^5$ inside $\PP^8$.
 The ideal $M_{3,1}$ of that surface is generated by $\{q_{12},q_{13},q_{23}\}$
and the $2$-minors of~$P$.
\hfill $\diamondsuit$  \end{example}

 \begin{proof}[Proof of Theorem  \ref{thm:matricesmain}]
 We consider pencils spanned by a symmetric matrix $P$
 and a skew-symmetric matrix $Q$. 
 These are called {\em T-even matrix pencils} in the
 linear algebra literature. Their normal forms
 are found e.g.~in \cite{LR, Thom}. They correspond to
  normal forms of $S = P+Q$ under the ${\rm GL}(d,\CC)$-action 
  by {\em congruence} $S \mapsto {\bf X} S {\bf X}^T$.
By Theorem  \ref{thm:realmatrixstuff},
  our two canonical pencils in    (\ref{eq:SPVpencil}) are in the
  same orbit. It consists of rank $m$ pencils $(P,Q)$ where $P$ has rank~$1$.
  The closure of this orbit also contains orbits where the rank drops from $m$
  to $\leq m-1$.  
  
  If $m$ is odd then $Q$ has  rank $m-1$.
  There is no constraint imposed on the rank $1$ matrix~$P$.
  The normal form is the direct sum of 
one pencil $  \bigl((1),(0) \bigr)$  and $d-m$ zero pencils
and $\frac{m-1}{2}$ pencils
\begin{tiny} $\biggl( \begin{pmatrix} 0 & 0 \\ 0 & 0 \end{pmatrix},
  \begin{pmatrix} \,0 & 1 \\ \! -1 & 0 \end{pmatrix} \biggr)$. \end{tiny}
  If $m$ is even then ${\rm image}(P) \subseteq {\rm image}(Q)$ holds
  on the orbit in question.   The normal form is the direct sum of one pencil
\begin{tiny} $\biggl( \begin{pmatrix} 1 & 0 \\ 0 & 0 \end{pmatrix},
  \begin{pmatrix} \,0 & 1 \\ \! -1 & 0 \end{pmatrix} \biggr)$ \end{tiny}
  and $\frac{m-2}{2}$ pencils
\begin{tiny} $\biggl( \begin{pmatrix} 0 & 0 \\ 0 & 0 \end{pmatrix},
  \begin{pmatrix} \,0 & 1 \\ \! -1 & 0 \end{pmatrix} \biggr)$. \end{tiny}
The rank constraints in item 3 characterize the closure of this orbit.
 We conclude that the variety
in item 3 is indeed irreducible and equals those in items 1 and~2.

To prove our formula for the dimension, we start with two binomial identities:
$$ \begin{matrix} md - \binom{m}{2} \,\, = \,\,  \bigl[ \binom{d}{2} - \binom{d-m}{2} \bigr]\, + m 
\,\, = \,\, \bigl[ \binom{d}{2} - \binom{d-(m-1)}{2} \bigr] \, + d. \end{matrix} $$
We claim that this is the dimension of the affine cone over $\mathcal{M}_{d,m}$.
The parenthesized expressions give the dimension of the variety of
skew-symmetric $d \times d$ matrices $Q$ of rank $\leq m$. 
When $m$ is odd then the formula on the right is used.
In that case we can set $P =v v^T $ where $v$ is any column vector in $\RR^d$.
When $m$ is even then the formula in the middle is used.
Given $Q$, we now choose $v$ in the column span of $Q$, so there are 
$m$ degrees of freedom for~$P$.

It remains to establish the assertions about the prime ideal $M_{d,m}$
associated with $\mathcal{M}_{d,m}$.
The  polynomials we list are in that ideal, and they cut out the
variety $\mathcal{M}_{d,m}$ set-theoretically.  We must show that they generate a prime ideal.
If $m$  is odd then this follows from the well-known fact that the 
$2$-minors of $P$  and the $(m+1)$-pfaffians of $Q$ both generate prime ideals.
Since they share no variables, the sum of these two ideals is also prime.
If $m$ is even then we must argue that incorporating the
entries of $P \cdot C_m(Q)$, which are homogeneous of bidegree $(1,m/2)$, 
does not destroy primality. This can be shown by computing the reduced 
Gr\"obner basis with respect to the lexicographic term order,
where the matrix entries are ordered row-wise with 
the $p_{ij}$ coming before the $q_{ij}$.
The leading terms in that Gr\"obner basis are square-free.
The initial ideal is radical, and hence so is the ideal in question.
Since it cuts out the correct irreducible variety,
Hilbert's Nullstellensatz ensures that the ideal is prime. 
 \end{proof}

\begin{table}[h] 
\begin{tabular}{ | l | l | l | l | l | l | l | l | p{1.5cm} |} \hline  $d$ & $k$ & $m$ & $a$ & 
$\!\dim\!$& 
$\! \deg \! $ & $\!$gens$\!\!$ \\ 
\hline 2 & 2 & 1 &2  & 1 & 2 & 1 \\
\hline 2 & 2 & 2 &3 & 2 & 2 & 1 \\
\hline 3 & 2 & 1 &5  & 2 & 4 & 6 \\
\hline 3 & 2 & 2 & 8 & 4 & 6 & 9 \\
\hline 3 & 2 & 3 & 8 &  5& 4 &6  \\
\hline 4 & 2 & 1 & 9 &  3&  8 & 20 \\
\hline 4 & 2 & 2 & 15 &  6&  20 & 36  \\
\hline 4 & 2 & 3 & 15&  8& 16 & 21 \\
\hline 4 & 2 & 4 & 15 &  9& 8 & 20 \\
\hline 5 & 2 & 1 & 14 &  4& 16 &  50\\
\hline 
 \end{tabular} 
\quad 
\begin{tabular}{ | l | l | l | l | l | l | l | l | p{1.5cm} |} \hline  $d$ & $k$ & $m$ & $a$ &
$\!\dim\!$ & $\!\deg\!$ & $\!$gens$\!\!$ \\ 
\hline 5 & 2 & 2 &  24&  8&  70&  100 \\
\hline 5 & 2 & 3 &  24&  11& 80 &  55\\
\hline 5 & 2 & 4 &  24&  13&  40&  50, 5\\
\hline 5 & 2 & 5 & 24 &  14&  16&  50\\
\hline 6 & 2 & 1 & 20 &  5  & 32 & 105 \\
\hline 6 & 2 & 2 & 35 &  10  &  252 & 225 \\
\hline 6 & 2 & 3 &  35&  14  & 448 & 120  \\
\hline 6 & 2 & 4 &  35&  17  &  280 & $\!$ 105, 36 $\!$ \\
\hline 6 & 2 & 5 &  35&  19  &  96 & 105, 1\\
\hline 6 & 2 & 6 &  35&  20& 32 & 105  \\
\hline 
 \end{tabular} 
\caption{\label{tab:Mdm} Invariants of the ideal $M_{d,m}$ that defines the variety
      of signature matrices  $\mathcal{M}_{d,m}$.} \medskip
  \end{table}

Table \ref{tab:Mdm} lists invariants of the prime ideals $M_{d,m}$.
These can be computed in {\tt Macaulay2} \cite{M2}. The same format is used
in later sections when we pass from matrices to tensors. The column ``$a$''
denotes the {\em ambient dimension}. This refers to
the smallest  linear subspace of $\PP^{d^2-1}$ containing $\mathcal{M}_{d,m}$.
 For instance, $M_{3,1}$ is an ideal in a polynomial
ring with $9$ unknowns and it contains $3$ linear forms. Hence, the ambient
dimension is $5 = 9-3-1$. Indeed,  the  quadratic Veronese surface $\mathcal{M}_{3,1}
\simeq \PP^2$ lives in $\PP^5$. The column ``dim'' displays
${\rm dim}(\mathcal{M}_{d,m}) = md-\binom{m}{2} -1$.
The column ``deg'' displays the degree of $\mathcal{M}_{d,m}$.
For $m$ odd, this is $ 2^{d-1}$ times the degree of the
variety of skew-symmetric matrices of  rank $m-1$. The latter is 
a Catalan number when $m=3$, which explains
${\rm deg}(\mathcal{M}_{d,3})  = \frac{2^{d-1}}{d-1} \binom{2d-4}{d-2} $.
Note also ${\rm deg}(\mathcal{M}_{d,2})  =  \binom{2d-2}{d-1} $.
The last column ``gens'' lists $\mu_2,\mu_3,\ldots$
where $\mu_i$ is the number of minimal generators of $M_{d,m}$ in
degree $i$. For instance, $M_{6,3}$ is generated by
$120$ quadrics, while $M_{6,4}$ is generated by 
$105$ quadrics and $36$ cubics. The former express ${\rm rank}(P)=1$.
The latter come from  $P \cdot C_4(Q)$.

\section{Universal Varieties from Free Lie Algebras} \label{sec:universal}

The $k$th signature tensor of a path $X$ in $\RR^d$ is an element $\sigma^{(k)}(X)$ of
the tensor space~$(\RR^d)^{\otimes k}$. We denote the
coordinates on this space by $\sigma_{i_1 i_2 \cdots i_k}$
for $ 1 \leq i_1,i_2,\ldots,i_k \leq d$. We write $\RR[\sigma^{(k)}]$ for the polynomial ring
generated by these $d^k$ coordinates. This is the 
homogeneous coordinate ring of the projective space $\,\PP^{d^k-1}$,
taken over either $\RR$ or $\CC$.
The general linear group $G = {\rm GL}(d,\RR)$ acts naturally 
on the spaces $(\RR^d)^{\otimes k}$ and $\PP^{d^k-1}$,
and on the ring $\RR[\sigma^{(k)}]$.

Signature matrices $\sigma^{(2)}(X)$ were characterized in Section \ref{sec:vsm}.
In this section we deal with signature tensors of order $k \geq 3$. 
We shall derive their signature varieties in $\PP^{d^k-1}$.
This rests on the theory of free Lie algebras \cite{Reu}.

\subsection{Tensor Algebra, Lie Polynomials and Shuffles}

For any positive integer $n$, we consider the {\em truncated tensor algebra}
\begin{equation}
\label{eq:Tn} T^n(\RR^d) \,\,\, = \,\,\, \bigoplus_{k=0}^n \,(\RR^d)^{\otimes k} . 
\end{equation}
This is a non-commutative algebra whose multiplication is the tensor (or {\it concatenation}) product,
where tensors of order $\geq n+1$ are set to zero.
We write $T_\nu^n(\RR^d)$ for the hyperplane in $T^n(\RR^d)$
that consists of tensor polynomials with constant term $\nu$.
The standard basis of $\RR^d$ is denoted by
$e_1,e_2,\ldots,e_d$. The induced standard basis of $T^n(\RR^d)$ is abbreviated
$$ \qquad e_{i_1 i_2 \cdots i_k} \,\,  := \,\, 
e_{i_1} \otimes e_{i_2} \otimes \cdots \otimes e_{i_k} 
\qquad {\rm for} \,\, 1 \leq i_1 , i_2 , \ldots , i_k \leq d 
\,\, \,{\rm and} \,\,\,0 \leq k \leq n . $$
We refer to basis vectors with $k \geq 1$ as {\em words} and we 
often write them simply as $i_1 i_2 \cdots i_k$. The empty word spans the 
one-dimensional space $(\RR^d)^{\otimes 0}  \cong \RR$ of constant terms in $T^n(\RR^d)$.

Following \cite[\S 4]{HK} and \cite{Reu}, 
 the truncated tensor algebra $T^n(\RR^d)$ also carries several
 additional algebraic structures. First, it is a {\em Lie algebra}, where the Lie bracket is the commutator
 $$ [P, Q ] \,\, = \,\, P \otimes Q - Q \otimes P \qquad {\rm for} \,\,\, P,Q \in T^n(\RR^d). $$
Next, $T^n(\RR^d)$ is a commutative algebra with respect to the {\em shuffle product}  
$\shuffle$, where words of length $\geq n+1$ are set to zero.
The shuffle product of two words of lengths $r$ and $s$ (with $r+s \le n$) is the sum over the 
$\binom{r+s}{s}$ ways of interleaving the two words.
For a more formal definition see Reutenauer's book \cite[\S 1.4]{Reu}.
Here are some examples for the shuffle product:
$$ \begin{matrix}
 e_{12 \,\shuffle \,34}  =  e_{12} \shuffle e_{34}  =  e_{1234} + e_{1324} + e_{1342} + e_{3124} + e_{3142} + e_{3412},  \\
 e_{3 \,\shuffle 134}  =  e_{3} \shuffle e_{134}  = 
 e_{3134} + 2 e_{1334} + e_{1343}\,, \qquad e_{21 \, \shuffle \,21}  = 2 e_{2121} + 4 e_{2211}, \\
 e_{1 \,\shuffle \,111} = 4 e_{1111} \, , \quad
e_{11 \,\shuffle \,11} = 6 e_{1111}\, , \quad 
e_{12 \,\shuffle \,21} = 2 e_{1221} + e_{1212} + e_{2121} + 2 e_{2112}.
 \end{matrix}
 $$

The tensor algebra also carries two {\it coproducts}: the {\it deconcatenation} coproduct $\Delta_\otimes$ and the {\it deshuffle coproduct} $\Delta_\shuffle$, given by dualizing the corresponding products. One so
obtains two bialgebra (in fact, Hopf algebra) structures that are in natural duality, 
see e.g.~\cite[Proposition 1.9]{Reu}. The truncated tensor algebra naturally
inherits most of these structures.

\smallskip

The ring of polynomial functions on $T^n(\RR^d)$ is
denoted by   $\RR[ \sigma^{(\leq n)}]$. The coordinate functions are
$\sigma_I$, where $I$ runs over all words of length $\leq n$.
For two words $I$ and $J$ of length $\geq 1$, let 
$\sigma_{I \,\shuffle \,J}$ be the linear form in $\RR[\sigma^{(\leq n)}]$
that corresponds to the shuffle product $e_{I \,\shuffle \,J}$.
We refer to $\sigma_{I \,\shuffle \,J}$ as a {\em shuffle linear form}. For instance,
here are some examples of shuffle linear forms:
\begin{equation}
\label{eq:shuffleforms}
 \begin{matrix}
 \sigma_{12 \,\shuffle \,34}   \, =\,  \sigma_{1234} + \sigma_{1324} + \sigma_{1342} + \sigma_{3124} + \sigma_{3142} + \sigma_{3412},  \\
 \sigma_{3 \,\shuffle 134}  \,=  \,
 \sigma_{3134} + 2 \sigma_{1334} + \sigma_{1343}\,, \,\, \sigma_{21 \, \shuffle \,21}  \,=\,
  2 \sigma_{2121} + 4 \sigma_{2211}, \,\, \\
 \sigma_{1 \,\shuffle \,111} = 4 \sigma_{1111} \, , \quad
\sigma_{11 \,\shuffle \,11} = 6 \sigma_{1111}\, , \quad \\
\sigma_{12 \,\shuffle \,21} = 2 \sigma_{1221} + \sigma_{1212} + \sigma_{2121} + 2 \sigma_{2112}.
 \end{matrix}
\end{equation}

An element of $T^n(\RR^d)$ is a {\em Lie polynomial} if it
can be obtained from the standard basis vectors $e_1,e_2,\ldots,e_d$
by iterating the operations of taking Lie brackets and linear combinations.
The resulting set of all Lie polynomials of degree $\leq n$ is a vector space,
denoted by ${\rm Lie}^n(\RR^d)$.  Equivalently, ${\rm Lie}^n(\RR^d)$ is the
smallest Lie subalgebra of $T^n(\RR^d)$ containing $\RR^d$. By construction, all
Lie polynomials are elements of  $\,T_0^n(\RR^d) = \{ 0 \} \oplus \RR^d \oplus \cdots \oplus (\RR^d)^{\otimes n}$. 

\begin{lemma} \label{lem:shuffleforms}
Lie polynomials are characterized by the vanishing of all shuffle linear forms:
$$ {\rm Lie}^n(\RR^d) \,\, = \,\,\bigl\{ \,P \in T_0^{n}(\RR^d) \,:\,\,
\sigma_{I \,\shuffle \,J}(P) = 0 \,\,\, \hbox{for all (non-empty) words} \,\,I,J \,\bigr\}. $$
\end{lemma}

\begin{proof} This is the equivalence between (i) and (iv) in \cite[Theorem 3.1]{Reu}.
\end{proof} 
 
 We next recall (e.g.~from \cite[Section 1.1]{Reu}) 
 that the familiar series for the exponential function and the logarithm function
 determine polynomial maps from $\,T^n(\RR^d)\, $ to itself:
 $$ {\rm exp}(P) \,= \, \sum_{r \geq 0} \frac{1}{r!} P^{\otimes r} \quad {\rm and} \quad
 {\rm log}(1 + P) \,=\, \sum_{r \geq 1} \frac{(-1)^{r-1}}{r} P^{\otimes r} .
 $$
 The image of $T_0^n(\RR^d)$ under the exponential function is the set $T_1^n(\RR^d)$ of
   tensor polynomials with constant term $1$.
 The logarithm function inverts the exponential function on its image:
 $$\quad {\rm log}({\rm exp}(P))\, =\, P \qquad \hbox{for all $P \in T_0^n(\RR^d)$}. $$

The  {\em step-$n$ free nilpotent Lie group} is the image of the 
free Lie algebra under the exponential map:
\begin{equation}
\label{eq:Gnd}
\mathcal{G}^n (\RR^d) \,:=\,\, {\rm exp}({\rm Lie}^n(\RR^d))\,\, \subset \,\,T_1^n (\RR^d) .
\end{equation}
It is known (e.g.~\cite[Theorem 7.30]{FV}) that $ \mathcal{G}^n (\RR^d)  $ is a Lie group, 
with polynomial group law. It
plays a central role as state space of geometric rough paths (cf. Section \ref{sec:rough}).
Elements in $\mathcal{G}^n (\RR^d) $ are also known as {\it group-like elements} \cite{HK, Reu}. 

\begin{lemma}  \label{lem:GroupLike}
Group-like elements are characterized by multiplicativity:
$$ \mathcal{G}^n(\RR^d)  = \bigl\{ P \in T_1^{n}(\RR^d) :
\sigma_{I \,\shuffle \,J}(P) = \sigma_{I}(P) \sigma_{J}(P) \,\,\, \hbox{for all} \,\,I,J \, 
\hbox{with} \,\,|I| + |J| \leq n \,\bigr\}. $$
\end{lemma}

\begin{proof} This is a reformulation of \cite[Theorem 3.2 (ii)]{Reu}. See also  \cite[(4.2)]{HK}.
\end{proof} 

\begin{remark}[From $\RR$ to $\KK$] \rm The definitions above carry over verbatim
from the real numbers $\RR$ to an arbitrary field $\KK$ of characteristic zero.
Lemmas \ref{lem:shuffleforms} and \ref{lem:GroupLike} remain unchanged.
They furnish equational characterizations of
the Lie algebra ${\rm Lie}^n(\KK^d)$ and the Lie group $\mathcal{G}^n(\KK^d)$.
The former is a linear space and the latter is an affine algebraic variety.
\end{remark}

We now come to the connection with paths. It is established by the following fundamental 
result due to Chen \cite{Chen57, Chen58}.
This can also be viewed as a consequence of Chow's work in~\cite{Chow40}.

\begin{theorem}[Chen-Chow] \label{thm:chen}
The step-$n$ free nilpotent Lie group $ \mathcal{G}^n (\RR^d) $ is precisely the 
image of the step $n$ signature map in (\ref{eq:tensorsupton}) when
that map is applied to all  smooth paths in $\RR^d$:
\begin{equation}
\label{eq:explog}
\mathcal{G}^n (\RR^d)  \,\, = \,\,
\bigl\{ \,\sigma^{\leq n}(X) \,\,: \,\,\, X : [0,1] \rightarrow \RR^d \,\,
\hbox{any smooth path} \,\bigr\} \ .
\end{equation}
\end{theorem}

We revisit the $n=2$ case in
Section \ref{sec:vsm} from the perspective of this theorem.

\begin{example}[$n=2$] \rm
The truncated tensor algebra $T^2 (\RR^d)$ consists of elements
$$ \begin{matrix} P \, \, = \,\,\, \sum_{i=1}^d \sum_{j=1}^d \alpha_{ij} e_{ij} \, + \,
\sum_{i=1}^d \beta_i e_i \, + \, \gamma. \end{matrix} $$
By Lemma \ref{lem:shuffleforms}, $P$  is in ${\rm Lie}^2(\RR^d)$ if and only if $\gamma = 0$
and $\,\alpha_{i \,\shuffle\, j} = \alpha_{ij} + \alpha_{ji} \,=\, 0\,$ for all $i,j$.
Lemma \ref{lem:GroupLike} says that the exponentials of these Lie polynomials are precisely the expressions
$$ \begin{matrix} {\rm exp}(P)\,\,=\,\,\sum_{i=1}^d \sum_{j=1}^d \sigma_{ij} e_{ij} \,+\,
\sum_{i=1}^d \sigma_i e_i \,+\, 1. \end{matrix}
$$
where  $\,\sigma_i \sigma_j\,$ equals  the shuffle linear form $\sigma_{i \,\shuffle \,j} = \sigma_{ij} + \sigma_{ji} $
for $1 \leq i,j \leq d$. These inhomogeneous quadratic equations cut out
the step-$2$ free nilpotent Lie group $\mathcal{G}^2(\RR^d)$. 
By eliminating $\sigma_1,\ldots,\sigma_d$ from these equations, we obtain the
homogeneous ideal $M_{d,d}$ in Theorem \ref{thm:matricesmain}. 
\hfill $\diamondsuit$
\end{example}

\subsection{Gr\"obner Basis for the Free Lie Group}

In what follows we work over an algebraically closed field $\KK$ of characteristic zero.
Our varieties are defined over $\QQ$.
Computations refer to polynomials with
rational coefficients.
We consider the following ideal in the ring $\KK[ \sigma^{(\leq n)}]$
of polynomial functions on $T^n(\KK^d)$:
\begin{equation}
\label{eq:shuffleideal}
G_{d,n} \,\, = \,\,
 \big\langle \,\sigma_I \sigma_J - \sigma_{I \, \shuffle \,J} \,\,:\,\,
\hbox{for all words} \,\,I\,\,{\rm and} \,\,J \,\, 
\hbox{with} \,\,|I| + |J| \leq n
\,\, \big\rangle.
\end{equation}
This ideal is not homogeneous. Its affine variety $\mathcal{G}_{d,n}$ is precisely the 
free Lie group $\mathcal{G}^n(\KK^d)$. The exponential map is a polynomial parametrization
of this variety. The following example illustrates the
associated implicitization problem, whose solution is given  by Lemma \ref{lem:GroupLike}.

\begin{example}[$d=2,n=3$]  \label{ex:P23M2}  \rm
The linear space ${\rm Lie}^3(\KK^2)$  has dimension $5$.  Its elements are 
$$  \quad
\sigma \,= \,
r e_1 + s e_2 + t [e_1,e_2] + u [e_1,[e_1,e_2]] +  v [[e_1,e_2],e_2] ,
\quad \hbox{where} \,\,r,s,t,v,u \in \KK. 
$$
${\rm Lie}^3(\KK^2)$ is the subspace of
$T_0^3(\KK^2)  \simeq \KK^{14}$ defined by nine
shuffle relations like
$$ \sigma_{1 \shuffle 2} = \sigma_{12} + \sigma_{21} \,=\, 0\, ,\,\,\,
\sigma_{1 \shuffle 11} = 3 \sigma_{111}\, =\, 0 \,,\,\, \,
\sigma_{1 \shuffle 12} = 
2 \sigma_{112}+ \sigma_{121}\, =\, 0.
$$
The exponential $ {\rm exp}(\sigma) $
of the Lie polynomial $\sigma$ is the following expression:
$$ \begin{matrix}
  1+  r e_1 + s e_2
+  \frac{r^2}{2} e_{11}
+ \bigl(\frac{rs}{2}+t\bigr) e_{12}
+ \bigl(\frac{rs}{2}-t \bigr) e_{21} + \cdots  \\
+ \bigl(\frac{rs^2}{6} -   2v \bigr)     e_{212}
+ \bigl(\frac{rs^2}{6}- \frac{st}{2}+v \bigr) e_{221}
+  \frac{s^3}{6} e_{222}. 
\end{matrix}
$$
The coefficients of ${\rm exp}(\sigma)$ define the exponential map
from ${\rm Lie}^3(\KK^2) \simeq \KK^5$ into $T_1^3 (\KK^2) \simeq \KK^{14}$.
Its image is the $5$-dimensional variety $\mathcal{G}_{2,3}$. We compute 
its ideal $G_{2,3}$ using the computer algebra package {\tt Macaulay2} as follows:
\begin{small}
\begin{verbatim}
R = QQ[s11,s21,s22,s111,s121,s211,s212,s221,s222,s1,s2,s12,s112,s122,
        MonomialOrder=>Lex];
ExponentialMap = map( QQ[r,s,t,u,v] , R , 
{r^2/2,r*s/2-t,s^2/2,r^3/6,r^2*s/6-2*u,r^2*s/6-r*t/2+u,r*s^2/6-2*v,
r*s^2/6-s*t/2+v, s^3/6, r,s,r*s/2+t, r^2*s/6+r*t/2+u, r*s^2/6+s*t/2+v});
G23 = kernel ExponentialMap;  
gens gb G23
G23 == ideal( s1^2-2*s11, s1*s2-s12-s21, s2^2-2*s22, s1*s11-3*s111, 
s1*s12-2*s112-s121, s1*s21-s121-2*s211, s1*s22-s122-s212-s221, 
s2*s11-s121-s211-s112,  s2*s12-2*s122-s212, s2*s21-2*s221-s212,
s2*s22-3*s222)
\end{verbatim}
\end{small}
This computes the lexicographic Gr\"obner basis with nine elements 
 to be seen in Theorem~\ref{thm:lyndonmain}.
The last command uses {\tt ==} to verify that $G_{2,3}$ is generated by the 
quadratic relations
 $\sigma_I \sigma_J - \sigma_{I \shuffle J}$.
\hfill $ \diamondsuit $
\end{example}

We are interested in the structure of  the free Lie group $ \mathcal{G}_{d,n}$
as an affine algebraic variety. To this end, we need to first record some 
combinatorial facts about free Lie algebras.

A word $I$ on the alphabet $\{1,2,\ldots,d\}$ is a {\em Lyndon word}  if it
is strictly smaller in lexicographic order than all of its rotations.
Lyndon words are the Hall words for a particular Hall set 
(cf.~\cite[Chapter 4]{Reu}). Since they are easy to define,
and seen widely in the combinatorics literature,
we will use Lyndon words in what follows.  With any Lyndon word
$I$ one associates an iterated Lie bracketing $b(I)$ by induction on
$k = {\rm length}(I)$.  If $k = 1$ and $I = i$ then $b(I) = e_i$.
If $k \geq 2$ then it is $b(I) = [ b(I_1), b(I_2)]$ where
$I = I_1 I_2$ and  $I_2$ is the longest Lyndon word
appearing as a proper right factor of $I$. Some authors refer to the bracketing image $b(I)$ of a Lyndon word $I$ as {\it Lyndon bracket}. We point to  \cite{LRam}.
That article also features connections to Gr\"obner bases.
We use the following about Lyndon words.

\begin{proposition} \label{prop:lyndon}
The bracketings $b(I)$ of Lyndon words $I$ of length $\leq n$  form a basis
for ${\rm Lie}^n(\KK^d)$. 
The dimension of this Lie algebra, which is the number of Lyndon words, equals
$$ \qquad \qquad \quad
\lambda_{d,n} \,\, = \,\,\,  \sum_{k=1}^n  \,\sum_{\ell \, {\rm divides} \,k} \! \frac{\mu(\ell)}{k} d^{k/\ell}, \qquad
\hbox{where $\mu$ is the M\"obius function.}
$$
\end{proposition}

\begin{proof} The inner sum is the number of Lyndon words of length  $k$. See
 \cite[Corollary 4.14]{Reu}. \end{proof}
 
 \begin{example} \rm There are   $\lambda_{2,2} = 3$  Lyndon words of length 
 $n \leq 2$ and $d=2$. Their bracketings are
   $e_1 = b(1), e_2 = b(2) , e_{12}-e_{21} = b(12)$. These three
   elements form a  basis of ${\rm Lie}^2(\KK^2)$.
   \hfill $\diamondsuit$
 \end{example}

\begin{example}[$d=2, k=4$] \rm
The three Lyndon words have the bracketings
\begin{equation}
\label{eq:lyndonwords}
  \begin{matrix}
 b(1112) & = & [1,[1,[1,2]]] & = &  e_{1112} - 3 e_{1121} + 3e_{1211} - e_{2111}, \\
 b(1122) & = & [1,[[1,2],2]] &  = & e_{1122} - 2 e_{1212} + 2 e_{2121} - e_{2211}, \\
 b(1222) & = & [[[1,2],2],2] & = & e_{1222} - 3 e_{2122} + 3 e_{2212} - e_{2221} .
 \end{matrix}
 \end{equation}
The shuffle forms $  \, \sigma_{1 \,\shuffle \,111} , \,
\sigma_{11 \,\shuffle \,11} \,$ and $ \,\sigma_{12 \,\shuffle \,21}\,$
in (\ref{eq:shuffleforms}) vanish on the span of (\ref{eq:lyndonwords}), as seen in Lemma \ref{lem:shuffleforms}.
Hence the subspace of shuffle forms in $\KK[\sigma^{(4)}]$ has dimension $13  = 16-3$.
\hfill $\diamondsuit $
\end{example}

We now present our main result in this subsection. It shows that the ideal $G_{d,n}$ is prime.
We fix a {\em lexicographic term order} in the polynomial ring $\KK[\sigma^{(\leq n)}]$.  
The underlying  variable ordering is assumed to satisfy the following requirement:
we have $\sigma_I  > \sigma_J$ if ${\rm length}(I) > {\rm length}(J)$, or if 
$J$ is obtained from $I$ by rotation and $J$ is lexicographically smaller than~$I$.
Consider any word $I$ that is not a Lyndon word. In the following theorem,
the polynomial $\phi_I$ is the unique expression of $\sigma_I$ 
on $ \mathcal{G}_{d,n} $ in terms of
the unknowns $\sigma_J$ that are indexed by the Lyndon words~$J$.

\begin{theorem} \label{thm:lyndonmain}
The ideal $\,G_{d,n}$ in (\ref{eq:shuffleideal}) is prime. Its
irreducible variety  $\,\mathcal{G}_{d,n}$ has dimension $\lambda_{d,n}$ in $T^n(\KK^d)$. 
The reduced Gr\"obner basis of $G_{d,n}$ consists of 
the polynomials $\,\sigma_{I} - \phi_I(\sigma_{\rm lyndon}) \,$
where $I$ runs over non-Lyndon words  of length $\leq n$.
\end{theorem}

Example  \ref{ex:P23M2} illustrates Theorem \ref{thm:lyndonmain} for $d=2, n=3$.
The vector of Lyndon variables equals
$\sigma_{\rm lyndon} = (\sigma_1,\sigma_2,\sigma_{12}, \sigma_{112}, \sigma_{122})$.
Its length is  $\lambda_{2,3} = 5$. The last line of the {\tt Macaulay2} code is
the right hand side of (\ref{eq:shuffleideal}). The output of the command {\tt gens gb G23}
is the reduced Gr\"obner basis in Theorem \ref{thm:lyndonmain}.

\begin{proof}
By Lemma \ref{lem:GroupLike}, the Lie group $\mathcal{G}_{d,n}$
is the zero set of the polynomials $\sigma_I \sigma_J - \sigma_{I \, \shuffle \,J}$. 
This affine variety   is irreducible because it is
the image of  the  linear space ${\rm Lie}^n (\KK^d)$ 
under the polynomial map ${\rm exp}$. This map has
a polynomial inverse, namely ${\rm log}$. Hence the dimension
of $\mathcal{G}_{d,n}$ agrees with that of ${\rm Lie}^n(\KK^d)$.
The latter dimension is $\lambda_{d,n}$, by Proposition \ref{prop:lyndon}.

 Let $\sigma_{\rm lyndon}$ denote the vector of all variables 
 $\sigma_J$ that are indexed by Lyndon words $J$.
 We claim that, for every non-Lyndon word $I$, there is a polynomial
$\phi_I = \phi_I( \sigma_{\rm lyndon})$ in the Lyndon variables $\sigma_J$ such that 
the difference $\,\sigma_I - \phi_I( \sigma_{\rm lyndon}) \,$ lies in the ideal $G_{d,n}$.
This follows from  a theorem of Radford \cite{Ra}. We shall prove  it by induction on
lexicographic~order.

We construct the polynomial $\phi_I$ with the technique used by Melan\c con and Reutenauer
in \cite[\S 4]{MR}. The non-Lyndon word $I$ has a unique factorization 
into Lyndon words $I_1 I_2 \cdots I_\kappa$. 
Recall that shuffle multiplication is associative and commutative.
This ensures that the expression  $\sigma_{I_1 \shuffle I_2 \shuffle \cdots \shuffle I_\kappa} -
\sigma_{I_1} \sigma_{I_2} \cdots \sigma_{I_\kappa} $ lies in the ideal $G_{d,n}$.
Its highest term equals $\sigma_I$ times a positive integer.
We divide the expression by that integer and then subtract it from $\sigma_I$.
Each of the non-Lyndon variables $\sigma_{I'}$ seen in the resulting polynomial
has either $|I'| < |I|$ or $| I' | = |I|$ and $I'$ comes before $I$.
Iterating this process, we are done by induction.

The prime ideal of $\mathcal{G}_{d,n}$ contains $G_{d,n}$
which in turn contains the ideal generated by the expressions $\,\sigma_{I} - \phi_I(\sigma_{\rm lyndon}) \,$
for $I$ non-Lyndon. The first ideal is prime by definition. The third ideal is prime because its
initial monomial ideal is prime. This holds as
it is generated by the  variables $\sigma_I$ for $I$ non-Lyndon.
Moreover, the two prime ideals have the same dimension, namely $\lambda_{d,n}$.
This implies that all three ideals in our chain of inclusions are equal.

The argument also proves Radford's result that the Lyndon variables $\sigma_J$ are
algebraically independent modulo $G_{d,n}$, so the polynomials
$\phi_I$ seen in the Gr\"obner basis are unique.
\end{proof}

\subsection{The Universal Variety}

In this paper we study signature tensors of a fixed order $k$.
The elements in the Lie group $\mathcal{G}_{d,n}$ record these tensors
simultaneously for all values of $k$ between $1$ and $n$. We must
extract the homogeneous component of degree $k$, by projecting
 to the $k$th summand in (\ref{eq:Tn}). The algebraic
counterpart to the geometric operation of projection is the elimination of variables.

The {\em universal ideal} for order $k$ tensors of format $d \times d \times 
\cdots \times d $ is defined as 
\begin{equation}
\label{eq:Idk}
U_{d,k} \,\, := \,\, G_{d,k} \,\cap \, \KK[ \sigma^{(k)}]. 
\end{equation}
This is a homogeneous prime ideal in a polynomial ring in $d^k$ variables.
We define the {\em universal variety} $\,\mathcal{U}_{d,k}$ to be the zero set of
the ideal $U_{d,k}$ in the projective space $\PP^{d^k-1}$.
As in earlier sections, the field $\KK$ is algebraically closed of characteristic zero.
In applications to paths in $\RR^d$, we would take $\KK = \CC$, but we could
also consider paths in $\KK^d$.

\begin{corollary} If $\,\RR \subset \KK$ then
the universal variety $\,\mathcal{U}_{d,k}$ is the smallest projective variety in $\,\PP^{d^k-1}\,$
that contains the $k$th signature tensors $\,\sigma^{(k)}(X)$ of all smooth paths $X$ in $\,\RR^d$.
\end{corollary}

\begin{proof}
This follows immediately from the Chen-Chow Theorem~\ref{thm:chen}.
\end{proof}

The universal variety $\mathcal{U}_{d,k}$ is the ambient space
for the varieties of signature tensors in Section \ref{sec:pwlp} below.
We saw this already  in the matrix case in Section~\ref{sec:vsm}.

\begin{example}[$k=2$]       \rm
The universal variety $\mathcal{U}_{d,2}$ is the
variety $\mathcal{M}_{d,d}$ in  (\ref{eq:Mchain}).
Its points are the $d \times d$ matrices whose symmetrization has rank $1$.
It is universal in the sense that, for all $m \in \mathbb{N}$,
it contains the signature images  $\mathcal{L}^{\rm im}_{d,2,m}$ and
$\mathcal{P}^{\rm im}_{d,2,m}$ and their varieties  $\mathcal{M}_{d,m}$.
\hfill $\diamondsuit$
\end{example}

The varieties $\mathcal{U}_{d,k} $ and its subvarieties in the
next section are invariant under the action of $G = {\rm GL}(d,\KK)$
on tensors. In describing their ideals, it is sometimes convenient to perform a change of
basis that realizes the decomposition of
$(\KK^d)^{\otimes k}$ into irreducible $G$-modules.

\begin{example}[$d=2,k=3$]  \label{ex:twothree} \rm
The universal variety $\mathcal{U}_{2,3}$  lives in the space $\PP^7$
of $2 \times 2 \times 2$ tensors. We compute its ideal $U_{2,3}$
with the following two commands after Example \ref{ex:P23M2}.
The output consists of six quadrics:
\begin{verbatim}
U23 = eliminate({s1,s2,s11,s12,s21,s22},G23)
codim U23, degree U23, betti mingens U23
\end{verbatim}
To understand the output, we perform the change of coordinates
\begin{equation}
\label{eq:CoB}
\begin{small}
 \begin{matrix}
\sigma_{111} = 6 \alpha_1 & \quad
\sigma_{112} = 2 \alpha_2-\beta_1 & \quad
\sigma_{121} = 2 \alpha_2-\gamma_1 & \quad
\sigma_{211} = 2 \alpha_2+\beta_1+\gamma_1 \\
\sigma_{222} = 6 \alpha_4 & \quad
\sigma_{221} = 2 \alpha_3-\beta_2 & \quad
\sigma_{212} = 2 \alpha_3-\gamma_2 & \quad
\sigma_{122} = 2 \alpha_3+\beta_2+\gamma_2 
\end{matrix}
\end{small}
\end{equation}
Here $\alpha_i,\beta_j, \gamma_k$ are coordinates on  three irreducible $G$-modules 
$S_{3}(\KK^2)$, $S_{21}(\KK^2)$, $S_{21}(\KK^2)$ in $(\KK^2)^{\otimes 3} $.
These have dimensions $4,2,2$, by the Hook Length Formula \cite[\S 2.8]{Lan}.
The new coordinates reveal that the ideal $U_{2,3}$ is 
generated by the $2 \times 2$ minors of the $2 \times 4$ matrix
\begin{equation}
\label{eq:24matrix}
\begin{pmatrix}
    \,  3 \alpha_1 & \alpha_2 &  \alpha_3 &\,  \,\,2 \beta_1 + \gamma_1  \\
\,      \alpha_2 & \alpha_3 & 3 \alpha_4 & -2 \beta_2 - \gamma_2  \,
\end{pmatrix}.
\end{equation}
The universal variety $\mathcal{U}_{2,3}$ has dimension $4$ and degree $4$ in $\PP^7$.
\hfill $\diamondsuit $
\end{example}

\begin{example}[$d=k=3$]  \label{ex:1424} \rm
The universal variety $\mathcal{U}_{3,3}$ of $3 {\times} 3 {\times} 3$
signature tensors has dimension $13$ and degree $24$ in $\PP^{26}$.
Using {\tt Macaulay2}, we find that
its homogeneous prime ideal $U_{3,3}$ is minimally generated by $81$ quadrics.
To exhibit these quadrics, we first decompose the ambient tensor space into four irreducible
$G$-modules:
$$  ( \KK^3)^{\otimes 3} \,  \simeq \,
{\rm S}_3(\KK^3) \,\oplus \,
{\rm S}_{21}(\KK^3) \, \oplus \,
{\rm S}_{21}(\KK^3) \, \oplus \,
{\rm S}_{111}(\KK^3) \,\,\,\, {\rm or} \,\,\,\,
\KK^{27}_s \, \simeq \, \KK^{10}_a \oplus \KK^8_b \oplus \KK^8_c \oplus \KK^1_d.
$$
 The first and last component represent
symmetric and skew-symmetric tensors respectively.
The following linear change of coordinates makes the isomorphisms
above explicit:
$$ \begin{small}
\begin{matrix}
a_{111} = s_{111}, \,\,
a_{112} = s_{112} {+} s_{121} {+} s_{211},  \,\,
a_{113} = s_{113} {+} s_{131} {+} s_{311}, \,
a_{122} = s_{122} {+} s_{212} {+} s_{221}, \\
a_{123} = s_{123} {+} s_{132} {+} s_{213} {+} s_{231} {+} s_{312} {+} s_{321}, \quad
a_{133} = s_{133} {+} s_{313} {+} s_{331}, \quad
a_{222} = s_{222},   \\
a_{223} = s_{223} {+} s_{232} {+} s_{322},  \,
a_{233} = s_{233} {+} s_{323} {+} s_{332}, \,
a_{333} = s_{333},   \,
b_{121} =  {-} 4 s_{112} {+} 2 s_{121} {+} 2 s_{211},  \\
b_{122} = 2 s_{122} {+} 2 s_{212} {-} 4 s_{221},  \,
b_{123} = 2 s_{123} {+} 2 s_{213} {-} 2 s_{231} {-} 2 s_{321},  \,
b_{131} =  {-} 4 s_{113} {+} 2 s_{131} {+} 2 s_{311},   \\
b_{132} = 2 s_{132} {-} 2 s_{231} {+} 2 s_{312} {-} 2 s_{321}, \,
b_{133} = 2 s_{133} {+} 2 s_{313} {-} 4 s_{331},   \,
b_{232} =  {-} 4 s_{223} {+} 2 s_{232} {+} 2 s_{322},  \\
b_{233} = 2 s_{233} {+} 2 s_{323} {-} 4 s_{332},  \quad
c_{112} = 2 s_{112} {-} 4 s_{121} {+} 2 s_{211}, \quad
c_{113} = 2 s_{113} {-} 4 s_{131} {+} 2 s_{311},   \\
c_{122} = 2 s_{122} {-} 4 s_{212} {+} 2 s_{221},  \,
c_{123} = 2 s_{123} {-} 2 s_{213} {-} 2 s_{312} {+} 2 s_{321},  \,
c_{132} = 2 s_{132} {-} 2 s_{213} {+} 2 s_{231} {-} 2 s_{312},  \\
c_{133} = 2 s_{133} {-} 4 s_{313} {+} 2 s_{331},  \quad
c_{223} = 2 s_{223} {-} 4 s_{232} {+} 2 s_{322}, \quad
c_{233} = 2 s_{233} {-} 4 s_{323} {+} 2 s_{332},   \\
d_{123} = s_{123} {-} s_{132} {-} s_{213} {+} s_{231} {+} s_{312} {-} s_{321}
\end{matrix}
\end{small}
$$
The ideal $U_{3,3}$ of the universal variety is generated by a space of
$6 \cdot 2 + 3 \cdot 2 + 3 \cdot 3 + 6 \cdot 7 + 12 = 81 $ quadrics.
Its $\ZZ^3$-grading has $19$ components
that come in five symmetry classes:
$$
\begin{small}
\begin{matrix}
{\rm six} &
(4, 2, 0) & 3 a_{111} a_{122}-a_{112}^2 \, , &  
6 a_{111} b_{122}+3 a_{111} c_{122}+2 a_{112} b_{121}+a_{112} c_{112}, \smallskip \\
{\rm three} & (3,3,0) &
9 a_{111} a_{222}-a_{112} a_{122}\,, & 2 a_{112} b_{122} + a_{112} c_{122} 
+ 2 a_{122} b_{121} + a_{122} c_{112} , \smallskip \\
{\rm three}\! &
(4, 1, 1) & 3 a_{111} a_{123} - 2 a_{112} a_{113}  , &  6 a_{111} 
b_{123}{+}3 a_{111}c_{132}{-}3 a_{111} d_{123}{+}2 a_{112} b_{131}{+}a_{112} c_{113}, \\
& & & 
 6 a_{111} b_{132}{+}3 a_{111} c_{123}{+}3 a_{111} d_{123}
{+}2 a_{113} b_{121}{+}a_{113} c_{112} .
\end{matrix}
\end{small}
$$
\vspace{-0.05in}
$$
\begin{small}
\begin{matrix}
{\rm six} \,\,
(3, 2, 1) \,\,\,
6 a_{111} a_{223}-a_{112} a_{123}   ,\,
2 a_{112} b_{132}+a_{112} c_{123}+a_{112} d_{123}-2 a_{113} b_{122}
-a_{113} c_{122},  \ldots \, {\rm etc.}  \\
\end{matrix}
\end{small}
$$
There are seven quadrics in degree $(3,2,1)$. In addition, we have twelve quadrics
in the central degree $(2,2,2)$, like
$\,4 a_{112} b_{233}+2 a_{112} c_{233}-4 a_{122} b_{133}-2 a_{122} c_{133}-3 a_{123} d_{123}-4 a_{233} b_{121}-2 a_{233} c_{112} $. \hfill $ \diamondsuit $
\end{example}

\begin{table}[h]
\begin{center} \begin{tabular}{ | l | l | l | l | l | l | l | l | p{1.5cm} |} \hline  $d$ & $k$  & $a$ &$\dim$ & $\deg$ & gens \\ 
\hline 2 & 3 & 7  & 4 & 4 & 6 \\
\hline 2 & 4 & 15 & 7 & 12 & 33 \\
\hline 2 & 5  & 31  & 13 & 40  & 150  \\
\hline 3 & 3  & 26 & 13 & 24 & 81  \\
\hline 3 & 4 & 80 & 31 & ? & 954 \\
\hline 4 & 3 & 63 & 29 & 200 & 486\\
\hline 
 \end{tabular} 
\vspace{-0.11in}
\end{center}
  \caption{\label{tab:Udk}  Invariants of the ideal $U_{d,k}$ that defines the universal variety $\mathcal{U}_{d,k}$} \medskip
  \end{table}

We computed the ideals $U_{d,k}$ for all values of $d$ and $k$ with $d +k \leq 7$.
Since $\mathcal{U}_{d,2}$ equals the matrix signature variety $\mathcal{M}_{d,d}$ 
(cf. Table \ref{tab:Mdm}), we only consider $k\geq 3$.  The results are listed in Table \ref{tab:Udk}.
In each case we computed, we found that $U_{d,2}$ is generated by quadrics.
The last column gives the number of generators.
The second-to-last column reports the degree of the variety $\mathcal{U}_{d,k} \subset \PP^{d^k-1}$.

 In the first version of this paper we asked 
whether $U_{d,k}$ is always generated by quadrics.
This question has since been answered, to the negative, by Francesco Galuppi.
In Section \ref{sec:DimIde} we shall prove that the universal variety $\mathcal{U}_{d,k}$ has the expected dimension
$\lambda_{d,k}-1$.

\section{Piecewise Linear Paths and Polynomial Paths} \label{sec:pwlp}

This section is the heart of this paper. We introduce, study and relate the
signature varieties of two natural families of paths. These all live in the
universal varieties seen in Section \ref{sec:universal}.

\subsection{Polynomial Maps into Tensor Space}

We now study paths  $X : [0,1] \rightarrow \RR^d$
whose coordinates are polynomials of degree $m$
or  piecewise linear with $m$ pieces. Each of these is
represented by a real $d \times m$ matrix, also denoted by $X = (x_{ij})$.
With this convention,  a polynomial path has coordinate functions
\begin{equation}
\label{eq:path1}
 X_i(t) \,\, = \,\,  x_{i1} t + x_{i2} t^2 + x_{i3} t^3 + \cdots +  x_{im} t^m.
 \end{equation}
 The differential $1$-forms seen in the iterated integrals (\ref{eq:iteratedint}) are
\begin{equation}
\label{eq:diffform}
 {\rm d} X_{ i}(t) \,= \, X'_{ i}(t) {\rm d} t  \,=\,
\bigl( x_{ i1} + 2 x_{ i2} t  + 3 x_{ i3} t^2 + \cdots + m x_{ im} t^{m-1} \bigr) 
{\rm d} t .
\end{equation}
Each coordinate $\sigma_{i_1 i_2 \cdots i_k}$ of the tensor $\sigma^{(k)}(X)$ 
is a homogeneous polynomial of degree $k$ in the $dm$ unknowns $x_{ij}$
with coefficients in $\QQ$. Formulas for $d=2,k=3$ are shown in
Example~\ref{ex:d2k3}.
The  $d \times d \times \cdots \times d$
 tensor $\sigma^{(k)}(X)$ can    be computed from the
 $m \times m \times \cdots \times m$ 
 tensor $\sigma^{k}(C_{\rm mono})$ in Example~\ref{ex:cmp}
by multiplying each of its $k$ sides with the $d \times m$ matrix $X$.
This is the tensor analogue to the congruence action on matrix space seen
in (\ref{eq:XSXmono}). 

The $x_{ij}$ are homogeneous coordinates on the projective space $\PP^{dm-1}$ 
over an algebraically closed field $\KK$ that contains $\RR$.
The matrix-tensor multiplication described above defines a rational map of degree~$k$:
\begin{equation}
\label{eq:polymap}
\sigma^{(k)} \,: \,\,\PP^{dm-1} \dashrightarrow \PP^{d^k-1}\,,\,\,
X \mapsto \sigma^{(k)}(X). 
\end{equation}
The Zariski closure of the image of this map is the {\em polynomial signature variety}   $\mathcal{P}_{d,k,m}$. 
 The homogeneous prime ideal $P_{d,k,m}$ of this variety in $\KK[\sigma^{(k)}]$ 
is  the {\em polynomial signature~ideal}.

\begin{example}[$d=k=3, \,m=2$] \label{ex:itisbest} \rm
The third signature variety $\mathcal{P}_{3,3,2}$ for quadratic paths in $3$-space
lies in the space of $3 {\times} 3 {\times} 3$ tensors. Its linear span is the 
hyperplane $\PP^{25}$ defined~by
\begin{equation}
\label{eq:alternating}
\sigma_{123} - \sigma_{132}- \sigma_{213}+ \sigma_{231}
+\sigma_{312}-\sigma_{321} \,\, = \,\, 0. 
\end{equation}
It is best to write the $162$ quadrics in its ideal $P_{3,3,2}$ as in Example~\ref{ex:1424}.
\hfill $\diamondsuit $
\end{example}

Piecewise linear paths are also represented by $d \times m$ matrices $X$.
Their steps are the column vectors $\,X_1,\ldots,X_m \in \RR^d$.
To be explicit, our path has the following parametrization:
 $$ \begin{matrix}  t \,\mapsto \, X_1+\cdots+X_{i-1} \,+\, (mt-i+1)\cdot X_i\,\,
 \,\hbox{where}\,\,\frac{i{-}1}{m} \leq  t  \leq   \frac{i}{m}\,\,\,
\hbox{and}\, \,\, i=1,2,\ldots,m. \end{matrix}
$$
The  tensor $\sigma^{(k)}(X)$ is obtained from the ``upper triangular''
 $m \times m \times \cdots \times m$ 
 tensor $\sigma^{k}(C_{\rm axis})$ in Example~\ref{ex:cap}
by multiplying each of its $k$ sides with the matrix $X$.
This defines a rational map (\ref{eq:polymap}) of degree~$k$.
  The closure of its image 
 is the {\em piecewise linear signature variety}   $\mathcal{L}_{d,k,m}$.
 Its homogeneous prime ideal $L_{d,k,m}$ in $\KK[\sigma^{(k)}]$ is called the
   {\em piecewise linear signature ideal}.

\smallskip 

Let us reconcile these definitions with those for $k=2$ in Subsection \ref{sec:rsm},
 by viewing them through the lens of  Subsection \ref{sec:manylives}.
We are ultimately interested in 
\begin{equation}
\label{eq:PLPL}
  \begin{matrix}
 \Pim_{d,k,m}  \,\,:= \,\, \{ \,\,\sigma^{(k)}(X) \ : \ X : [0,1] \to \RR^d \ \text{ polynomial path of degree $\le m$}\, \},  \\
  \Lim_{d,k,m}\,  :=  \, \{\, \sigma^{(k)}(X) \ : \ X : [0,1] \to \RR^d \ \text{piecewise linear with $m$ segments}   \} . 
  \end{matrix}
  \end{equation}
  These signature images are semialgebraic subsets of $(\RR^{d})^{\otimes k}$. 
In this section we study the polynomials that vanish on these sets.
They are recorded in the ideals $P_{d,k.m}$ and $L_{d,k,m}$.
Equivalently, we examine the tightest outer approximations
of (\ref{eq:PLPL}) by algebraic varieties.
Finding inequalities for
$ \Pim_{d,k,m}$  inside $  \mathcal{P}^\RR_{d,k,m} $,
and for $ \Lim_{d,k,m}$  inside $  \mathcal{L}^\RR_{d,k,m} $,
is left to future research.

\begin{remark}  \label{rem:P=L} \rm 
If $m \leq d $ then $\mathcal{P}_{d,k,m}$
and $\mathcal{L}_{d,k,m}$ are closures of ${\rm GL}(d,\KK)$-orbits. \vspace{-0.1in}
\begin{enumerate}
\item[(a)] If $m=1$ then $X$ is a linear path and
$\mathcal{L}_{d,k,1}$ is the {\em Veronese variety},
whose points are symmetric tensors of rank $1$.
In general, $m$ plays a role similar to that of {\em tensor rank} in
multilinear algebra \cite{Lan}.
 \vspace{-0.1in}
\item[(b)]  The varieties $\mathcal{L}_{d,2,m} = \mathcal{P}_{d,2,m} = \mathcal{M}_{d,m}$
were determined in  Theorem \ref{thm:matricesmain}.
\end{enumerate}
\end{remark}

Let $X$ be the piecewise linear  path with steps $X_1,X_2,\ldots,X_m$ in $\RR^d$.
Chen \cite{Chen54, Chen57} showed that the $n$-step signature of the path $X$ is given by the
tensor product of tensor exponentials:
\begin{equation}
\label{eq:PL1}
 \sigma^{\leq n}(X) \,\,\, = \,\,\,
{\rm exp}(X_1) \,\otimes \,{\rm exp}(X_2) \,\otimes \,\cdots\, \otimes\, {\rm exp}(X_m)
 \quad \in \,\,T^n(\RR^d). 
 \end{equation}
  Hence the $k$th signature tensor of $X$ is the following element
 of $(\RR^d)^{\otimes k}$ or $\PP^{d^k-1}$:
 \begin{equation}
 \label{eq:PL2}
  \sigma^{(k)}(X) \,\, = \,\,\,\hbox{the sorted expansion of} \,\,\,
\frac{1}{k!} (X_1 + X_2 + X_3 + \cdots + X_m)^{\otimes k} . 
\end{equation}
Here, by ``sorted expansion'' we mean that every rank one summand
$\,X_{i_1} \otimes X_{i_2} \otimes \cdots \otimes X_{i_k} \,$ is to be replaced by the 
corresponding rank one summand where the $k$ indices are sorted.
This replacement is done after the expansion of 
the $m^k$ terms and prior to summing them.

\begin{corollary}
The $k$th signature tensor of a piecewise polynomial path equals
\begin{equation}
\label{eq:PL3}
\sigma^{(k)}(X) \,\, = \,\,\,
\sum_\tau\, \prod_{\ell=1}^m \frac{1}{ |\,\tau^{-1}(\ell)|\, !} \cdot
X_{\tau(1)} \otimes X_{\tau(2)} \otimes X_{\tau(3)} \otimes \cdots \otimes X_{\tau(k)}.
\end{equation}
The sum is over all weakly increasing functions
$\,\tau: \{1,2,\ldots,k\} \rightarrow \{1,2,\ldots,m\}$.
\end{corollary}

\begin{example} \rm \label{ex:hereareforumulas}
The third signature $(k=3)$ of a piecewise linear path $X$ equals
$$ \sigma^{(3)}(X) \,\, = \,\,\,
\frac{1}{6} \cdot \sum_{i=1}^m X_i^{\otimes 3}\,\,\,+\,\,
\frac{1}{2} \cdot \!\! \sum_{1 \leq i < j \leq m}\!\! \! \bigl(X_i^{\otimes 2} \otimes X_j \,+
\,X_i \otimes X_j^{\otimes 2} \bigr) \,\,\, + \! \sum_{1 \leq i < j < l \leq m} \!\!\!\!\!
X_i \otimes X_j \otimes X_l .
$$
The fourth signature $(k=4)$ of a two-step path  $(m=2)$ 
is the $d {\times} d {\times} d {\times} d$ tensor
\begin{equation}
\label{eq:dddd}
 \sigma^{(4)}(X) \,\,= \,\,\,
\frac{1}{24} \cdot \biggl[\, X_1^{\otimes 4} \,+ \,
 4  \,X_1^{\otimes 3} \otimes X_2 \,+ \,
  6 \,X_1^{\otimes 2} \otimes X_2^{\otimes 2} \,+ \,
   4 \,X_1  \otimes X_2^{\otimes 3} \,+ \,
   X_2^{\otimes 4} \, \biggr].
\end{equation}
The projective variety $\mathcal{L}_{d,4,2} $ parametrizes
tensors in $\PP^{d^4-1}$ of this special form.
\hfill $\diamondsuit $
\end{example}

\begin{example}[$d=m=2, k=4$] \label{eq:zweivierzwei} \rm
Consider paths consisting of two segments
 $X_1 = ({\tt a},{\tt b})$ and $X_2 = ({\tt A}, {\tt B})$.
The following  {\tt Macaulay2} code realizes the equation (\ref{eq:dddd})
and it computes the ideal $L_{2,4,2}$:
\begin{small}
  \begin{verbatim}
R = QQ[s1111,s1112,s1121,s1122,s1211,s1212,s1221,s1222,
       s2111,s2112,s2121,s2122,s2211,s2212,s2221,s2222];
S = QQ[ a,b, A,B ];
f = map(S,R,{
   a*a*a*a + 4*a*a*a*A + 6*a*a*A*A + 4*a*A*A*A + A*A*A*A,
   a*a*a*b + 4*a*a*a*B + 6*a*a*A*B + 4*a*A*A*B + A*A*A*B,
   a*a*b*a + 4*a*a*b*A + 6*a*a*B*A + 4*a*A*B*A + A*A*B*A,
   a*a*b*b + 4*a*a*b*B + 6*a*a*B*B + 4*a*A*B*B + A*A*B*B,
   a*b*a*a + 4*a*b*a*A + 6*a*b*A*A + 4*a*B*A*A + A*B*A*A,
   a*b*a*b + 4*a*b*a*B + 6*a*b*A*B + 4*a*B*A*B + A*B*A*B,
   a*b*b*a + 4*a*b*b*A + 6*a*b*B*A + 4*a*B*B*A + A*B*B*A,
   a*b*b*b + 4*a*b*b*B + 6*a*b*B*B + 4*a*B*B*B + A*B*B*B,
   b*a*a*a + 4*b*a*a*A + 6*b*a*A*A + 4*b*A*A*A + B*A*A*A,
   b*a*a*b + 4*b*a*a*B + 6*b*a*A*B + 4*b*A*A*B + B*A*A*B,
   b*a*b*a + 4*b*a*b*A + 6*b*a*B*A + 4*b*A*B*A + B*A*B*A,
   b*a*b*b + 4*b*a*b*B + 6*b*a*B*B + 4*b*A*B*B + B*A*B*B,
   b*b*a*a + 4*b*b*a*A + 6*b*b*A*A + 4*b*B*A*A + B*B*A*A,
   b*b*a*b + 4*b*b*a*B + 6*b*b*A*B + 4*b*B*A*B + B*B*A*B,
   b*b*b*a + 4*b*b*b*A + 6*b*b*B*A + 4*b*B*B*A + B*B*B*A,
   b*b*b*b + 4*b*b*b*B + 6*b*b*B*B + 4*b*B*B*B + B*B*B*B});
P = kernel f;
toString mingens P
dim P, degree P, betti mingens P
\end{verbatim}
\end{small}
The output produced by this code reveals that the  variety 
$\mathcal{L}_{2,4,2}$ is a threefold
of degree $24$ in a hyperplane $\PP^{14}$ inside
the space $\PP^{15}$ of  $2 \times 2 \times 2 \times 2$ tensors.
The ideal $L_{2,4,2}$ has $55$ quadratic minimal generators.
\hfill $\diamondsuit $
\end{example}

\subsection{Inclusions and Separating Invariants}

Our aim in this subsection is to compare 
polynomial paths and piecewise linear paths. 
Some intuition for our choice of these two families
is offered in Remark \ref{rem:kontsevich}. Of course,
numerous other families would also be interesting,
including piecewise-quadratic paths, trigonometric paths, etc.
Polynomial paths and piecewise linear paths have the
same signature tensors for all cases in Remark \ref{rem:P=L}.
Our next result states that this also holds when $m$ is large:

\begin{theorem} \label{thm:chains}
We have the following chains of inclusions between the $k$th Veronese variety
and the $k$th universal  variety. Here $M$ and $M'$ are constants
that depend only on $d$ and $k$:
$$
\begin{matrix}
\nu_k(\PP^{d-1}) =   \mathcal{L}_{d,k,1} \subset  \mathcal{L}_{d,k,2} \subset  \mathcal{L}_{d,k,3\,} 
\subset \, \cdots\, \subset\, \mathcal{L}_{d,k,M-1}
\subset\, \mathcal{L}_{d,k,M}  \,=\, \mathcal{U}_{d,k} \,\subset \,
\PP^{d^k-1},
\\ \nu_k(\PP^{d-1}) = 
 \mathcal{P}_{d,k,1} \subset \mathcal{P}_{d,k,2} \subset  \mathcal{P}_{d,k,3\,} 
\subset \, \cdots\, \subset\, \mathcal{P}_{d,k,M'-1}
\subset\, \mathcal{P}_{d,k,M'}  \,=\, \mathcal{U}_{d,k} \,\subset \,
\PP^{d^k-1}. 
\end{matrix}
$$
\end{theorem}

\begin{proof}
Chow's Theorem \cite[Theorem 7.28]{FV}  says that the
elements~in the step-$n$ free Lie group $\mathcal{G}^n(\RR^d)$
are precisely the signatures (\ref{eq:PL1})  of piecewise linear paths.
The projection of $\mathcal{G}^n(\RR^d)$ into the $k$th factor of (\ref{eq:Tn})
is the universal signature image $\mathcal{U}_{d,k}^{\rm im}$. We conclude 
$$ \mathcal{U}_{d,k}^{\rm im} \,\, = \,\, \bigcup_{m=1}^\infty \mathcal{L}_{d,k,m}^{\rm im} . $$
By passing to Zariski closures, we obtain the same result for the projective varieties, namely
$\, \mathcal{U}_{d,k} =  \bigcup_{m=1}^\infty \mathcal{L}_{d,k,m} $.
The right hand side is a nested family of irreducible varieties, all contained in
$\mathcal{U}_{d,k} \subset \PP^{d^k-1}$. The number of distinct varieties
in such a nested family is certainly bounded above by $d^k$. Hence there exists an integer $M$
such that $\, \mathcal{L}_{d,k,M}  \,\,=\, \mathcal{U}_{d,k} $.

Similarly, among the inclusions $\mathcal{P}_{d,k,i} \subset \mathcal{P}_{d,k,i+1}$,
only finitely many can be strict, since each set is an irreducible variety in $\PP^{d^k-1}$.
Hence there exists a positive integer $M'$ such that $\mathcal{P}_{d,k,M'} = \mathcal{P}_{d,k,M'+j}$ 
for all $j > 0$. Suppose that this terminal signature variety $\mathcal{P}_{d,k,M'}$ is 
strictly contained in $\mathcal{U}_{d,k}$.
Then there exists a piecewise linear path $X$ such that $S = \sigma^{(k)}(X)$ is not
in $\mathcal{P}_{d,k,M'}$. In particular, there exists a polynomial
$\,f \in P_{d,k,M'}\,$ such that $f(S) = 1$. 

 By the Weierstrass Approximation Theorem, the piecewise linear path 
 $X$ can be approximated arbitrarily closely by a sequence of polynomial paths
$X_\epsilon$ with $\epsilon \to 0$.  Here we employ a reparametrization with the property
that $X$ slows down (to velocity zero) before each kink and then speeds up again. 
This allows us to use the $C^1$ version of Weierstrass Approximation,
which is what is needed here.
The signature tensors $S_\epsilon = \sigma^{(k)}(X_\epsilon)$ of the nearby paths depend
continuously on $\epsilon$, and they satisfy $f(S_\epsilon) = 0$ for all $\epsilon > 0$. This implies
$$ 0\, =\, {\rm lim}_{\epsilon \to 0} f( S_\epsilon) \,=\, f \bigl( \,{\rm lim}_{\epsilon \to 0} \,S_\epsilon \,\bigr) 
\,=\, f(S) \,=\, 1. $$
From this contradiction we now conclude that $\mathcal{P}_{d,k,M'}  = \mathcal{U}_{d,k}$.
\end{proof}

It was shown in Section \ref{sec:vsm}
that the signature matrices of piecewise linear and polynomial 
paths are the same. This result does not
extend to signature tensors:

\begin{theorem} \label{thm:different}
The two-segment paths and the quadratic paths in the plane $\RR^2$
have different signature threefolds. More precisely, for $k \geq 3$
we have $\,\mathcal{L}_{2,k,2} \not= \mathcal{P}_{2,k,2}\,$ 
in $\,\PP^{2^k-1}$.
\end{theorem}

\begin{proof}[Proof and Discussion]
For the sake of exposition, we show
 different proof techniques for $k=3$ and $k=4$.
 The case $k \geq 5$ is obtained by embedding
 small tensors into bigger ones.
 
We begin with $k=3$. A computation shows that
 $M=M'=3$ in Theorem  \ref{thm:chains}.
The ideal  $U_{2,3} = L_{2,3,3} = P_{2,3,3} $ 
is generated by six quadrics, displayed in Example
\ref{ex:twothree}. Both  $P_{2,3,2}$ and $L_{2,3,2}$
are generated by three quadrics modulo $U_{2,3}$.
For $P_{2,3,2}$  these three generators are
\begin{equation}                                                                                                     
\label{eq:10versus9}                                                                                                 
 \begin{matrix}                                                                                                      
              (2 \beta_1 + \gamma_1)^2 - {\bf 10} (\alpha_2 \gamma_1+3 \alpha_1 \gamma_2) ,\\                        
              (2 \beta_1 + \gamma_1) (2 \beta_2 + \gamma_2) + {\bf 10} (\alpha_3 \gamma_1+\alpha_2 \gamma_2), \\     
              (2 \beta_2 + \gamma_2)^2 - {\bf 10} (\alpha_3 \gamma_2 + 3 \alpha_4 \gamma_1).                         
\end{matrix}                                                                                                         
\end{equation}                                                                                                       
The generators of $L_{2,3,2}$ are obtained                                                             
by replacing the coefficient ${\bf 10}$ by ${\bf 9}$.                                                                

Now let $k=4$. The varieties $\mathcal{P}_{2,4,2}$ and
$\mathcal{L}_{2,4,2}$ are orbit closures for the ${\rm GL}(d,\KK)$-action on $\PP^{15}$.
We can use invariant theory to show that these orbits are different.
According to Diehl and Reizenstein \cite[Remark 14]{DR},
the space of ${\rm SL}(d,\KK)$-invariants  linear forms
 on $(\KK^2)^{\otimes 4}$ has dimension $2$ and
is spanned~by
\begin{equation}
\label{eq:linearinvariants}
 \begin{matrix}   & \ell_1 &\! =\! & \sigma_{1212}-\sigma_{1221}-\sigma_{2112}+\sigma_{2121} \\ \,\,
 {\rm and} \,\, & \ell_2 & \! =  \! &  \sigma_{1122}-\sigma_{1221}-\sigma_{2112}+\sigma_{2211}.
 \end{matrix}
\end{equation}
Their ratio $\ell_1/\ell_2$ is an absolute invariant, i.e.~a
 rational function on $\PP^{15}$ that is constant on orbits.
It takes value $0$ on $C_{\rm axis}$ and value $1/5$ on $C_{\rm mono}$.
Hence the orbit closures  $\mathcal{L}_{2,4,2}$ and $\mathcal{P}_{2,4,2}$ are different.
Indeed, {\tt Macaulay2} confirms that $\ell_1$ and $5 \ell_1 - \ell_2$ are
the unique linear forms in the ideals $L_{2,4,2}$ and $P_{2,4,2}$.
This explains the hyperplane $\PP^{14}$  in Example \ref{eq:zweivierzwei}.
 \end{proof}

\begin{remark} \rm It is instructive to explore the
geometric meaning of linear invariants such as (\ref{eq:linearinvariants}).
These specify hyperplanes that contain our  signature varieties.
For instance, the invariant (\ref{eq:alternating}) is
the volume of the convex hull of the path in $\RR^3$,
provided the path contains no four coplanar points \cite[Proposition 22]{DR},
or it is a limit of such paths.
This volume is zero for paths in $\RR^3$ that lie in a plane.
Therefore, the linear form (\ref{eq:alternating})  is contained in 
both of the ideals $P_{3,3,2}$ and $L_{3,3,2}$.

The analogous statement holds in all dimensions $d$. Consider
the alternating sum over all permutations, $\,\sum_{I \in S_d} {\rm sign}(I) \sigma_I$.
This invariant measures the volume of the convex hull of a path in $\RR^d$, assuming 
no $d$ points on the path lie in a hyperplane.  If $d$ is even then such a path can be closed. 
Example \ref{ex:itisbest} can thus be generalized as follows.
The projective variety $\mathcal{P}_{d,d,d-1}$ has
dimension $d^2-d-1$, and it lies on one hyperplane, given by the invariant above.
\end{remark}

\begin{remark} \rm
\label{rem:kontsevich} One might wonder why polynomial and piecewise linear paths
are so similar.
The ideals $P_{d,k,m}$ and $L_{d,k,m}$ have the same
numbers of minimal generators in all cases discussed so far.
However, this is not always the case. Our next example will show this:
the ideal $L_{d,k,m}$ can have more minimal generators than $P_{d,k,m}$.
Both the similarities and the differences of our two models can perhaps
be understood via degenerations of polynomial maps
$\PP^1 \rightarrow \PP^{d-1}$ to trees of lines in $\PP^{d-1}$.
Indeed,  an algebraic geometer might speculate that our
signature varieties are related to {\em Kontsevich's space of stable maps}
from $\PP^1$ to~$\PP^{d-1}$.
\end{remark}

\begin{table}[h] \qquad 
 \begin{center} 
\begin{tabular}{ | l | l | l | l | l | l | l | l | p{1.5cm} |} \hline  $d$ & $k$ & $m$ & $a$ &$\dim$ & $\deg$ & gens \\ 
\hline 2 & 3 & 2 &7 & 3 & 6 & 9 \\
\hline 2 & 3 & $\geq$ 3 & 7 & 4 &  4 & 6 \\
\hline 2 & 4  & 2 &14 & 3 & 24 & 55  \\
\hline 2 & 4  & 3 &15 & 5 & $192^{\mathcal{P}}, 64^{\mathcal{L}}$ & $(33^{\mathcal{P}},34^{\mathcal{L}})$, $(0^{\mathcal{P}},3^{\mathcal{L}})$, ? \\
\hline 2 & 4  & $\geq$ 4 &15 & 7 & 12 & 33 \\
\hline 2 & 5  & 2 &25 & 3 & 60 & 220 \\
\hline 2 & 5  & 3 &31 & 5 &  $1266^{\mathcal{P}}$, $492^{\mathcal{L}}$ & $(160^{\mathcal{P}},185^{\mathcal{L}})$, ? \\
\hline 2 & 6  & 2 &41 & 3 & 120 & 670 \\
\hline 2 & 6  & 3 &62 & 5 & $4352^{\mathcal{P}}$, $1920^{\mathcal{L}}$ & $(945^{\mathcal{P}},1056^{\mathcal{L}})$, ? \\
\hline 3 & 3  & 2 &25 & 5 & 90 & 162 \\
\hline 3 & 3  & 3 &26 & 8 & $756^{\mathcal{P}},396^{\mathcal{L}}$  & $(83^{\mathcal{P}},91^{\mathcal{L}})$ , ? \\
\hline 3 & 4  & 2 &65 & 5 & 600 & 1536  \\
\hline 3 & 4  & 3 &80 & 8 &  ?& $(1242^{\mathcal{P}},1374^{\mathcal{L}})$  , ? \\
\hline 
 \end{tabular} 
\end{center}
  \caption{\label{tab:PLdkm}  Invariants of the ideals $P_{d,k,m}$, 
  $L_{d,k,m}$ that define the varieties $\mathcal{P}_{d,k,m}$, $\mathcal{L}_{d,k,m}$} \medskip
  \end{table}
 
Table \ref{tab:PLdkm} summarizes the computational results 
we found for $\mathcal{P}_{d,k,m}$ and $\mathcal{L}_{d,k,m}$. 
Since  $\mathcal{P}_{d,2,m} = \mathcal{L}_{d,2,m} = \mathcal{M}_{d,m}$, we only consider $k\geq 3$. 
The columns have the same meanings as in Tables \ref{tab:Mdm} and \ref{tab:Udk}.
We use upper indices $\mathcal{P}$ and $\mathcal{L}$ to mark
distinctions between the polynomial case  and the piecewise linear case.
All our computations were done with 
{\tt Macaulay2} \cite{M2} and {\tt Bertini} \cite{bertini}.
To compute the degree we imposed linear constraints on the variety in question. We then
pulled the resulting equations back to the parameter space, 
where we counted their complex zeros using either Gr\"obner bases or numerical homotopy methods.

\begin{example}[$d{=}2, k{=}4$] \rm \label{ex:zweivier}
We have $M = M'=4$ in Theorem \ref{thm:chains}. 
The varieties $\mathcal{P}_{2,4,m}$ and $\mathcal{L}_{2,4,m}$
have dimensions $1,3,5,7$  for $m=1,2,3,4$.
The case $m=2$ was discussed above.
The case $m=3$ reveals the distinction.
The varieties $\mathcal{P}_{2,4,3}$ and $\mathcal{L}_{2,4,3}$ live in $\PP^{15}$. 
Both are $5$-dimensional, but their degrees differ. 
The former has degree $192$; the latter has degree $64$.
The universal ideal $U_{2,4} = P_{2,4,4}  = L_{2,4,4}$ has $33$ quadrics, namely $4,7,11,7,4$ 
in bidegrees  $(2,6),(3,5),(4,4),(5,3),(6,2)$. These
are also the quadrics in $P_{2,4,3}$. But, $L_{2,4,3}$
has $34$ quadrics and $3$ cubics. The extra quadric 
in $L_{2,4,3} \backslash P_{2,4,3}$
has bidegree $(4,4)$. 
\hfill $\diamondsuit $
\end{example}

When computing the degrees of our signature varieties by pulling back linear equations to the
parameter spaces, it is useful to know that the varieties
are rationally identifiable. In particular, their dimensions should be  $dm-1$, as expected.
This is the case for all instances in Table \ref{tab:PLdkm}.
We conjecture that it holds in general.
 We will study this topic in Section~\ref{sec:DimIde}.

\subsection{A Question of Lyons and Xu}
\label{subsec:lyonsxu}

This subsection concerns {\em axis-parallel paths}.
Each step $X_i$ in such a path is a multiple 
$\,a_i \cdot e_{\nu_i}\,$ of a standard basis vector $e_{\nu_i}$. Combinatorially, these
are characterized by the sequence $\nu = (\nu_1,\nu_2,\ldots,\nu_m) \in \{1,2,\ldots,d\}^m$
of steps that are chosen for the path. The  signature tensors of order $k$ of axis-parallel paths
form an irreducible subvariety $\mathcal{A}_{\nu,k}$ of $\mathcal{L}_{d,k,m}$.
This variety is parametrized by varying the step lengths
$a_1,a_2,\ldots,a_m$.

Unlike all the other varieties seen so far in this paper,
the {\em axis-parallel signature variety} $\mathcal{A}_{\nu,k}$
is not invariant under ${\rm GL}(d,\KK)$. It would be interesting
to study the prime ideal $A_{\nu,k}$ that defines $\mathcal{A}_{\nu,k}$,
as well as geometric invariants and the fibers of its parametrization.

We illustrate the study of fibers by providing an affirmative answer 
to the following question due to Lyons and Xu  \cite[Question 2.5]{LX2}:
{\em Can one find a nontrivial lattice path with length shorter than $2^{n+1}$ such that the first $n$ levels in its signature are all zero?} Here {\em length} is taken to mean discrete standard lattice length: $l = \abs{a_1}+\abs{a_2}+\ldots+\abs{a_m}$. 

\begin{proposition} \label{prop:LyXu}
There exists an axis-parallel path in the plane $\RR^2$ with $m=8$ steps in alternating axis directions
and length $\,l=14 < 16 = 2^{n+1}\,$ whose first $n=3$ signature tensors are all zero.
\end{proposition}

\begin{proof}[Derivation and proof]
Consider the parametrization of the varieties
$\mathcal{A}_{\nu,k}$ where $k=1,2,3$ and $\nu = (1,2,1,2,1,2,1,2)$.
These represent axis paths with steps
$(a_1, 0),\, (0, a_2),\,(a_3, 0),\, (0, a_4),\, (a_5, 0),\, (0, a_6),\, (a_7, 0),\, (0, a_0)$.
For $k=1$ we have
$$ \sigma^{(1)}(X) = (\sigma_1,\sigma_2) 
\qquad \hbox{where} \quad \qquad \begin{matrix}
\sigma_1 &=& a_1 + a_3 + a_5 + a_7 , \\
\sigma_2 &=& a_0 + a_2 + a_4 + a_6 .\\
\end{matrix}
$$
The $2 \times 2$ signature matrix  $\sigma^{(2)}(X)$ has entries
$$  \begin{small} \begin{matrix}
\sigma_{11} & = &  \frac{1}{2}(a_1+a_3+a_5+a_7)^2 \\
\sigma_{12} & = & a_0a_1+ a_0 a_3+a_0 a_5+a_0 a_7+a_1a_2+a_1a_4+a_1a_6+a_3 a_4+a_3 a_6+a_5 a_6 \\
\sigma_{21} & = &a_2 a_3+a_2 a_5+a_2 a_7+a_4 a_5+a_4 a_7+a_6 a_7 \\
\sigma_{22} & = & \frac{1}{2}(a_0+a_2+a_4+a_6)^2
\end{matrix} \end{small}
$$
The signature tensor $\sigma^{(3)}(X)$ has format $2 \times 2 \times 2$. Its entries are the cubic
polynomials
\begin{small}
$$ \begin{matrix}
\sigma_{111} &=& \frac{1}{6} (a_1+a_3+a_5+a_7)^3 \\
\sigma_{112} &=& \!\! \frac{1}{2} a_0 a_1^2{+} \frac{1}{2} a_0 a_3^2 {+} \frac{1}{2} a_0 a_5^2 {+} \frac{1}{2} a_0 a_7^2 {+} \frac{1}{2} a_1^2 a_2{+} \frac{1}{2} a_1^2 a_4 {+} \frac{1}{2} a_1^2 a_6{+} \frac{1}{2} a_3^2 a_4{+} \frac{1}{2} a_3^2 a_6{+} \frac{1}{2} a_5^2 a_6 +\\
 & &a_1 a_3 a_4 {+}
 a_1 a_3 a_6 {+} a_1 a_5 a_6  {+}a_3 a_5 a_6 {+} a_0 a_1 a_3{+}
 a_0 a_1 a_5{+}a_0 a_1 a_7{+}  a_0 a_3 a_7 {+} a_0 a_3 a_5 {+} a_0 a_5 a_7 \\
\sigma_{121} & = &\!\! a_1 a_2 a_3
\! + \! a_1 a_2 a_5 \! + \! a_1 a_2 a_7 \!+\! a_1 a_4 a_5 \!+\!  a_1 a_4 a_7 
\!+\! a_1 a_6 a_7 \!+\! a_3 a_4 a_5 \!+\! a_3 a_4 a_7  \!+\! a_3 a_6 a_7\!+\! a_5 a_6 a_7 \\
 \cdots &   & \cdots \qquad \cdots \qquad \cdots \qquad \cdots \qquad \cdots \\
\sigma_{222} &=& \frac{1}{6} (a_0+a_2+a_4+a_6)^3
\end{matrix}
$$
\end{small}
We seek a common zero of these $2+4+8$ expressions with no zero coordinate. The radical of the ideal they generate in $\RR[a_1,\ldots,a_7,a_0 ]$
is the intersection of $15$ prime ideals. The first $14$ primes
 all contain at least one of the parameters $a_i$ as an element. The only exception is
$$
 \langle\, a_3+a_7\,,\,\,a_2+a_6\,,\,\,a_1+a_5\,,\,\, a_0+a_4\,,
 a_4 a_5-a_5 a_6+a_4 a_7+a_6 a_7 \,\rangle.
 $$
 Each zero of this ideal gives a path $X$ with $\sigma^{(k)}(X) = 0$ for $k=1,2,3$.
We choose 
$$ a_0 = -3, \,\, a_1 = 1, \,\,a_2 = 1, \,\,a_3 = -2, \,\,a_4 = 3, \,\,
a_5 = -1, \,\,a_6 = -1, \,\,a_7 = 2. $$
This axis-parallel path with eight steps in $\RR^2$ (see Figure \ref{fig:paths}) proves the claim.
\end{proof}

\begin{remark} \rm
The case $m=8$ is the smallest counterexample. If $m \leq 7$ then each
zero of the corresponding ideal in $\RR[a_1,\ldots,a_m]$ represents a tree-like path.
For instance, if $m=5$ then  $\langle \,a_3 \,, \,a_1+a_5 \, ,\,a_2+a_4 \, \rangle$
is a prime ideal each of whose zeros gives a tree-like path.
\end{remark}

\begin{figure}[ht!]
\centering
\includegraphics[width=42mm]{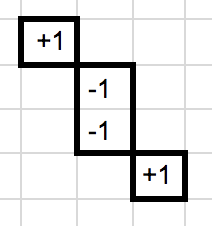}
\quad \quad \quad \quad
\includegraphics[width=50mm]{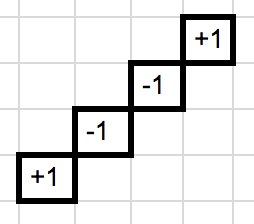}
\caption{\label{fig:paths} 
\textit{Left:} the path with 8 steps and length 14 from Proposition \ref{prop:LyXu}; \textit{Right:} path with 12 steps and length 16 from Lyons and Xu \cite{LX2}. Both paths have the first 3 signature levels equal to zero. The L\'evy areas are marked with the corresponding sign.}
\end{figure} 

\subsection{Rough Paths and the Rough Veronese} \label{sec:rough}

We now  connect our algebro-geometric study to
the theory of rough paths which underlies much recent progress in stochastic analysis.
This leads us  to a rough version of the  classical Veronese variety that 
represents linear paths $(m=1)$, as in Remark  \ref{rem:P=L} (a).

Every smooth path $X$ in $\RR^d$ lifts to a path in the Lie group
$\mathcal{G}^m (\RR^d) $, where  $t \in [0,1]$ is now mapped to the step-$m$ signature of $X|_{[0,t]}$.
In other words, we may replace the definite integrals in (\ref{eq:iteratedint})
and (\ref{eq:tensor}) with indefinite integrals that end at $t$ instead of~$1$.
We can also consider the same integrals starting at $s$ and ending at $t$.
The signature tensors of the resulting paths $X|_{[s,t]}$ are functions of $s$ and $t$ that satisfy the following
{\em H\"older condition}:
\begin{equation}
\label{eq:holderk}  \bigl|\, \sigma^{(k)} (X|_{[s,t]}) \,\bigr| \,\lesssim\, |\,t-s\,|^{k}  \quad
\hbox{for $\,0\le s < t \le1 \,$ and $\,k=1,\ldots,m$.}
\end{equation}
In Lyons' {\it rough path analysis} (see e.g. \cite{FH, FV, LyoICM} and references therein) this inequality 
now turns into a definition.   A {\em rough path of order $m$} is
any path in the Lie group $\mathcal{G}^m (\RR^d) $ that satisfies (\ref{eq:holderk}) but
with the exponent $k$ replaced with $k/m$. Given any rough path of order $m$, {\em Lyons' Extension Theorem}
provides a unique lift  to a rough path of order $n$, for all $n>m$,
such that the analytic estimate (\ref{eq:holderk}) remains valid for that extension. 
Every such rough path has a well-defined unique signature; that is, 
 $k$th signature tensors for any positive integer $k$.

Theorem \ref{thm:chains} furnishes two interpolations
$\,\mathcal{L}_{d,k,\bullet}\,$  and $\,\mathcal{P}_{d,k,\bullet}\,$ 
between the Veronese 
$\nu_k(\PP^{d-1})$ and  the universal variety $\mathcal{U}_{d,k}$ for fixed $d,k$.
In what follows we introduce a third such chain:
\begin{equation}
\label{eq:thirdsuchchain}
\nu_k(\PP^{d-1})\,=\,\mathcal{R}_{d,k,1}\,\subset\, \mathcal{R}_{d,k,2} \, \subset
\,\cdots\, \subset \, \mathcal{R}_{d,k,k}\,=\,\mathcal{U}_{d,k}.
\end{equation}
The rough Veronese variety $\mathcal{R}_{d,k,m}$ 
parametrizes  the $k$th signature tensors of the following set of order $m$ rough paths, 
indexed by Lie polynomials of degree $m$:
$$ \mathbb{X}^L \,:\, t \,\,\mapsto \,\, \mathbb{X}_t = \exp(tL)  \ , \quad L \in {\rm Lie}^m(\RR^d) \ . $$
This is a path in the Lie group $\mathcal{G}^m(\RR^d)$. It satisfies 
the H\"older condition $\, |\mathbb{X}^L_{s,t}| \lesssim \left\vert t-s\right\vert^{k/m}$ for 
$k \leq m$, with increments  $\mathbb{X}^L_{s,t}
\,=\,(\mathbb{X}^L_{s})^{-1}\otimes \mathbb{X}^L_{t}$ taken 
with respect to the group structure in 
$\mathcal{G}^m (\RR^{d}) $. 

\begin{remark} \rm We think of $\mathbb{X}^L$ as a $\mathcal{G}^{m}( \RR ^{d}) $-valued path 
that has constant speed $L$ when viewed through its log-chart. In that sense, 
$\mathbb{X}^L$ is a log-linear path. For $m=1$  we recover the linear paths in $\RR^d$, 
whose $k$th signature tensors form the Veronese variety
$\nu_k(\PP^{d-1}) = \mathcal{L}_{d,k,1}$.
\end{remark}

The signature of the geometric rough path $\mathbb{X}^L$ is simply the exponential
$\sigma(\mathbb{X}^L) = {\rm exp}(L)$.
The homogeneous component of degree $k$ in this tensor series
is the signature tensor $\sigma^{(k)}( \mathbb{X}^L)$.
Formally,  we define the {\em rough Veronese variety} $\mathcal{R}_{d,k,m}$ to be the subvariety of $\PP^{d^k-1}$ given by the
degree $k$ components of the exponential series  ${\rm exp}(L)$
where $L$ runs over ${\rm Lie}^m(\RR^d)$. 
 By construction, the inclusions (\ref{eq:thirdsuchchain}) hold, and we have
 $\mathcal{R}_{d,k,m} = \mathcal{U}_{d,k}$ whenever $m \geq k$.
 
\begin{example}[$d=2,m=2$] \rm
The Lie polynomials in ${\rm Lie}^2(\RR^2)$ are $\,\sigma = r e_1 + s e_2 + t [e_1,e_2]$,
where $r,s $ are parameters of degree 1 and $t $ is a parameter of degree 2.
 The rough Veronese $\mathcal{R}_{2,k,2}$ is a surface in
the space $\PP^{2^k-1}$ of $2 {\times} 2 {\times} \cdots {\times} 2$ tensors. It parametrizes the
 degree $k$ components of the exponential series ${\rm exp}(\sigma)$.
For a concrete computation set $k=3$ and revisit Examples \ref{ex:P23M2} and \ref{ex:twothree}.
We modify our earlier {\tt Macaulay2} code by setting $u=v=0$.
The output of the following fragment shows that $\mathcal{R}_{2,3,2}$ is a quartic surface
in a $\PP^5$ inside $\PP^7$:
\begin{verbatim}
ExpMap = map( QQ[r,s,t] , R ,
{r^2/2, r*s/2-t, s^2/2, r^3/6, r^2*s/6, 
  r^2*s/6-r*t/2, r*s^2/6, r*s^2/6-s*t/2, s^3/6,
  r, s, r*s/2+t, r^2*s/6+r*t/2, r*s^2/6+s*t/2 });
R232 = eliminate({s1,s2,s11,s12,s21,s22},kernel ExpMap)
codim R232, degree R232, betti mingens R232
\end{verbatim}
The ideal $R_{2,3,2}$ is generated by the two linear forms
$\sigma_{112} - 2 \sigma_{121} + \sigma_{211}$ and
$\sigma_{122} - 2 \sigma_{212} + \sigma_{221}$,
together with the six $2 \times 2$ minors of the $2 \times 4$ matrix
$ 
\begin{pmatrix}
\sigma_{111} & \! \sigma_{121} & \! \sigma_{212} & \sigma_{112} \\
\sigma_{121} & \! \sigma_{212} & \! \sigma_{222} & \sigma_{122 }
\end{pmatrix}
$.
\hfill $\diamondsuit$
\end{example}

\begin{example}[$m=2, k = 3$] \label{ex:XQ} \rm
Elements of ${\rm Lie}^2(\RR^d)$ can be written as $X + Q$,
where $X$ is in $\RR^d$ and $Q$ is a skew-symmetric $d \times d$ matrix.
The cubic component of ${\rm exp}(X+Q)$ equals
\begin{equation}
\label{eq:XQ} \begin{matrix}
\frac{1}{6}X^{\otimes 3}  \,+\,\frac{1}{2}\left( X\otimes Q + Q\otimes X\right) . \end{matrix}
\end{equation}
The rough Veronese $\mathcal{R}_{d,3,2}$ lives in $\PP^{d^3-1}$.
It consists of all $d {\times} d {\times} d$ tensors (\ref{eq:XQ}).
For $d=3$ we find that $\mathcal{R}_{3,3,2}$ 
has dimension $5$, degree $24$, and is cut out by
$81$ quadrics in a $\PP^{18}$ in $\PP^{26}$.
\hfill $\diamondsuit$
\end{example}

\begin{remark} \label{rem:LK1} \rm
If we replace the skew-symmetric $Q$ in (\ref{eq:XQ}) by a symmetric matrix,
then $\mathcal{R}_{d,3,2}$ turns into the
expected Brownian signature variety $\mathcal{B}_{d,3}$ in
Section~\ref{sec:expected}. This is no coincidence: 
$\mathbb{X}^L$ is a (deterministic) example of a Lie group valued (more precisely: $\mathcal{G}^m(\RR^d)$-valued) {\it L\'evy process}. These are stochastic processes with independent, stationary increments. Another example is Brownian motion $B=B(t,\omega)$ lifted to a rough path $\mathbb{B} = (1,B,\int B \otimes \circ \, {\rm d}B)$ via iterated Stratonovich stochastic integration. The L\'evy-Kintchine formula for such {\it L\'evy rough paths} \cite{FS} provides a unified formalism to study the expected signatures. This explains the resemblance of
the  rough Veronese and the expected Brownian signature variety.
\end{remark}

A detailed study of the rough Veronese
variety $\mathcal{R}_{d,k,m}$ was undertaken by Francesco Galuppi
in response to this article.
We might also consider a common generalization of
$\mathcal{R}_{d,k,m}$ and the piecewise linear signature variety
$\mathcal{L}_{d,k,m}$, namely the variety of signature tensors
 of piecewise log-linear paths in $\mathcal{G}^m(\RR^d)$.

\section{Dimension and Identifiability}
\label{sec:DimIde}

One of the motivations for this article is to develop a
framework for recovering paths $X$ from their signature tensors $\sigma^{(k)}(X)$.
The material in this section serves as a step in that direction, outlining 
an approach based on algebraic geometry to \cite[Problems 12 and 13]{HL}.

Solutions for this reconstruction problem when the path $X$ is piecewise linear 
were proposed by Lyons and Xu in \cite{LX1, LX2}.
This problem amounts to solving a system of polynomial equations in
the parameters $x_{ij}$,
but this was not made explicit in the articles \cite{LX1, LX2}. In earlier work,
Lyons and Sidorova \cite{LS} used signatures of paths for an application to
sound compression. 
They phrase the question of how to solve for $X$
given its signature as follows:

\smallskip

{\em  ``.... the problem of finding such a solution is equivalent to the problem of solving a system of polynomial equations. Provided that there exists a solution, this can be done for example by using 
Gr\"obner bases. In some sense, the remaining problem is to relate the level of truncation and the number of linear pieces in such a way that a solution exists. A disadvantage of this approach is a high complexity of the problem of solving polynomial systems''}  \cite[\S 3]{LS}.

\smallskip

State-of-the-art methods from applied algebraic geometry combine
Gr\"obner bases with other tools, e.g.~tensor methods and numerics.
 We are confident that these can be developed
 to the point where the polynomial systems for
 path reconstruction will be solved in practice.
 Particularly promising would be an approach based
 on the theory of {\em tensor decomposition} \cite{KB}.
 The proof of Proposition  \ref{prop:ab} suggests a
  Tucker decomposition with a fixed core tensor.
 
\subsection{Identifiability of the Universal Variety}

The universal variety  $\mathcal{U}_{d,k}$
comprises all $k$th order signature tensors of
deterministic paths in $\RR^d$. Its subvarieties
$\mathcal{L}_{d,k,m}$ and $\mathcal{P}_{d,k,m}$
represent piecewise linear paths with $m$ segments
and polynomial paths of degree $m$, respectively.
In what follows we examine whether the 
parameters of their natural parametrizations can be recovered from the tensors.

Theorem \ref{thm:lyndonmain}  implies that the dimension of
$\mathcal{U}_{d,k}$ is bounded above by
$\lambda_{d,k}-1$, where $\lambda_{d,k}$ is
the number of Lyndon words of length $\leq k$ over
the alphabet $\{1,2,\ldots,d\}$.
We compute these numbers using
Proposition \ref{prop:lyndon}, and we 
display some  values in Table \ref{tab:Udkdim}.
Note that each entry is considerably smaller
than $d^k-1$.

\begin{table}[h]
\begin{center} \begin{tabular}{ | l | l | l | l | l | l | l | l |   p{1.5cm} |} 
\hline  $d  \setminus k$ & 2  & 3 & 4 & 5  & 6 & 7 & 8 & 9  \\
\hline   2 & 2 & 4 & 7 & 13 & 22 & 40 & 70 & 126 \\
\hline  3 &  5 & 13 & 31 & 79 & 195 & 507 & 1317 & 3501 \\
\hline  4 &  9 & 29 & 89 & 293 & 963 & 3303 & 11463 & 40583	\\
\hline  5 &  14 & 54 & 204 & 828 & 3408 & 14568 & 63318 & 280318 \\
\hline 6  & 20 &  90 &  405 & 1959 & 9694 & 49684 & 259474 & 1379194 \\
\hline
 \end{tabular}
\vspace{-0.11in}
\end{center}
  \caption{\label{tab:Udkdim} The dimension
  of the universal variety $\mathcal{U}_{d,k}$ equals
  $\lambda_{d,k}-1$.
} \medskip
  \end{table}

We now prove that the upper bound for the dimension
of the universal variety is attained. We work
over an algebraically closed field $\KK$ of characteristic $0$.

\begin{theorem} \label{thm:EPidentifiability}
The projection from the Lie group $\,\mathcal{G}_{d,n} $ onto the affine
cone over the universal variety $\,\mathcal{U}_{d,n}$ is generically $\,n$-to-$1$. 
Hence, the projective variety $\mathcal{U}_{d,n}$ has the expected dimension $\lambda_{d,n}-1$.
\end{theorem}

\begin{proof}
Consider a general point $S$ in the irreducible variety
 $\mathcal{G}_{d,n}$. Its coordinates $\sigma_I$ are all non-zero.
The point $S$ is a sequence of tensors $S^{(1)}, S^{(2)},S^{(3)}, \ldots, S^{(n)}$
 whose formats are $d,d^2,d^3,\ldots,d^n$.
 The image of $S$  under the map from $\mathcal{G}_{d,n}$ onto
 the affine cone over $\mathcal{U}_{d,n}$ is the $n$th order tensor $T = S^{(n)}$.
We will prove that $T$ has precisely $n$ preimages 
with coordinates in $\KK$ under this map.

The symmetrization of $T$ is a symmetric tensor of rank $1$. It is the image
of the vector $S^{(1)} = (\sigma_1,\sigma_2,\ldots,\sigma_d)$ under the $n$th Veronese map.
This follows from the shuffle relation which states that 
$\sigma_{i_1} \sigma_{i_2} \cdots \sigma_{i_n}$  is equal to 
$\sigma_{i_1 \shuffle i_2 \shuffle \cdots \shuffle i_n}$
The latter is a linear form in the entries of 
the $n$th order tensor $T$, so it can be read off from $T$.
By taking $(i_1 , i_2 ,\ldots,i_n) = (1,1,\ldots,1)$, we see that there 
are $n$ choices for $\sigma_1$, namely the  distinct $n$th roots of  
$\sigma_{1 1 \cdots  1} = \frac{1}{n!} \sigma_1^n $.
The other $d-1$ entries of $S^{(1)}$ are found by setting
$\,\sigma_i =  \sigma_{ i \shuffle 1 \shuffle \cdots \shuffle 1}/\sigma_1^{n-1}\,$
for $i=2,3,\ldots,d$.

We can now recover the entries of $S^{(n-1)}$ uniquely.
This is done by setting  $\,\sigma_I = \sigma_{I \shuffle i} / \sigma_i \,$
for every word $I$ of length $n-1$.
 This works because each $\sigma_i$ is non-zero by our assumption on $S$.
At this point, we  proceed by downward
induction on the order of the tensor, recovering next
$\,S^{(n-2)}$ from  $S^{(n-1)}$,
 then $\, S^{(n-3)}$ from $S^{(n-2)}$,  etc. Finally, we recover $S^{(2)}$ from $S^{(3)}$.

This argument shows that $T$ has precisely $n$ preimages in $\mathcal{G}_{d,n}$.
For any $n$th root of unity $\eta$, we can multiply the degree $k$ part of the original 
point  $S$ by $\eta^k$ and obtain another preimage $S'$.
 This proves that our map has degree $n$.
In particular, it preserves the dimension.
The second assertion follows because we know from Theorem \ref{thm:lyndonmain} that 
$\dim( \mathcal{G}_{d,n}) = \lambda_{d,n}$.
Hence the projective variety $\mathcal{U}_{d,n}$
has dimension $\lambda_{d,n}-1$ in $\PP^{d^n-1}$.
\end{proof}

A key step in the proof was to extract the 
$n$th roots of a scalar in $\KK$.
If $\KK = \RR$ then the number of such roots
is $1$ if $n$ is odd and $2$ if $n$ is even. This implies:

\begin{corollary}
The map from the step-$n$ free nilpotent Lie group $\,\mathcal{G}^n(\RR^d)\,$
to the cone  over the real universal variety $\,\mathcal{U}^\RR_{d,n}$ is generically $\,1$-to-$1$
if $n$ is odd, and it is $2$-to-$1$ if $n$ is even.
\end{corollary}

Recall from Theorem \ref{thm:lyndonmain}  that the
signatures $\sigma_I$, where $I$ runs over all Lyndon words of length $n$,
can be used as coordinates for $\mathcal{G}^n(\RR^d)$.
In light of Theorem \ref{thm:EPidentifiability}, we might use these Lyndon 
coordinates to also express the equations defining  subvarieties of $\,\mathcal{U}_{d,n}$.

\begin{proposition} \label{prop:allnonzero}
Fix a sufficiently smooth path $X : [0,1] \rightarrow \RR^d$ that is not a loop, i.e.~$X(1) \not= X(0)$.
For any word $I$ and  odd integer $n > |I|$,
the coordinate $\sigma_I$ of the signature of $X$
   can be recovered uniquely
from the $n$th signature tensor $\sigma^{(n)}(X)$. The same statement holds true for any rough path (cf. Section \ref{sec:rough}), provided its first signature tensor is non-zero. 
\end{proposition}

\begin{proof}
The hypothesis ensures that $\sigma_i \not=0 $ for $i=1,2,\ldots,d$,
possibly after a linear change of coordinates in $\RR^d$.
The algorithm in the proof of  Theorem \ref{thm:EPidentifiability} shows
that there is a unique element in the Lie group $\mathcal{G}^n(\RR^d)$ which maps to 
the signature tensor $\sigma^{(n)}(X)$.
\end{proof}

The hypothesis in Proposition \ref{prop:allnonzero} cannot be removed.
It fails for loops.

\begin{example} \rm Fix a parameter $c \in \RR$.
Using Chen's formula (\ref{eq:PL1}),
 we can construct piecewise linear paths $X$
in the plane $\RR^2$ whose level-$3$ signature is 
$$ \exp \bigl(\, c [e_1, e_2] \,+\, [e_1, [e_1, e_2]]\, \bigr) 
\,=\, \exp (L)  \,\,\in \,\,\mathcal{G}^3(\RR^2) . $$
This is the signature of the order-$3$ rough path $\mathbb{X}^L$ with constant speed $L$.
The third signature is $\,\sigma^{(3)}(X) = [e_1, [e_1, e_2]]$.
For the Lyndon word $I = 12$, we have $\sigma_I = c$.
This cannot be recovered  from $\sigma^{(3)}(X)$. Note that 
$\sigma^{(i)}(X)  = 0$ for $i=1,2$.
\hfill $\diamondsuit $
\end{example} 

\begin{remark} \rm
The rough Veronese variety $\mathcal{R}_{d,k,m}$ in Subsection \ref{sec:rough}
is identifiable in the same $k$-to-$1$ sense as the universal variety.
Given a generic point $\sigma^{(k)}(\mathbb{X}^L)$ in $\mathcal{R}_{d,k,m}$,
we recover $\sigma(\mathbb{X}^L) = {\rm exp}(L)$ via Theorem~\ref{thm:EPidentifiability},
and from this we compute $L$ by taking the logarithm.
\end{remark}

\subsection{Recovering Paths from Signature Tensors}

Our main reconstruction problem is as follows. Given
a real $d {\times} d {\times} \cdots {\times} d$ tensor $S$,
we seek to find a piecewise linear or polynomial path $X$
of small complexity $m$ such that $\,S = \sigma^{(k)}(X)$.
Of course, a necessary condition is that $S$ lies in $\mathcal{L}^{\rm im}_{d,k,m}$
resp.~$\mathcal{P}^{\rm im}_{d,k,m}$. Assuming this to be the case,
we hope to express each parameter $x_{ij}$ as a function of the entries $\sigma_I$ of the tensor $S$.

In this subsection we discuss a variant of that reconstruction
problem, namely recovery up to scaling. We replace the signature images above
with the signature varieties  $\mathcal{L}_{d,k,m}$ and $\mathcal{P}_{d,k,m}$.
These are the closed images of our rational maps
$\, \PP^{md-1} \dashrightarrow \PP^{d^k-1}\,,\, \, X \,\mapsto \, \sigma^{(k)}(X)$.
We seek rational formulas that invert these maps on the signature varieties.
As before, the points in the projective spaces have their coordinates in
a field $\KK$ that is algebraically closed.

The corresponding polynomial map 
$\, \KK^{md} \rightarrow \KK^{d^k}\,$ does not admit a rational inverse
on the affine cones of our signature varieties. The reason is that
the signature tensor $S = \sigma^{(k)}(X)$ remains unchanged
 if we multiply the parameter matrix $X = (x_{ij})$ by a $k$th root of unity.

 \begin{remark}[$d=m=2, k=3$] \rm Consider the eight equations seen in Example \ref{ex:d2k3},
 where the $S = (\sigma_{ijk})$ obey the quadratic constraints in $P_{2,3,2}$.
 We seek to solve the eight equations for $x_{11},x_{12},x_{21},x_{22}$.
 To do this, we must extract a third root, e.g.~as in $x_{11} + x_{21} = (6 \sigma_{111})^{1/3}$.
These roots of unity disappear when we work in projective space.
In other words, if we regard both $X$ and $S$ up to scale, then we can compute
 formulas that express $X$ as a rational function in the points $S$
of  the signature variety $\mathcal{P}_{2,3,2}$.  Gr\"obner bases will find such formulas when they exist.
We demonstrate this for the simplest piecewise linear paths.
 \end{remark}
 
\begin{example}[$d=m=2, k=3$]  \label{ex:recovery1} \rm
The reconstruction of two-segment paths in the plane from their
$2 {\times} 2 {\times} 2$ signature tensors is solved by running the following {\tt Macaulay2} code:
\begin{small}
\begin{verbatim}
R = QQ[x11,x12,x21,x22,s111,s112,s121,s122,s211,s212,s221,s222,
     MonomialOrder => Lex];
M = matrix { 
 { s111, 1/6*(x11+x21)^3                                            },
 { s112, 1/6*x11^2*x12+1/2*x11^2*x22+1/2*x11*x21*x22+1/6*x21^2*x22  },
 { s121, 1/6*x11^2*x12+1/2*x11*x12*x21+1/2*x11*x21*x22+1/6*x21^2*x22},
 { s122, 1/6*x11*x12^2+1/2*x11*x12*x22+1/2*x11*x22^2+1/6*x21*x22^2  },
 { s211, 1/6*x11^2*x12+1/2*x11*x12*x21+1/2*x12*x21^2+1/6*x21^2*x22  },
 { s212, 1/6*x11*x12^2+1/2*x11*x12*x22+1/2*x12*x21*x22+1/6*x21*x22^2},
 { s221, 1/6*x11*x12^2+1/2*x12^2*x21+1/2*x12*x21*x22+1/6*x21*x22^2  },
 { s222, 1/6*(x12+x22)^3                                           }};
I = saturate( minors(2,M) , ideal(x11+x21,x12+x22) );
toString gens gb I
\end{verbatim}
\end{small}

The ideal {\tt I} represents the graph of the map 
 $\sigma^{(3)} : \PP^3 \dashrightarrow \PP^7$
 whose image is the threefold $\mathcal{L}_{2,3,2}$.
 The output of the {\tt Macaulay2} code is the reduced
 Gr\"obner basis for {\tt I}. This includes nine quadrics that
 generate $ L_{2,3,2}$. In the Gr\"obner basis we also see the following polynomials:
  \begin{equation}
  \label{eq:recovery1}  \begin{matrix}
 (\sigma_{122}- \sigma_{212}) \cdot  \underline{x_{21}} - (\sigma_{121} -\sigma_{211}) \cdot x_{22}, \quad \\
   (\sigma_{121}-\sigma_{211}) \cdot \underline{x_{12}} -  (\sigma_{212}-\sigma_{221}) \cdot x_{21}, \quad \\
\quad       3 \sigma_{211} \cdot \underline{x_{11}}  - 3 \sigma_{111} \cdot x_{12}
 + (\sigma_{211} - \sigma_{121}) \cdot x_{21}.
\end{matrix}
\end{equation}
This is a rational formula for $X \in \PP^3$ in terms of 
the point $\, \sigma^{(3)}(X)$ in $\,\mathcal{L}_{2,3,2} \subset \PP^7$.

The same method works for quadratic paths in the plane, namely
we can invert the
map $\sigma^{(3)} : \PP^3 \dashrightarrow \mathcal{P}_{2,3,2}\subset \PP^7$
by running {\tt Macaulay2}.
In the lexicographic Gr\"obner basis we find
\begin{equation}
\label{eq:recovery2}
\begin{matrix} 
3 ( \sigma_{112}- \sigma_{211}+\sigma_{212}) \cdot \underline{x_{21}}
\,+\, (\sigma_{112}+ \sigma_{121}-5 \sigma_{211} ) \cdot x_{22}, \\
(\sigma_{122}-2 \sigma_{212}\,+\,\sigma_{221}) \cdot \underline{x_{12}}
+ (\sigma_{112}-2 \sigma_{121}+ \sigma_{211}) \cdot x_{22} ,\\
5 \sigma_{121} \cdot \underline{x_{11}}
\,-\, 5 \sigma_{111} \cdot x_{21}
\,-\, (\sigma_{112}-5 \sigma_{121}- \sigma_{211}) \cdot x_{12}
\,-\,5 \sigma_{111} \cdot x_{22}.
\end{matrix}
\end{equation}
This is a  formula for $X \in \PP^3$ in terms of 
the point $\sigma^{(3)}(X)$  in $\,\mathcal{P}_{2,3,2} \subset \PP^7$.
\hfill $\diamondsuit $
\end{example}

Counting parameters gives an upper bound on the dimension
of our varieties:
\begin{equation}
\label{eq:twoineqs}
\! {\rm dim}( \mathcal{L}_{d,k,m}) \leq 
{\rm min} \{ \lambda_{d,k}-1, dm-1 \}
\,\,{\rm and} \,\,
 {\rm dim}( \mathcal{P}_{d,k,m}) \,\leq \,
{\rm min} \{ \lambda_{d,k}-1, dm-1 \}.
\end{equation}
If the dimension equals $dm-1$ then the variety is {\em algebraically identifiable}.
This means that, for some positive integer $r$, the map from $d \times m$ matrices that represent
paths $X$ to their  signature tensors $\sigma^{(k)}(X)$ is $r$-to-$1$ even up to scaling.
If $r=1$ then the map is {\em birational}, and the variety is {\em rationally identifiable}.
This  best-case scenario  happened in Example~\ref{ex:recovery1}.

In general, our signature varieties can be {\em non-identifiable},
i.e.~the inequalities (\ref{eq:twoineqs}) are strict.
This happens for the matrix case ($k=2$) studied in Section \ref{sec:vsm}.
We saw in Theorem \ref{thm:matricesmain} that 
$\mathcal{L}_{d,2,m} = \mathcal{P}_{d,2,m}$ has dimension
$\,dm - \binom{m}{2}-1$. For $2 \leq m < d $, this is strictly less than
the minimum of $dm-1$ and $\lambda_{2,d} -1  = (d+2)(d-1)/2$
 on the right hand side of (\ref{eq:twoineqs}).
In the non-identifiable case, the fibers of the map $X \mapsto \sigma^{(k)}(X)$
are positive-dimensional. It is interesting to study the geometry of such families of 
 paths with fixed signature tensor.

\begin{example}[$k=m=2$] \rm
We fix the $d \times d$ signature matrix $S$ given by
a two-segment path $X$ or a quadratic path $X$.
Suppose the path runs from ${\bf 0} = (0,0,\ldots,0)$
to ${\bf 1} = (1,1,\ldots,1)$.
For a two-segment path, we write
$X_1 = (a_1,a_2,\ldots,a_d)$ for the first step
and hence $X_2 = (1-a_1,1-a_2,\ldots,1-a_d)$
for the second step.  Then
$\sigma^{(1)}(X) ={\bf 1}$
and $S = \sigma^{(2)}(X) = \frac{1}{2} ({\bf 1} \cdot {\bf 1}^T + Q)$,
where $Q$ is the skew-symmetric matrix whose entries $q_{ij}$
record the L\'evy areas 
\begin{equation}
\label{eq:levyv}
 q_{ij} \,=\, a_i  - a_j \qquad \hbox{for} \,\,\, 1 \leq i,j \leq d. 
 \end{equation}
From this we see that $\mathcal{L}_{d,2,2}$ is not identifiable.
The fiber over $S$ is a curve, representing all paths
obtained from $X$ by adding a multiple of ${\bf 1} = (1,1,\ldots,1)$ to $X_1= (a_1,a_2,\ldots,a_d)$.
For a quadratic path, we write the $i$th coordinate as
$3a_i t + (1-3a_i)t^2$. Then we get the same signature
matrix as above. In particular, (\ref{eq:levyv}) holds
and the fiber is the same curve.
\hfill $\diamondsuit$
\end{example}

\begin{example}[$k=2,m=3$] \rm
Let $d \geq 5$ and consider three-segment paths $X$ that go from
${\bf 0}$ to $(a_1,a_2,\ldots,a_d)$, then
to $(b_1,b_2,\ldots,b_d)$, and finally to ${\bf 1}$. We
fix the signature matrix $S = \sigma^{(2)}(X) = \frac{1}{2} ({\bf 1} \cdot {\bf 1}^T + Q)$
as above. The L\'evy areas of the coordinate projections are
\begin{equation}
\label{eq:levyv2}
 q_{ij} \,= (a_i-1) b_j - (a_j-1) b_i  \,\,\, \hbox{for} \,\,\, 1 \leq i,j \leq d. 
\end{equation}
The $q_{ij}$ are now fixed.
 The fiber over $S$, consisting of
all three-segment paths with L\'evy areas $q_{ij}$, is a threefold.
Inside the $4$-dimensional affine subspace of $\RR^{2d}$ that is cut out by
$$
q_{ij} b_k - q_{ik} b_j + q_{jk} b_i \, = \,
q_{ij} (a_k-1) - q_{ik} (a_j-1) + q_{jk} (a_i-1) \, = \, 0 
\,\,\, \hbox{for} \,\, 1 \leq i < j < k  \leq d, $$
our threefold is the hypersurface defined by any
 \underbar{one} of the equations in (\ref{eq:levyv2}).
\hfill $\diamondsuit$
\end{example}

We now come to the general case $k \geq 3$.
We present two conjectures on identifiability.

\begin{conjecture} 
\label{conj:big1}
Let $k \geq 3$. Then equality holds in (\ref{eq:twoineqs}) for all
integers $d,m \geq 2$.
\end{conjecture}

If this is true then we know the thresholds for
stabilization in Theorem \ref{thm:chains}.

\begin{corollary}
\label{cor:MMM}
If Conjecture \ref{conj:big1} holds then
the constants in Theorem \ref{thm:chains} are
\begin{equation}
\label{eq:MMM}
 M \,\, = \,\,  M' \,\, = \,\, \biggl\lceil \frac{\lambda_{d,k}}{d} \biggr\rceil. 
 \end{equation}
 Furthermore,
 for all $\,m \geq M$,
we have $\,\mathcal{P}_{d,k,m} \,=\, \mathcal{L}_{d,k,m}$, and this variety equals
 the universal variety $\, \mathcal{U}_{d,k}$.
\end{corollary}

\begin{proof}
All three varieties $ \mathcal{P}_{d,k,m} , \mathcal{L}_{d,k,m}$
and $\mathcal{U}_{d,k}$  are irreducible of
  the same dimension. The first two are contained in the third.
 Hence all three coincide.
 \end{proof}

\begin{table}[h]
\begin{center} \begin{tabular}{ | l | l | l | l | l | l | l | l |   p{1.5cm} |} 
\hline  $d  \setminus k$   & 3 & 4 & 5  & 6 & 7 & 8 & 9  \\
\hline     2 & 3 & 4 & 7 & 12 & 21 & 36 & 64 \\
\hline     3 & 5 & 11 & 27 & 66 & 170 & 440 & 1168 \\
\hline     4 & 8 & 23 & 74 & 241 & 826 & 2866 & 10146 \\
\hline     5 & 11 & 41 & 166 & 682 & 2914 & 12664 & 56064 \\
\hline     6 & 16 & 68 & 327 & 1616 & 8281 & 43246 & 229866 \\
\hline
 \end{tabular}
\vspace{-0.11in}
\end{center}
  \caption{\label{tab:MMM}
The formula in (\ref{eq:MMM}) suggests
the value $M$  at which the signature varieties stabilize.
} \medskip
  \end{table}

The following stronger conjecture generalizes the conclusion of 
Example \ref{ex:recovery1}.

\begin{conjecture} 
\label{conj:big2}
If $k \geq 3$ and $m$ is strictly less than the constant $M$ in Corollary \ref{cor:MMM}
then both signature varieties $\,\mathcal{P}_{d,k,m} \,$ and $\, \mathcal{L}_{d,k,m}\,$
are rationally identifiable.
\end{conjecture}

The hypothesis $m < M$ is essential.
Conjecture \ref{conj:big2} fails when
 $m {=} M {=} \lambda_{d,k}/d$.

\begin{remark} \rm
The question of identifiability  is delicate
in the borderline case $\lambda_{d,k} = md$, when
the signature variety exactly fills the universal variety.
 Here we expect algebraic identifiability
(by Conjecture~\ref{conj:big1}), but
 rational identifiability generally fails.
This can be seen for $d=2, k=m=4$.
The $7$-dimensional variety $\mathcal{P}_{2,4,4} = \mathcal{L}_{2,4,4} = \mathcal{U}_{2,4}$
has degree $12$  in $\PP^{15}$. We have two parametrizations from the $\PP^7$ of
 $2 \times 4$ matrices. The parametrization using quartic paths
in $\RR^2$ is $48$-to-$1$. The parametrization using four-segment paths in $\RR^2$
 is only  $4$-to-$1$. 
\end{remark}

\begin{example}[$d=2,m=k=4$] \rm
Consider the four-step path in $\RR^2$ given  by
$$
X = \begin{bmatrix}
29 & 15 & 13 & 2 \\
23 & 26 & 6 & 27 \\
\end{bmatrix}
$$
There are three other paths that have the same $2 \times 2 \times 2 \times 2$ signature tensor:
$$ \begin{bmatrix}
36.74838 &-17.80169 & 37.75532 &    \, 2.29799 \\
 27.39596 & - 9.82926 & 40.23084 &  24.20246
\end{bmatrix},
$$
$$
\begin{bmatrix}
 102.16286 &  -131.13298 & 85.92484 & 2.04528 \\
 104.55786 &  -136.84738 &  86.56467  &   27.72484
 \end{bmatrix},
 $$
 $$
 \begin{bmatrix}
  38.53237 &  38.8057 &  -79.20533 &   60.86735 \\
  28.69523 &   82.7734 & -147.7839 & 118.3152
  \end{bmatrix}.
  $$
Each general point in $\mathcal{L}^{\rm im}_{2,4,4}$ has four preimages.
At least two of them are real.
\hfill $\diamondsuit$
  \end{example}

In the next result we reduce Conjectures~\ref{conj:big1} and \ref{conj:big2} to
$3$-way tensors and paths that span their ambient space. This uses
 concepts from the theory of tensor decomposition~\cite{KB}.

\begin{proposition} \label{prop:ab}
Fix integers $d,k,m$ that satisfy  $d \geq m \geq 1 $ and $k \geq 3$.
\begin{itemize}
\item[(a)] If  $\mathcal{L}_{m,3,m}$  is rationally (resp.~algebraically) identifiable
then so is $\mathcal{L}_{d,k,m}$. \vspace{-0.06in}
\item[(b)] If  $\mathcal{P}_{m,3,m}$  is rationally (resp.~algebraically) identifiable
then so is $\mathcal{P}_{d,k,m}$.
\end{itemize}
\end{proposition}

\begin{proof}
We claim that identifiability, for both 
piecewise linear and polynomial paths,
can be reduced to the case $k=3$ and $d=m$.
In either case, the reduction to $k=3$ follows from Theorem
\ref{thm:EPidentifiability}. If we are given the signature tensor
$\sigma^{(k)}(X)$ of a path $X$ for $k \geq 3$ then we can 
recover each tensor $\sigma^{(l)}(X)$ for $l \leq k$ up to a multiplicative constant.
In particular, for $l=3$, we obtain a unique point $\sigma^{(3)}(X)$ 
in the projective variety $\mathcal{L}_{d,3,m}$ or 
$\mathcal{P}_{d,3,m}$ respectively.

We next show that the case $d>m$ reduces to the case $d=m$.
Our path is encoded in a $d \times m$ matrix $X$. Its third
signature tensor has the format $d {\times} d {\times} d$:
\begin{equation}
\label{eq:tucker}
 \sigma^{(3)}(X) \,\, = \,\,\,
[[ C_\bullet ; X,X,X ]] \,\,\,= \,\,\, C_\bullet \times_1 X \times_2 X \times_3 X . 
\end{equation}
Here $C_\bullet$ is a tensor of format $m \times m \times m$ that is independent of $X$.
It represents either $C_{\rm axis}$ in Example \ref{ex:cap} or $C_{\rm mono}$ in Example \ref{ex:cmp}.
The notation used in (\ref{eq:tucker}) is standard in
 the literature on {\em tensor decomposition}, e.g.~\cite{LMV} or \cite[\S 4, eqn (9)]{KB}.
On the right in (\ref{eq:tucker}) we take the product of
a small $3$-way tensor with the same matrix on all three sides to create a larger $3$-way tensor.

We claim that the {\em core tensor} $C_\bullet$ has rank $\geq m$. 
Consider the  $m \times m$ matrix
$\tilde C_{\bullet}$ seen in the first slice $(i=1)$ of that tensor.
It suffices to show that this matrix is invertible.
For the axis paths, the matrix $\tilde C_{\rm axis}$ is invertible because it is upper-triangular
with non-zero diagonal entries.
For the monomial path,  the $m \times m$ matrix $\tilde C_{\rm mono} $ 
has the entries $\frac{  jk }{ (j+1)(j+k+1)}$. After scaling rows and columns,
this is a Cauchy matrix. Similarly to (\ref{eq:cauchymatrix}), we compute
$$
{\rm det}(\tilde C_{\rm mono}) \,\, = \,\, \frac{d \, !}{d+1} \cdot \frac{\prod_{1 \leq i < j \leq d}\, (j-i)^2}
{\prod_{i=1}^d \prod_{j=1}^d (i+j+1)} \,\,\,\, \not= \,\, 0. $$

We can assume that also the matrix
$X$ has rank $m$ since the definitions of rational and algebraic identifiability refer to a path that is sufficiently general. Under these hypotheses, the {\em Tucker decomposition} of a $3$-way tensor is 
unique in the following sense. All solutions  $(D,Y)$ to the 
system of $d^3$  polynomial equations given by
$\,[[ D; Y, Y,Y]] = \sigma^{3}(X)\,$  have the form
$D = [[C_\bullet; g,g,g]]$ and $Y = \eta g^{-1} X$ for some $g \in {\rm GL}(m,\RR)$
and some $\eta \in \CC$ with $\eta^3 = 1$.

Fix any such solution $(D,Y)$.  Consider the system of $m^3$ polynomial equations given by
$ [[C_\bullet; g,g,g]] = D$. By identifiability for $d=m$, this
 has a unique solution $g \in {\rm GL}(m,\RR)$. Here uniqueness is up to scale.
We recover the matrix encoding our path by setting $X = g Y$.
\end{proof}

We established the following partial result using computational methods.

\begin{lemma} \label{lem:computational}
For both signature varieties  $\mathcal{L}_{m,3,m}$ and $\mathcal{P}_{m,3,m}$ in $\,\PP^{m^3-1}$,
rational identifiability holds up to $m = 5$, and algebraic identifiability holds up to $m = 15$.
\end{lemma}

\begin{proof}[Computational Proof]
We prove algebraic identifiability by examining the Jacobian matrix of the parametrization
$X \mapsto \sigma^{(3)}(X)$.
This is a matrix of format $m^2 \times m^3$ whose entries are 
homogeneous quadrics in the parameters $x_{ij}$.
We show that this matrix has rank $m^2$. We do this by computing the rank 
 for \underline{one} random choice of 
integer parameters. This has rank $m^2$, and we conclude that
the resulting  variety in $\PP^{m^3-1}$ has the expected  dimension $m^2-1$.

We prove rational identifiability by checking it for \underbar{one} random
matrix  $X_0 \in \ZZ^{m \times m}$.  The identity
 $\,\sigma^{(3)}(X) = \sigma^{(3)}(X_0)\,$ translates into a system of
 $m^3$ inhomogeneous cubic equations in $m^2$ unknowns $x_{ij}$, namely the entries
 of $X$. We compute a Gr\"obner basis for this system. The output shows
that it  has precisely three solutions, namely $X = \eta X_0$ where $\eta^3 = 1$.
\end{proof}

\begin{corollary}
Conjecture~\ref{conj:big1} holds when $m \leq 15$.
Conjecture~\ref{conj:big2} holds when $m \leq 5$.
\end{corollary}

\begin{proof} If $d \geq m$ then this follows immediately from
Proposition \ref{prop:ab} and Lemma \ref{lem:computational}.
For $d < m$ this requires separate computations, like those in
the derivation of Lemma~\ref{lem:computational}.
\end{proof}

\begin{remark}  \rm 
Chen's formula (\ref{eq:PL1}) offers an alternative approach to
algebraic identifiability of $\mathcal{L}_{d,k,m}$.
The product rule of calculus is valid in the tensor algebra,
and can be used to write the Jacobian of the map that takes
$X = (X_1,\ldots,X_m)$ to (\ref{eq:PL1}) and then further to
 $\sigma^{(3)}(X)$. This leads to a tensor algebra
formula for the Jacobian that is inductive in $m$.
Using this, for small values of $m$, 
we also derived Proposition \ref{prop:ab} and
checked the rank of the Jacobian.
\end{remark}

\section{Expected Signatures} \label{sec:expected}

We now shift gears, in that we replace our deterministic paths
in $\RR^d$ with random paths~$X=X(\omega)$. The signature $\sigma^{(k)}(X)$ is a well-defined tensor-valued random variable, which takes values in the real universal variety $\mathcal{U}^\RR_{d,k}$ in affine space $(\RR^d)^{\otimes k}$, 
whenever sample paths are sufficiently regular, or a stochastic integration theory which respects the chain-rule is used. Thus, any such process, with randomly generated paths, induces a 
probability distribution on $\mathcal{U}^\RR_{d,k}$. In this section we study  projective varieties in
$\PP^{d^k-1}$ that arise from the expectation of such a probability distribution on paths.

\subsection{Brownian Motion}

Following Lyons \cite{CL, LyoICM}, we are interested in the 
{\em expected signature tensor} $\mathbb{E}(\sigma^{(k)}(X))$.
This is a convex combination in $(\RR^d)^{\otimes k}$ of points in $\,\mathcal{U}^\RR_{d,k}$,
weighted by the probability distribution. Since the
universal variety $\mathcal{U}^\RR_{d,k}$
is not a linear space, the tensor $\mathbb{E}(\sigma^{(k)}(X))$
 is usually not a point in $\mathcal{U}^\RR_{d,k}$.
The varieties of expected signature tensors that follow all
live in $\PP^{d^k-1}$. However, unlike the signature varieties in Section \ref{sec:pwlp},
they are now not subvarieties of~$\,\mathcal{U}_{d,k}$.

The prototype of a stochastic process 
in $\RR^d$ is the simple random walk $Z=Z(\omega,d)$
that starts at the origin $Z_0=0$ and takes equally likely 
axis-parallel steps $Z_{i+1} - Z_i$.
These steps are independent and identically distributed (i.i.d.) in 
$\{e_1,-e_1,\ldots,e_d,-e_d\}$, with covariance given by $d^{-1} I$ where $I = \sum_{i=1}^d e_i^{\otimes 2}$ is the $d \times d$ identity matrix.
A realization of this process, on $\{1,\ldots,m\}$ say,  is viewed as random axis-parallel path in $\RR^d$, moving linearly  at unit speed from $Z_i$ to $Z_{i+1}$.

The Central Limit Theorem suggests that we rescale the path as follows. Set
$$  Z^{(m)}_t \,\,:=\,\, m^{- \frac{1}{2}} Z_{mt} $$
where now $Z^{(m)}$ is a path on $[0,1]$. This is the construction of 
Subsection \ref{subsec:lyonsxu}, but with~random 
steps $X_i := m^{- \frac{1}{2}} (Z_{i+1} - Z_i), \ i =1, \ldots,m$. 
Under the given assumptions, one computes
$$
\mathbb{E}\bigl(\sigma (Z^{(m)})\bigr) \,=\, (\mathbb{E}( \exp (X_i) )^{\otimes m} 
\,=\, \left( {2d} \right)^{-m} 
\biggl(\sum_{i=1}^d {\rm exp}( m^{-1/2} e_i) + 
\sum_{i=1}^d {\rm exp}(- m^{-1/2} e_i) \biggr)^{\otimes m}.
$$
This expression converges, as $m\to\infty$, to the tensor exponential $\exp ( \frac{1}{2d}I ) $; equivalently the expected signature of $d^{1/2} Z^{(m)}$ converges
to $\exp ( \frac{1}{2} I )$. On the other hand, by the central limit theorem, have convergence in law $\,d^{1/2} Z^{(m)}_1 \to N(0,I)$.
 More generally, by  {\it Donsker's invariance principle} (a.k.a.~functional central limit theorem), the
 random path $d^{1/2} Z^{(m)}$ converges in law to a $d$-dimensional ``standard'' Brownian motion $B$ (with zero drift and unit covariance). Hence, the tensor series $\exp ( \frac{1}{2}I ) $ has the natural interpretation as {\it expected signature of Brownian motion}. This is known as {\it Fawcett's formula}, which is a starting point for 
  stochastic integration and rough paths. For details we refer to the textbook \cite{FH}.
 
 We now fix  $\mu \in \RR^d$ and  $\Sigma \in (\RR^d)^{\otimes 2}$  symmetric and positive definite.
 By adjusting the discrete walk approximations, or by a direct transformation of the form $X_t := \mu t + \sqrt{\Sigma} B_t$, one obtains {\em Brownian motion} $X$ on $[0,1]$
 with {\rm drift} $\mu $ and covariance $\Sigma$.
Note that $X_1 \sim N(\mu,\Sigma)$.

Both $\mu$ and $\Sigma$
are elements in  $T^n(\RR^d)$, provided $n \geq 2$.
Hence so is the sum $\mu + \frac{1}{2}\Sigma$ and its exponential.
If we define $\sigma(X)$ via iterated {\em Stratonovich stochastic integration} then,
by \cite[Ex.~3.22]{FH}, the following identity holds in the truncated tensor algebra $T^n(\RR^d)$:
\begin{equation}
\label{eq:stratono}
 \mathbb{E}(\sigma(X) )\,\, = \,\, {\rm exp}\bigl(\, \mu +  \frac{1}{2} \Sigma \,\bigr) .
  \end{equation}

The coefficients $ \sigma_{i_1 i_2 \cdots i_k}$ of this tensor series refine
classical Gaussian moments. We think of them as ``non-commutative moments''.
Indeed, $ \sigma_{i_1 i_2 \cdots i_k}$  plays the role of the moment associated with the
monomial $e_{i_1} \otimes e_{i_2} \otimes \cdots \otimes e_{i_k}$ in $T^n (\RR^d)$. We note that all moments of multivariate Gaussians are linear expressions in the
expected signatures of Brownian motion. 

\begin{proposition} \label{prop:moments}
Fix a $d$-dimensional Gaussian $Z\sim N\left( \mu ,\Sigma
\right) $ and corresponding Brownian motion $X$ with drift $\mu$ and covariance $\Sigma$.
For $u \in \mathbb{N}^d$ consider the shuffle product
$ w(u) = 1^{\shuffle u_1 }\shuffle \dots \shuffle d^{\shuffle u_d} $.
Then the  Gaussian moment $m_u =  \mathbb{E} \left( Z_{1}^{u_{1}}\cdots
Z_{d}^{u_{d}}\right)$ is equal to the shuffle linear form $ \sigma_{w(u)}$
evaluated at the coefficients of the expected signature
$ \mathbb{E}(\sigma(X) )$ in (\ref{eq:stratono}).
\end{proposition}

Here is an example for $d=3$. The moment
$\, m_{(1,1,1)} = \mathbb{E}(Z_1 Z_2 Z_3) \,$ equals
\begin{equation}
\label{eq:w111}
  \quad \sigma_{w(1,1,1)} \, = \, \sigma_{1 \,\shuffle \,2\,\shuffle \,3 } \, = \,
\sigma_{123} + \sigma_{132} + \sigma_{213} + \sigma_{231} + \sigma_{312} + \sigma_{321}.
\end{equation}
If we write $\Sigma = (s_{ij})$ then the six 
signatures in (\ref{eq:w111}) are $\sigma_{ijk} = \frac{1}{6} \mu_i \mu_j \mu_k 
+ \frac{1}{4} \mu_i s_{jk} + \frac{1}{4} s_{ij} \mu_k$.
Their sum is the familiar formula for Gaussian moments:
$\, m_{(1,1,1)} = \mu_1 \mu_2 \mu_3 \, + \,\mu_1 s_{23} \,
+ \mu_2 s_{13} + \mu_3 s_{12}$.

\begin{proof}
If we replace each non-commutative monomial
$e_{i_1} \otimes e_{i_2} \otimes \cdots \otimes e_{i_k}$  in the series
(\ref{eq:stratono}) by the corresponding commutative monomial
$t_{i_1} t_{i_2} \cdots t_{i_k}$ then we obtain the familiar moment 
generating function for Gaussians, displayed in e.g.~\cite[eqn.~(1)]{AFS}.
This symmetrization map from tensor series to commutative series
aggregates the coefficients as sums $m_u = \sigma_{w(u)}$.
\end{proof}

If the covariance matrix $\Sigma$ is set to $0$  then the path is linear,
and all signature tensors are symmetric of rank $1$.
In this case, the moment $m_u$ equals $\sigma_I$ times the multinomial
coefficient $\binom{|u|}{u}$ where $I = 1^{u_1} 2^{u_2} \cdots d^{u_d}$.
Geometrically: moment variety equals
signature variety equals Veronese. This remains true if $\Sigma$ is set to a 
fixed constant matrix (see Theorem \ref{thm:identhom}).

\smallskip

In what follows, we consider Brownian motion $X$ and its associated Gaussian $Z$
parametrized by $(\mu,\Sigma)$ as above.
We focus on the degree $k$ component in the tensor  series (\ref{eq:stratono}).
This $d {\times} d {\times} \cdots {\times} d$ tensor is denoted by $\mathbb{E}(\sigma^{(k)}(X))$ and called
the {\em expected Brownian signature tensor} of step $k$. Each of its entries $\sigma_{i_1 i_2 \ldots i_k}$
is a homogeneous polynomial of degree $k$
in the $d+\binom{d+1}{2}$ model parameters. Here, the
$d$ parameters coming from $\mu$ have degree one, and the
$\binom{d+1}{2}$ parameters coming from $\Sigma$ have degree two.
One can derive explicit formulas for $\sigma^{(k)}(X)$ in terms of $\mu$ and $\Sigma$,
similar to those in Example  \ref{ex:hereareforumulas}. See 
Example \ref{ex:XQ} and Remark \ref{rmk:mag} for the analogous 
situation where $\Sigma$  is skew-symmetric instead of symmetric.

We define the {\em expected Brownian signature variety} to be the
subvariety of $\PP^{d^k-1}$ that is parametrized by the
tensors $\sigma^{(k)}(X)$. This projective variety 
is denoted by $\mathcal{B}_{d,k}$.
Similarly, we can also define a projective variety
$\mathcal{B}_{d,\leq n}$ that lives in the projective space
associated with $T^n(\RR^d)$.
The premise of this section is the study of these varieties and their secant varieties.

\begin{example} \label{ex:Bd2} \rm
Let $k=2$. The expected Brownian signature matrix is given by 
$\frac{1}{2}(\mu \mu^T +  \Sigma)$, and as such it is symmetric.
 It equals half of the Gaussian second order moment matrix.  The variety $\mathcal{B}_{d,2}$ lives in
  $\PP^{d^2-1}$ but it is 
 actually a linear
subspace of dimension  $\binom{d+1}{2}-1$, defined by the equations that reflect the symmetry: 
\,$b_{ij}-b_{ji}=0 \,$ for $1\leq i <  j \leq d$. 
\hfill $\diamondsuit$
\end{example}

The symmetry observed in Example \ref{ex:Bd2} generalizes to expected signatures of order $ k \geq 3$.

\begin{remark} \label{prop:symm} \rm
The projective variety $\mathcal{B}_{d,k}$ lives in $\PP^{d^k-1}$
and its dimension is bounded above by $  d + \binom{d+1}{2}-1$.
It lives in a linear subspace of  dimension $\,\frac{d^k +  d^{\lceil \frac{k}{2} \rceil}}{2} -1 $,
namely the set of tensors that satisfy the reflectional symmetries
$\,\sigma_{i_1 i_2 \cdots i_{k-1} i_k} = \sigma_{i_k i_{k-1} \cdots i_2 i_1}$.
This holds because every term in  the expansion of
$\,\mathbb{E}(\sigma^{(k)}(X)) \,=\,\bigl(\mu + \frac{1}{2} \Sigma\bigr)^{\otimes k}\,$
has a matching term with respect to reflectional symmetry.
For instance, for $k=5$, the term  $\mu \otimes \Sigma \otimes \mu \otimes \mu $
matches $\mu \otimes \mu \otimes \Sigma \otimes \mu$.
\end{remark}

\begin{example} \rm
The expected Brownian signature tensor variety $\mathcal{B}_{2,3}$ is a cubic threefold in $\PP^7$.
Its ambient dimension is $5$, thanks to 
  the two relations $b_{122}=b_{221}$ and $b_{112}=b_{211}$ from Remark \ref{prop:symm}.
  In that $\PP^5$, our variety has codimension $2$ and is cut out by the three $2 \times 2$ minors of 
\begin{equation}                                                                                                     
\label{eq:B23}
 \begin{pmatrix}                                                                                                      
b_{111} & b_{121} & 2b_{122}-b_{212} \\                        
2b_{112}-b_{121} & b_{212} & b_{222} 
   \end{pmatrix}                                                                                                     .
\end{equation}
If we also consider the first and second order signatures, then we obtain the variety $\mathcal{B}_{2,\leq 3}$ sitting in $\PP^{14}$ but with ambient space $\PP^{11}$.
The linear relation $b_{12}=b_{21}$ gets added to the two above. This resulting variety is $5$-dimensional of degree $16$.
Its  ideal $B_{2,\leq 3}$ is  generated by $8$ quadrics and $19$ cubics. The five extra quadrics,
beyond those in (\ref{eq:B23}), are
\begin{equation}                                                                                                     
 \begin{matrix}                                                                                                      
          b_1b_{212} - b_2b_{121}, \quad b_1b_{222}-b_2(2b_{122}-b_{212}), \quad  b_1(2b_{112}-b_{121})-b_2b_{111},\\ b_1b_{22}-b_2b_{12}+2b(b_{212}- b_{122}), \quad    
       b_1b_{12}-b_2b_{11}+2b ( b_{112} - b_{121}).                         
\end{matrix}                                                                                                         
\end{equation}
Here $b$ is the homogenizing variable from the constant zero order signature. 
\hfill $\diamondsuit$
\end{example}

\begin{table}[h]
\begin{center} \begin{tabular}{ | l | l | l | l | l | l | l | l | p{1.5cm} |} \hline  $d$ & $k$ & $a$ &  $\dim$ & $\deg$ & gens \\ 
\hline $d$ &$  2$ & $\binom{d+1}{2}-1$ &$\binom{d+1}{2}-1$  & 1 & 0\\
\hline $d$ &  $\leq 2$ & $\binom{d+2}{2}-1$ &$\binom{d+2}{2}-1$  & 1 &0  \\
\hline 2 & 3 & 5 & 3 & 3 & 3\\
\hline 2 & $\leq 3$ & 11 & 5 & 16 &  8, 19\\
\hline 3 & 3 & 16 & 7 & 21 &45\\
\hline 3 & $\leq 3$ & 27  & 9 & 130 & 77, 175\\
\hline 2 & 4 & 9 & 4 & 12 & 10 \\
\hline 2 & $\leq 4$ & 22  & 5 & 102 &  63, 40\\
\hline 
 \end{tabular} 
\vspace{-0.11in}
\end{center}
  \caption{\label{tab:Bdk}   Invariants of the ideals $B_{d,k}$ and  $B_{d,\leq n}$, that define the varieties $\mathcal{B}_{d,k}$ and $\mathcal{B}_{d,\leq n}$} \medskip
  \end{table}

Table \ref{tab:Bdk} becomes even more interesting once we compare with the Gaussian moment varieties 
in \cite{AFS, ARS}.  We denote these now by $ \mathcal{N}_{d,\leq n}$. They are subvarieties of $\PP^{\binom{d+n}{d}-1}$ consisting of the vectors of all moments of order $\leq n$ of a Gaussian distribution on $\RR^d$. These moments are polynomial expressions parametrized by the entries of the Gaussian's mean $\mu \in \RR^d$ and covariance matrix $\Sigma$. Propositions 5, 7 and 9 in \cite{AFS} 
 describe the dimension, degree and minimal generators for  $\mathcal{N}_{2,\leq 3}$,  $\mathcal{N}_{2,\leq 4}$ and  $\mathcal{N}_{3,\leq 3}$ respectively. The first two items match the ones obtained for 
 $\mathcal{B}_{d,\leq n}$ with the same values of $d,n$.
 This highlights the similarity between expected signatures of Brownian motion and moments of Gaussians.

While the degrees of the generators of $\mathcal{N}_{d,\leq n}$ increase with $d,n$
 (e.g.~there are octic generators in $\mathcal{N}_{3,\leq 3}$), it appears that $\mathcal{B}_{d,\leq n}$ 
 only has generators up to degree~$3$. We record the following basic
 fact concerning the map between the two varieties given in~Proposition~\ref{prop:moments}.

\begin{proposition}
The linear projection from the expected signature variety $\mathcal{B}_{d,\leq n}$ 
onto the Gaussian moment variety $\mathcal{N}_{d,\leq n}$ is birational.
Both varieties have dimension $\binom{d+1}{2} + d-1$.
\end{proposition}

\begin{proof} The Gaussian moment variety is rationally identifiable. We can
recover $\mu$ and $\Sigma$ rationally from a general point in $\mathcal{N}_{d,\leq n}$.
The preimage of that point in $\mathcal{B}_{d,\leq n}$ equals (\ref{eq:stratono}).
\end{proof}

We believe that the degree is also preserved. 
The next conjecture implies this.

\begin{conjecture} \label{conj76}
The birational map from $\,\mathcal{B}_{d,\leq n}\,$  to $\,\mathcal{N}_{d,\leq n}$ 
is defined everywhere on $\mathcal{B}_{d,\leq n}$. Equivalently,
in the language of projective geometry, this birational map has no base points.
\end{conjecture}

\begin{remark}[Brownian motion in magnetic field and renormalization] \label{rmk:mag}  \rm
We presented a fairly pedestrian derivation of the expected signature formula (\ref{eq:stratono}) for Brownian motion $X_t = \mu t + \sqrt{\Sigma} B_t $. 
A more sophisticated construction starts with the Brownian rough path
$$  \mathbb{X}_t \,\,=\,\, \bigl(\,1, \,X_t, \int_0^t X \otimes \circ \, {\rm d}X \bigr) , $$  
defined by Stratonovich stochastic integration. In a second step, 
one uses Lyons' extension theorem to get all signature tensors
associated to $\mathbb{X}$. The third and last step is to take (component-wise) the expected value to arrive at
 (\ref{eq:stratono}).
 Let $Q$ be a skew-symmetric $d \times d$ matrix and consider the ``perturbed'' 
 (translated) Brownian rough path 
 $T_Q \mathbb{X} := \mathbb{\tilde X}$ given by 
$$
      \mathbb{\tilde X}_t \,\, :=\,\, \bigl(\,1, \,X_t, \int_0^t X \otimes \circ \, {\rm d}X + tQ  \,\bigr) \ .
$$

This construction is far from ad hoc: a systematic study of such perturbations  is closely related to Hairer's theory of renormalized SPDEs \cite{H14}. Brownian rough paths of this form have arisen  in a number of concrete situations with non-reversible noise, including the prominent example of Brownian motion in a magnetic field (cf.~\cite[Ch.3]{FH}) and 
subsequent works on homogenization (by Melbourne and coworkers). On the level of expected signatures, 
 \begin{equation}
\label{eq:stratonoRough2}
 \mathbb{E}(\sigma(\mathbb{\tilde X}) )\,\, = \,\, {\rm exp}\bigl(\, \mu + Q +  \frac{1}{2} \Sigma \,\bigr) 
  \end{equation}
follows from the L\'evy-Kintchine formula in \cite{FS}, by viewing $\mathbb{\tilde X}$ as L\'evy rough paths with L\'evy triplet $(\mu+Q, \Sigma, 0)$.
In absence of a diffusive component, i.e. when $\Sigma = 0$, we are precisely in the ($m=2$) rough Veronese setting of Section \ref{sec:rough}, cf. Example \ref{ex:XQ}.

The formula (\ref{eq:stratonoRough2}) generalizes 
(\ref{eq:stratono}). It replaces
the symmetric matrix $\Sigma$ with $\tilde \Sigma = \Sigma+2Q$, 
where $Q$ is skew-symmetric.
The degree $k$ part in the tensor polynomial ${\rm exp}\bigl(\mu + \frac{1}{2} \tilde \Sigma
\bigr)$ is denoted $\mathbb{E}(\sigma^{(k)}(\tilde X))$ and is called
the {\em expected magnetic Brownian signature tensor}
of step $k$. Each entry $\sigma_{i_1 i_2 \ldots i_k}$
of this tensor is a homogeneous polynomial of degree $k$
in the $d+d^2$ parameters. The $d$ parameters from $\mu$ have degree $1$, and the
$d^2$ parameters, from $\Sigma$ and $Q$, have degree $2$.
One can derive explicit formulas for $\mathbb{E}(\sigma^{(k)}(\mathbb{ X}))$ in terms of $\mu$ and $\Sigma$ and $Q$.

Finally, we define the {\em expected magnetic Brownian signature variety} to be the
subvariety of $\PP^{d^k-1}$ that is parametrized by the
tensors $\sigma^{(k)}(X)$. This projective variety 
is denoted $\mathcal{B}^{\rm{mag}}_{d,k}$.
Similarly, we can also define a projective variety
$\mathcal{B}^{\rm{mag}}_{d,\leq n}$ that lives in the projective space
associated with $T^n(\RR^d)$. These varieties, whose detailed study is deferred to future work, are ``stochastic'' generalizations of the rough Veronese, in the same way that Brownian and Gaussian varieties (as studied in \cite{AFS, ARS}) are natural generalizations of the classical Veronese.
\end{remark}

\subsection{Mixtures and Secant Varieties}

The signature varieties of Brownian motion 
live above the moment varieties of Gaussians.
There is a natural map from the latter onto the former,
induced by tensor symmetrization  $\PP^{d^k -1} \dashrightarrow \PP^{\binom{d+k-1}{k}}$.
We can pass to secant varieties on either side of this linear projection.

Secant varieties are the algebraic representation of the
statistical notion of {\em mixtures} of probability distributions.
The moment varieties for mixtures of Gaussians were
studied in detail in \cite{ARS}. Mixtures of Brownian motion lift these varieties
to the tensor space $\PP^{d^k-1}$. In other words, the varieties that follow are 
non-abelian extensions of those studied in \cite{ARS}.

A mixture of $r$ Gaussians is parametrized by $r$ pairs
$(\mu _1,\Sigma _1), \ldots, (\mu_r,\Sigma_r)$ as above, along with
 non-negative mixture probabilities $\alpha_1,\ldots,\alpha_r$ 
satisfying $\sum_{j=1}^r \alpha_j = 1$. Let $X^{(j)}$ denote 
Brownian motion with drift and covariance given by the
$j$th pair $(\mu_j, \Sigma_j)$. We can then consider the
mixed process $\bar{X}$ that is defined by selecting
Brownian motion $X^{(j)}$ with probability $\alpha_j$.
As an immediate consequence from (\ref{eq:stratono}), the mixed expected signature is given~by
\begin{equation*}
\mathbb{E} \left( \sigma (\bar{X}) \right) \,\,=\,\,\sum_{j=1}^r \alpha _{j} \cdot 
\exp \left( \mu ^{\left( j\right) }+\frac{1}{2}\Sigma ^{\left( j\right) }\right).
\end{equation*}
By focusing on the terms of degree $k$, or on all terms of degree $\leq n$, we obtain
the signature varieties for Brownian mixtures, denoted
$\mathcal{B}_{d,k,r}$ and $\mathcal{B}_{d,\leq n,r}$. These
live in projective spaces, and they map onto the moment varieties 
of Gaussian mixtures studied in \cite{ARS}. 

As with our other signature varieties, we are interested in their equations, as well as in 
formulas for recovering parameters, 
i.e.~identifiability as defined in the last section. We boldly conjecture that, unlike their commutative shadows in \cite{ARS}, the expected signature varieties $\mathcal{B}_{d,\leq n, r}$ are always algebraically identifiable. 
This was verified computationally in the first case where a secant of a Gaussian moment variety fails to be identifiable: $d=3, n = 3, r=2$.
On the other hand, while algebraic identifiability works for generic parameters $(\mu _1,\Sigma _1), \ldots, (\mu_r,\Sigma_r)$,
it may fail over specific submodels. We now illlustrate this.

\begin{example}[$d=r=2$] \rm
Consider mixtures of two Brownian paths in $\RR^2$, each with identity covariance: 
$\, \Sigma_1 = \Sigma_2 = I$.
This model has five parameters, namely the mixture parameter $\alpha \in \left( 0,1\right) $ and the entries of the two drift vectors: $\mu_1 = (a_1,a_2)$, $\mu_2 = (b_1,b_2)$.
The mixed expected signature of this model is the bivariate tensor series
\begin{equation*}
\mathbb{E}\bigl( \sigma(\bar{X}) \bigr) \quad = \quad
\alpha \cdot \exp \left( \mu_1 +\frac{1}{2}I\right) \,\,+\,\,\left( 1-\alpha \right)\cdot \exp \left( \mu_2+\frac{1}{2}I\right) .
\end{equation*}
The first and second signature tensors have five distinct entries
$\sigma_1,\sigma_2,\sigma_{11}, \sigma_{12} = \sigma_{21}$ and
$\sigma_{22}$. We write their expressions in terms of
the five model parameters in {\tt Macaulay2} format:
\begin{verbatim}
R=QQ[a1,a2,b1,b2,p,s11,s12,s22,s1,s2,MonomialOrder=>Lex];
I = ideal(  p*a1 + (1-p)*b1 - s1, p*a2 + (1-p)*b2 - s2,
           (p/2)*(a1^2 +1) + ((1-p)/2)*(b1^2 +1) - s11,
           (p/2)*(a1*a2+0) + ((1-p)/2)*(b1*b2+0) - s12,
           (p/2)*(a2^2 +1) + ((1-p)/2)*(b2^2+1) - s22);
toString gens gb I           
\end{verbatim}
The command {\tt gens gb I}
computes a Gr\"obner basis. Prepending the command {\tt toString}
displays that Gr\"obner basis. This output consists of $14$ polynomials, and it
reveals that
the model is \textit{not} identifiable. The associated variety is precisely the Veronese surface in $5$-space, 
after the linear change of coordinates coming from fixing identity covariances.
 The mixture model of two unit-covariance Brownian paths in $\mathbb{R}^2$ is cut out 
 by the $3 \times 3$  determinant
$$ \left( \begin{matrix}
1 & s_1 & s_2 \\
s_1 & 2s_{11} - 1 & 2s_{12} \\
s_2 & 2s_{12} & 2s_{22} -1 \\
\end{matrix} \right) . $$
This determinant is the first relation in $s_{ij}$ seen
 in the Gr\"obner basis {\tt gb I}.
\hfill $\diamondsuit $
\end{example}

The last example generalizes to mixtures with common covariance matrix $\Sigma$.

\begin{theorem}
\label{thm:identhom}  Consider mixtures of Brownian motion in $\RR^d$ with $r>1$ components
having the same covariance $\Sigma$. Fix $n$ such that $\,\binom{d+n}{n} \geq (d+1)r$. The
$r$ mean vectors are algebraically identifiable from the expected signature tensors up
to order $n$,  except in the following cases:
$$  \begin{matrix} n=2 \quad \, {\rm or} \quad\,
n=3, \, d=4, \, r=7 \,\\ {\rm or} \quad\,
n=4, \, d=2, \, r=5 \,\quad {\rm or} \quad\,
n=4, \, d=4, \, r=14. \end{matrix} $$
The same statement holds true for Brownian rough paths in 
the presence of a magnetic field, as in  Remark \ref{rmk:mag},
given by a fixed skew-symmetric matrix $Q$, and $\Sigma$ replaced by $\tilde\Sigma = \Sigma + 2Q$. 
\end{theorem}

\begin{proof}
This result is a slight generalization of \cite[Theorem 6.2.3]{Athesis},
where the analog for homoscedastic Gaussian mixtures appears.
The proof relies on the same fact: the subvariety $\mathcal{B}^{\tilde \Sigma}_{d,\leq n}$ of $\mathcal{B}_{d,\leq n}$ in question
is isomorphic to the Veronese variety after a linear change of coordinates. Hence, the $r$th secant variety $\mathcal{B}^{\tilde \Sigma}_{d,\leq n, r}$ that corresponds to the mixture of processes is isomorphic to the $r$th secant of the Veronese variety.
Indeed, consider the expansion of 
\begin{equation*}
\exp  \left( \mu+\frac{1}{2} \tilde{\Sigma} \right),
\end{equation*}
where $\tilde \Sigma$ is a fixed matrix. While we cannot distribute the exponential due to the tensor noncommutativity, for each particular expected signature $\sigma_{i_1\ldots i_k}$, we can factor
out the entries  $s_{ij}$ of
the matrix $\tilde \Sigma$ that appear in each summand.
This implies that $\sigma_{i_1\ldots i_k}$ equals $\mu_{i_1}\ldots \mu_{i_k}$ plus a linear combination of such $\mu$-monomials of lower order with products of $s_{ij}$ as coefficients.

It remains to see when the secant variety fails to have the expected dimension. The celebrated \textit{Alexander-Hirschowitz Theorem} \cite{AH} from algebraic geometry provides a classification of all the cases where 
secant varieties of a Veronese variety have that special property.
  The listed values are the exceptions that are relevant for the above parameter conditions.
\end{proof}

\bigskip \bigskip

\noindent
{\bf Acknowledgments.} We thank   Joscha Diehl, Francesco Galuppi
 and Anna Seigal  for valuable comments on drafts of this paper.
   Fernando De Ter\'an Vergara helped us with some references.
Carlos Am\'endola was supported by the Einstein Foundation Berlin and the Deutsche Forschungsgemeinschaft (DFG) in the context of the Emmy Noether junior research group KR 4512/1-1. Peter Friz was partially supported by the European Research Council through Consolidator Grant 683164 and DFG research unit FOR2402. 
Bernd Sturmfels was supported by the  US National Science Foundation (DMS-1419018).

\end{document}